\documentclass{amsart}
\usepackage{subfiles}
\usepackage{amssymb,euscript,amsmath, mathrsfs,tensor}
\usepackage[dvips]{graphicx}
\usepackage[dvips]{color}
\usepackage[all,cmtip]{xy}
\usepackage{xcolor}
\usepackage[margin=1in]{geometry}
\usepackage{tikz-cd}
\usepackage{hyperref}
\usepackage[capitalize]{cleveref}
\usepackage{enumerate}

\newcounter{ENUM}

\def\Ad{\mathrm{Ad}}
\def\Lie{\mathrm{Lie}}
\def\unip{\mathrm{unip}}
\def\cG{{}^c \mathbf G}
\def\Loc{\mathrm{Loc}}

\def\QQ{\mathbb{Q}}
\def\ZZ{\mathbb{Z}}
\def\CC{\mathbb{C}}
\def\RR{\mathbb{R}}
\def\GG{\mathbb{G}}

\def\CS{\mathcal S}
\def\CT{\mathcal{T}}

\def\CA{\mathcal{A}}
\def\CD{\mathcal{D}}
\def\CalC{\mathcal C}

\def\Spec{\mathrm{Spec}}
\def\Stab{\mathrm{Stab}}

\def\Hom{\mathrm{Hom}}
\def\End{\mathrm{End}}
\def\Res{\mathrm{Res}}
\def\Fr{\mathrm{Fr}}
\def\Rep{\mathcal{R}ep}
\def\Rmod{R\mathrm{\text{-}mod}}
\def\ind{\mathrm{ind}}
\def\der{\mathrm{der}}

\def\Irr{\mathrm{Irr}}
\def\LocL{\mathrm{LLC}}
\def\GL{\operatorname{GL}}
\def\SL{\operatorname{SL}}
\def\SO{\operatorname{SO}}
\def\Sp{\operatorname{Sp}}

\def\PGL{\operatorname{PGL}}

\def\bT{\mathbf{T}}
\def\hT{\widehat{\mathbf{T}}}

\def\LT{{\tensor*[^L]{\mathbf{T}}{}}}

\def\nr{\mathrm{nr}}

\def\C{\mathrm{C}}

\def\G{\mathrm{G}}

\def\X{\mathrm{X}}

\def\CH{\mathcal{H}}
\def\CO{\mathcal{O}}

\def\Ind{\mathrm{Ind}}
\def\GL{\mathrm{GL}}

\def\LG{{\tensor*[^L]{\mathbf{G}}{}}}
\def\LP{{\tensor*[^L]{\mathbf{P}}{}}}
\def\LQ{{\tensor*[^L]{\mathbf{Q}}{}}}
\def\LT{{\tensor*[^L]{\mathbf{T}}{}}}
\def\LL{{\tensor*[^L]{\mathbf{L}}{}}}
\def\LB{{\tensor*[^L]{\mathbf{B}}{}}}
\def\LM{{\tensor*[^L]{\mathbf{M}}{}}}
\def\LH{{\tensor*[^L]{\mathbf{H}}{}}}
\def\hG{\widehat{\mathbf{G}}}
\def\hH{\widehat{\mathbf{H}}}
\def\hP{\widehat{\mathbf{P}}}
\def\hQ{\widehat{\mathbf{Q}}}
\def\hT{\widehat{\mathbf{T}}}
\def\hL{\widehat{\mathbf{L}}}
\def\hB{\widehat{\mathbf{B}}}
\def\hM{\widehat{\mathbf{M}}}
\def\hU{\widehat{\mathbf{U}}}

\def\bG{\mathbf{G}}
\def\bP{\mathbf{P}}
\def\bQ{\mathbf{Q}}
\def\bL{\mathbf{L}}
\def\bB{\mathbf{B}}
\def\bM{\mathbf{M}}

\def\bU{\mathbf{U}}

\def\bA{\mathbf{A}}
\def\bW{\mathbf{W}}
\def\cW{\mathcal{W}}

\def\ad{\mathrm{ad}}

\def\un{\mathrm{un}}
\def\gen{\mathrm{gen}}

\def\Coh{\mathrm{Coh}}
\def\IndCoh{\mathrm{IndCoh}}

\def\St{\mathrm{St}}
\def\Xnr{\mathfrak{X}_{\mathrm{nr}}}

\newcommand{\margh}[1]{}

\newtheorem{thm}{Theorem}[section]
\newtheorem{prop}[thm]{Proposition}
\newtheorem{lemma}[thm]{Lemma}
\newtheorem{remark}[thm]{Remark}
\newtheorem{cor}[thm]{Corollary}
\newtheorem{conj}[thm]{Conjecture}

\theoremstyle{definition}
\newtheorem{definition}[thm]{Definition}

\numberwithin{equation}{section}
\def\red{\mathrm{red}}
\AtBeginDocument{
   \def\MR#1{}
}

\begin{document}

\title[Categorical Local Langlands Correspondence]{A generic categorical local Langlands correspondence\\
for quasisplit reductive groups} 

\author{David Helm}
\address{Department of Mathematics, Imperial College London, London, England, UK}
\email{d.helm@imperial.ac.uk}

\author{Maarten Solleveld}
\address{Institute for Mathematics, Astrophysics and Particle Physics, 
Radboud Universiteit Nijmegen, Nijmegen,
The Netherlands}
\email{m.solleveld@science.ru.nl}

\author{Yujie Xu}
\address{Department of Mathematics, 
Columbia University, 
New York, NY, USA}
\email{xu.yujie@columbia.edu}

\begin{abstract} 
We \textit{unconditionally} prove a \textit{generic} categorical local Langlands conjecture for a large class of quasi-split reductive $p$-adic groups $G$, including all quasi-split classical groups and some non-classical groups. More precisely, we construct a natural fully faithful functor from the stable $\infty$-category of \textit{generic} Bernstein blocks on the automorphic side to the stable $\infty$-category of ind-coherent sheaves on the moduli stack of (arithmetic) $L$-parameters, generalizing earlier work of the first author with Ben-Zvi, Chen and Nadler \cite{BZCHN} for $\GL_n$. Moreover, for an \textit{arbitrary} quasi-split reductive $p$-adic group $G$, we formulate a classical local Langlands framework under which a classical correspondence can be lifted to an $\infty$-categorical correspondence. 

Our result builds upon the phenomenal recent work of Zhu \cite{Zhu-tame} on the unipotent block, as well as structure results such as \cite{solleveld-endomorphisms} on the automorphic side and \cite{moduli} on the spectral side.

In particular, our work establishes \cite[Conjecture 8.2.1]{Hansen-Mann}, which implies that the conditional proof of \cite{Hansen-Mann} for the Fargues-Scholze categorical local Langlands equivalence \cite{Fargues-Scholze} (conditional on the conjectured compatibility of the Fargues-Scholze construction with spectral Eisenstein series) applies as well to a large class of quasi-split reductive $p$-adic groups $G$. 
\end{abstract}

\maketitle

\tableofcontents

\section{Introduction}
\subsection{Overview and main results}
Classically, the local Langlands correspondence is a conjectural finite-to-one map from the set of isomorphism classes of representations of a reductive group $G$, defined over a non-archimedean local field $F$, to the set of Langlands parameters for $G = \bG (F)$.  In the past several years, there has been considerable interest--and progress--in reformulating the local Langlands correspondence not as a map between these sets, but rather as a {\em fully faithful functor}, or even an equivalence, between certain categorifications of these two sets.  Such conjectures have been formulated in various ways by various authors (see for example \cite{Zhu-stack, BZCHN, Fargues-Scholze}). 

On the spectral side, the target of such a functor is a certain derived category $\IndCoh(X_{\LG})$ of {\em ind-coherent sheaves} on a moduli stack $X_{\LG}$ parameterizing Langlands parameters for the group $G$.  On the automorphic side, there are two (conjecturally equivalent) candidates for the domain of this functor (sheaves on the Kottwitz stack in Zhu's formulation~\cite{Zhu-stack} and sheaves on the moduli stack $\mathrm{Bun}_G$ of $G$-bundles on the Fargues-Fontaine curve in the Fargues-Scholze formulation~\cite{Fargues-Scholze}).  When $G$ is quasi-split, both of these contain, as natural full subcategories, the derived category $\Rep(G)$ of smooth representations of $G$. Although both of these constructions work most naturally with $\ell$-adic coefficients, in this paper we will instead consider smooth complex representations of $G$; this is harmless after fixing an isomorphism $\overline{\QQ}_{\ell} \cong \CC$.

If one restricts these conjectural functors to $\Rep(G)$, for $G$ quasi-split, one obtains the following conjecture, which we shall henceforth refer to as the ``categorical local Langlands correspondence for $G$'':

\begin{conj}[Categorical Local Langlands Conjecture] \label{conj:categorical-intro} \ \\
    For each standard Levi subgroup $M$ of $G$ (including $G$ itself), there is a fully faithful functor:
    \begin{equation}\label{eqn:LLC-functor-conj-intro}
    \LocL_M: \Rep(M) \hookrightarrow \IndCoh(X_{\LM})\end{equation}
    such that
    \begin{enumerate}
        \item the image of the Whittaker module $\mathcal W_M$ under $\LocL_M$ is isomorphic to the structure sheaf $\CO_{X_{\LM}}$; 
        \item for each standard parabolic $F$-subgroup $P = M U_P$ of $G$, the following diagram commutes:
        \begin{equation}\label{eqn:diagram-conjecture-intro}
        \begin{tikzcd}\Rep(M)\arrow[hook]{r}{\LocL_M}\arrow[]{d}[swap]{i_{P}^G} & \IndCoh(X_{\LM})\arrow[]{d}{(\pi_P)_* r_P^*}\\ 
        \Rep(G)\arrow[hook]{r}{\LocL_G} &\IndCoh(X_{\LG})\end{tikzcd} .
        \end{equation}
    \end{enumerate}
\end{conj}

The functor $(\pi_P)_* r_P^*$ induced by the correspondence $X_\LM \leftarrow X_\LP \to X_\LG$ 
is sometimes referred to as the ``spectral parabolic induction'' functor. We remind the reader that these functors are 
their derived versions, as we are regarding $\IndCoh$ as a derived category (that is, as a stable $\infty$-category). 

Many special cases of this conjecture have now been established. Work of the first author and his collaborators Ben-Zvi, Chen, Nadler \cite{BZCHN} proves this conjecture for $G = \GL_n$ and for Iwahori-spherical representations of an arbitrary split $G$, by identifying the relevant affine Hecke algebra $\mathcal{H}$ with the endomorphism algebra of a certain \textit{coherent Springer sheaf} on a certain stack of unipotent Langlands parameters. The recent preprint \cite{Zhu-tame} establishes this conjecture for the full subcategory of depth zero representations of $G$ (and much more besides).

The goal of this paper is to generalize the main results in~\cite[\S 5]{BZCHN}, and obtain a fully faithful functor beyond the tame direct factor of $\Rep(G)$. To precisely state our results, we first recall that the seminal work of Bernstein \cite{BD} allows us to decompose the automorphic side as a direct product:
\begin{equation}
    \Rep(G)\cong \coprod\limits_{[M,\sigma]}\Rep(G)_{[M,\sigma]},
\end{equation}
where each factor $\Rep(G)_{[M,\sigma]}$ is a full subcategory called a \textit{Bernstein block}, indexed by \textit{inertial classes} of the form $[M,\sigma]$ consisting of a Levi $M$ of $G$ and a supercuspidal representation $\sigma$ on $M$.~Here $[,]$ refers to the $G$-conjugacy class of the pair $(M,\mathfrak{X}_{\nr}(M)\cdot\sigma)$, with $\mathfrak{X}_{\nr}(M)$ the group of unramified characters of $M$.~On the spectral side, one considers the connected component $X_{\LG}^{\varphi_{\sigma}}$ containing the (cuspidal) L-parameter $\varphi_{\sigma}$ of~$\sigma$. Conjecture \ref{conj:categorical-intro} then specializes to a fully faithful embedding block by block:
\begin{equation}\label{eqn:fully-faithful-intro}
    \Rep_{[M,\sigma]}(G)\hookrightarrow \IndCoh(X_{\LG}^{\varphi_{\sigma}}).
\end{equation}
When the supercuspidal representation $\sigma$ is \textit{generic} (see \S\ref{subsec:repthy} for a precise definition) with respect to a fixed Whittaker datum, we say the Bernstein block 
$\Rep_{[M,\sigma]}(G)$ is a \textit{generic} block. 

In this article, we prove the following: 

\begin{thm}\label{thm:CLLCgeneric}
Let $\bG$ be one of the quasi-split groups 
\[
\GL_n, \mathrm{SL}_n, \PGL_n, \mathrm{U}_n, \Sp_{2n}, \SO_n, \SO_{2n}^*, 
\mathrm{GSpin}_n, \mathrm{GSpin}_{2n}^*, G_2.
\]
(Here * means a quasi-split group defined by an order two automorphism of the Dynkin diagram.)\\
Conjecture \ref{conj:categorical-intro} holds for generic Bernstein blocks in $\Rep (G)$; 
that is, for $[M,\sigma]$ generic, the fully faithful functor~\eqref{eqn:fully-faithful-intro} exists and has the properties predicted in Conjecture~\ref{conj:categorical-intro}.
\end{thm}

\begin{remark}
We expect that the same methods apply to the groups $\mathrm{GSp}_{2n}, \mathrm{PSp}_{2n}, \mathrm{GSO}_{2n}, \mathrm{PSO}_{2n}$, but this requires more work.
\end{remark}

Theorem \ref{thm:CLLCgeneric} is a crucial ingredient in the strategy of Hansen--Mann towards 
a categorical local Langlands correspondence for reductive $p$-adic groups.  In particular, 
our Theorem \ref{thm:CLLCgeneric} establishes \cite[Conjecture 8.2.1]{Hansen-Mann} for all
generic Bernstein blocks of the quasi-split reductive groups $G = \bG (F)$ listed above. 
The role this conjecture plays in the Hansen-Mann strategy is given by the following result:

\begin{thm}\label{thm:intro-FS-application}
\cite[Theorem 8.3.7]{Hansen-Mann} 
Let $G$ be a very well-understood quasi-split group (in the sense of \cite[\S 8]{Hansen-Mann}). Suppose that \cite[Conjecture 8.2.1]{Hansen-Mann} holds for all generic 
Bernstein blocks of $G$ and all its Levi subgroups, and that $m_{\phi}^{\mathrm{spec}} = m_{\phi}^{\mathrm{aut}}$ for all discrete parameters $\phi$.  Assume further that the Fargues-Scholze construction is compatible with spectral Eisenstein series for $G$ and its Levi subgroups \cite[Conjecture 5.4.1]{Hansen-Mann}. Then the full Fargues--Scholze categorical local Langlands correspondence holds for the group $G$.
\end{thm}
Here $m_{\phi}^{\mathrm{aut}}$ is the number of supercuspidal $G$-representations in the 
$L$-packet for $\phi$, whereas $m_{\phi}^{\mathrm{spec}}$ is defined as a count of certain cuspidal coherent sheaves and equals the number of cuspidal enhancements of $\phi$ which 
are relevant for $G$ (as shown in work-in-progress of Koshikawa and Bertoloni Meli).
The equality $m_\phi^{\mathrm{spec}} = m_\phi^{\mathrm{aut}}$ is 
already known in literature for all the groups in Theorem \ref{thm:CLLCgeneric}, 
see Appendix \ref{appendix} and \cite{AMS4}.

\begin{cor}[Theorem \ref{thm:CLLCgeneric} + Theorem \ref{thm:intro-FS-application}]
Assume the compatibility of \cite{Fargues-Scholze} with spectral Eisenstein series, and that $m_{\phi}^{\mathrm{spec}} = m_{\phi}^{\mathrm{aut}}$ for all discrete parameters $\phi$ of very well-understood quasi-split group $G$ and all its Levi subgroups. Then the
full Fargues-Scholze \textit{equivalence} holds for $G$.
\end{cor}

For an arbitrary quasi-split reductive group, we formulate a classical local Langlands framework--inspired from certain standard desiderata--under which a classical correspondence can be lifted to an $\infty$-categorical correspondence:

\begin{thm}\label{thm:intro-wgsc-thm}
Let $G$ be a quasi-split connected reductive group over $F$ that admits a weak generic supercuspidal correspondence (See Definition~\ref{def:wgsc}).~Then Conjecture \ref{conj:categorical-intro} holds for any generic Bernstein block of $\Rep(G)$.
\end{thm}
We expect \textit{all} quasi-split connected reductive groups over $F$ to satisfy the requirement in Theorem
\ref{thm:intro-wgsc-thm}. As such, we expect that if one can enlarge the list of ``well-understood'' groups 
in \cite[Definition 6.1.1]{Hansen-Mann}, then the strategy of \cite{Hansen-Mann} towards a Fargues--Scholze categorical 
local Langlands equivalence will apply to those groups as well.

\subsection{Strategy and further technical results}
Our approach to this question is inspired by the approach of~\cite[\S 5]{BZCHN} in the case of $G=\GL_n$, which we now describe.  On the one hand, when $G = \GL_n$, the connected components of $\GL_n$ are indexed by ``inertial types'', i.e.~representations of $I_F$ that extend to Langlands parameters.  Moreover, the connected component $X_{\LG}^{\nu}$ of $X_{\LG}$ corresponding to an inertial type $\nu$ can be identified with the unipotent component (that is, the component $X^1_{\LG_{\nu}}$ corresponding to the trivial representation of $I_F$) for an associated group $G_{\nu}$, which for $G = \GL_n$ is simply a product of restrictions of scalars of general linear groups.  This gives an equivalence of categories between $\IndCoh(X_{\LG}^{\nu})$ and $\IndCoh(X^1_{\LG_{\nu}})$. 
On the other hand, the inertial type $\nu$ corresponds to a Bernstein block $\Rep(\GL_n(F))_{[M,\sigma]}$ for $\GL_n(F)$, and an explicit calculation provides an equivalence of categories between $\Rep(\GL_n(F))_{[M,\sigma]}$ and the 
Iwahori-spherical block $\Rep(G_{\nu})_{[T,1]}$.  Combining these equivalences with the fully faithful functor from $\Rep(G_{\nu})_{[T,1]}$ to $\IndCoh(X^1_{\LG_{\nu}})$ constructed in \cite[\S 1-\S4]{BZCHN} yields the desired functor \eqref{eqn:fully-faithful-intro} in the case of $\GL_n$.~The resulting functor carries a natural progenerator $\Pi_{[M,\sigma]}$ in $\Rep(G)_{[M,\sigma]}$ to a natural sheaf $\CS_G^{\phi}$ called the \textit{coherent Springer sheaf} on $X^{\nu}_{\LG}$. 
The endomorphism algebras of both $\Pi_{[M,\sigma]}$ and $S^{\phi}_{G_{\nu}}$ are isomorphic to the affine Hecke algebra associated to $G_{\nu}$.

When $G$ is not a general linear group, several complications arise.  On the spectral side, it is no longer true that an arbitrary component of $X_{\LG}$ may be identified with a space of unipotent Langlands parameters.  Similarly, on the automorphic side, it is not always possible to identify a block $\Rep(G)_{[M,\sigma]}$ with the Iwahori-spherical block for a related group. Instead, the relationship between the general situation and the unipotent situation is more subtle, as the corresponding Hecke algebras in question are no longer ``purely'' affine and they can be \textit{extended affine Hecke algebras with unequal q-parameters}. As such, the technical core of the paper is devoted to intricate analyses of the root systems and various Weyl groups associated to these Hecke algebras (realized as the endomorphism algebras of certain progenerators), as well as the endomorphisms of the relevant coherent Springer sheaves. 

More precisely, given a pair $[M,\sigma]$ and an associated parameter $\phi$, one can still construct the 
progenerator $\Pi_{[M,\sigma]} := i_P^G \cW^{[M,\sigma]}_M$ and the corresponding coherent Springer sheaf 
$\CS_G^{\phi} := (\pi_P)_* r_P^* \CO^{\phi}_{X_\LM}$ on the component $X_{\LG}^{\phi}$, as well as the relevant 
``centralizer'' group $G_{\nu}$.~Conjecture \ref{conj:categorical-intro} predicts (and indeed we prove this expectation in the proof of Theorem~\ref{thm:single block}) that the following holds:
\begin{equation}\End_G \big(i_P^G \cW^{[M,\sigma]}_M \big)\cong \End (\CS_G^\phi).
\end{equation}
However, it will no longer be true that the endomorphism algebra of $\CS_G^{\phi}$ is isomorphic to the endomorphism algebra of $\CS_{G_{\nu}}^1$: the latter is an affine Hecke algebra with possibly unequal parameters, whereas the former endomorphism algebra is often an extended affine Hecke algebra, which consists of an affine part extended by some group algebra. The issue that arises is that (unlike in the $\GL_n$ case) the centralizer of the inertial part $\nu$ of a Langlands parameter need not be connected. 
By an intricate analysis of the component groups of these centralizers, and their action on the relevant moduli spaces of L-parameters, we are able to establish a precise relationship between these two rings in Theorem~\ref{thm:Springer reduction}; in particular, the endomorphism algebra of $\CS_G^{\phi}$ is a semidirect product of the endomorphism algebra of $\CS_{G_{\nu}}^1$ with the group algebra of a certain ``spectral R-group'' that we denote by $\widehat{R}_{G,\phi}$ (see \eqref{eqn:pi-0-G-SES}). 
\begin{thm}[Spectral Side] There is a canonical isomorphism 
\[
\End(\CS^{\phi}_G) \cong \End(\CS^1_{G_{\nu}})^{\pi_{0,M,\phi}} \rtimes \CC[{\widehat R}_{G,\phi}]
\]
satisfying compatibilities with various group actions. 
\end{thm}

On the automorphic side of the picture, there is a long history of results
that describe the endomorphism algebra of a progenerator $\Pi_{[M,\sigma]}$ as a (twisted) semidirect product of an affine Hecke algebra with a group algebra.  The results relevant to us in this setting include recent results of Opdam and the second author \cite{solleveld-endomorphisms,OSgeneric}, extending earlier work of Heiermann~\cite{Heiermann}.  We apply these results in \S\ref{sec:Proof-main1} to obtain, under certain technical hypotheses, an isomorphism of the endomorphism algebra of $\Pi_{[M,\sigma]}$ with a semidirect product of the endomorphism algebra of $\Pi_{[M_{\nu},1]}$ and a group algebra; indeed, the ideas along this line of argument go as far back as \cite{Morris-Inventiones,Roche-Hecke-algebra,Reeder-2002}; see Theorem~\ref{thm:automorphic reduction} and Theorem~\ref{thm:rep thy isom} for the precise statements we prove on the automorphic side.
\begin{thm}[Automorphic Side]
There is a canonical algebra isomorphism:
\[
\End \big( i_{P}^G \mathcal{W}_M^{[M,\sigma]} \big) \cong 
\End \big( i_{P_{\nu}}^{G_{\nu}} \chi^{\un} \big)^{K_{M,\phi}} \rtimes \CC [R(\CO_{\sigma})]
\]
satisfying the relevant compatibilities. 
\end{thm}

Combining these isomorphisms of Theorem~\ref{thm:Springer reduction} and Theorem~\ref{thm:automorphic reduction} with the known unipotent categorical Langlands correspondence (see \eqref{eqn:LLC1M}) we obtain an isomorphism between the endomorphism algebras of $\Pi_{[M,\sigma]}$ and $\CS^{\phi}_G$,  which we denote by $\mathcal{E}$, then the functor 
\begin{equation}V \mapsto \Hom_G(\Pi_{[M,\sigma]}, V) \otimes_{\mathcal{E}} \CS^{\phi}_G
\end{equation}
yields the desired (derived) fully faithful $\infty$-categorical functor from $\Rep(G)_{[M,\sigma]}$ to $\IndCoh(X^{\phi}_{\LG})$.

\subsection{Relation to other work}\label{par:1.3}

\subsubsection{Relation to Zhu's correspondence}
A key input to our construction is the phenomenal recent manuscript \cite{Zhu-tame}; in particular, our approach involves a (highly technical) reduction step to the known unipotent categorical Langlands correspondence for unramified groups, which is proven in \cite[Theorem 1.9]{Zhu-tame}.~Zhu's proof of this correspondence proceeds by applying a categorical trace construction to a form of Bezrukavnikov's seminal equivalence between two categorifications of the affine Hecke algebra~\cite{Bezrukavnikov}.  We note that even if one is only interested in split groups, the result of~\cite{BZCHN} is not sufficient, as even when $G$ is split the associated group $G_{\nu}$ (for which we need the unipotent categorical local Langlands correspondence) need not be.

The overlap between our results and Zhu's consists of the depth-zero (``tame'') generic blocks of $\Rep(G)$; that is, those $[M,\sigma]$ for which $\sigma$ is of depth zero and generic.  In this case we expect that our categorical correspondence coincides with the one he constructs.  In contrast with our results, Zhu's results also include non-generic blocks and non-quasi-split groups, but remain limited to depth zero representations of those groups.  On the other hand, our results are, at the time of writing, the only results in the literature that construct a categorical local Langlands correspondence beyond the tame setting outside the case of $G = \GL_n$.

\subsubsection{Relation to the classical Local Langlands Correpondence}

As with the argument in \cite[\S 5]{BZCHN},
our result crucially uses a ``weak generic supercuspidal 
correspondence'': for each standard Levi $M$ of $G$ we need a bijection from irreducible generic supercuspidal representations of $M$ to supercuspidal parameters $W_F \rightarrow \LM$, satisfying some fairly standard compatibilities; c.f.~Definition~\ref{def:wgsc}. 
The functor we construct will then, by its very construction, be compatible with this classical correspondence in the sense that for any irreducible smooth representation $\pi$ of $G$, the corresponding object of $\IndCoh(X_{\LG})$ will be supported on the locus in $X_{\LG}$ of parameters that agree, up to semisimplification, with the Langlands parameters associated to $\pi$.

\subsection{Organization of the paper}
Section \ref{sec:prelim} contains background information on the various tools used in the paper, particularly the theory of stable $\infty$-categories, the Bernstein decomposition, and the notion of genericity with respect to a Whittaker datum. In \S\ref{sec:categorical-LLC}, we recall the categorical local Langlands conjecture in the form most relevant for our arguments. We give the precise statement of the known unipotent categorical local Langlands correspondence that we use as an ingredient, and we prove that it is compatible with parabolic induction.

The new ideas in this paper begin in \S\ref{sec:spectral-side-reduction}, where we obtain a geometric relationship between an arbitrary component of the moduli stack of Langlands parameters (for an arbitrary quasi-split group) and a unipotent component of the stack of parameters for a related group.  This leads us to a relationship between the endomorphism algebras of the Springer sheaves that live on the two components.

Sections \ref{sec:conjecture} and \ref{sec:Proof-main1} are dedicated to formulating, and proving, the representation-theoretic analogue of the main results of \S\ref{sec:spectral-side-reduction}; that is, the analogous relationship between the endomorphism algebras of two progenerators.  We note that to precisely formulate this statement one needs a weak generic supercuspidal correspondence; since this is not available in all cases we instead work in terms of a pair $(\sigma,\phi)$, where $\sigma$ is a supercuspidal generic representation and $\phi$ is a Langlands parameter, that are compatible in a weak sense that we make precise in \S\ref{sec:conjecture}. In cases where a weak generic supercuspidal correspondence is known and $\phi$ is the parameter attached of $\sigma$, these conditions are satisfied.

Section \ref{sec:main} combines the representation-theoretic and spectral results to prove our main results. In the Appendix, we prove the existence of a weak generic supercuspidal correspondence for the groups listed in Theorem \ref{thm:CLLCgeneric}.

\subsection{Notation and conventions} \label{par:notation}
Throughout, $F$ will be a non-archimedean local field, and $W_F$ will denote the Weil group of $F$. 
Let $\Fr \in W_F$ be a geometric Frobenius element. For our Langlands correspondence we use the version of
Artin reciprocity normalized so that $\Fr$ is sent to a uniformizing element of $F$.

In the introduction we did not distinguish between a connected reductive group and its group of $F$-points, but it will be necessary to do so in the sequel. We will let the bold-font $\bG$ denote a connected reductive group over $F$, and we let $G:=\mathbf{G}(F)$ denote its $F$-rational points. The (complex) dual group to $\bG$ will be denoted $\hG$, and will be considered as a pinned group over $\CC$, with an action of $W_F$.  We let $\LG$ denote the L-group of $G$; strictly speaking this is a group of the form $\hG \rtimes W_F/K$, where $K$ is a subgroup of $W_F$ acting trivially on $\hG$.  Since the $K$ will be difficult to keep track of throughout our arguments, we instead write $\LG = \hG \rtimes W_F$; this should cause no confusion.

We follow the geometric Langlands convention in which all categories and functors are implicitly derived; thus $\Rep(G)$ denotes the {\em derived}\footnote{By ``derived category'' we always mean the enhancement of the classical derived category to a stable $\infty$-category.} category of smooth complex representations of $G$, and for a morphism $f$ of stacks, $f_*$ denotes derived pushforward. 
Similarly, we denote by $\otimes$ the \textit{derived} tensor product. The underlying abelian category of a derived category is indicated by the symbol $\heartsuit$, so that the usual category of smooth representations of $G$ is denoted $\Rep(G)^{\heartsuit}$.

In general, we use ``$\vee$'' upperscript to denote coroot spaces, and ``$\;\widehat{}\;$'' to denote objects that live naturally on the spectral side of the Langlands correspondence.

\subsection{Acknowledgements}
The first author is grateful to Jean-Fran\c{c}ois Dat for useful conversations on the subject of this paper, to David Ben-Zvi for helpful discussions about categorical traces, and to David Hansen for explaining the relevance of these results to the Hansen-Mann approach to the Fargues-Scholze correspondence. The authors would like to thank Xinwen Zhu for answering a question about his work and Alexander Bertoloni Meli and Teruhisa Koshikawa for explaining their cuspidal coherent sheaves to us. The third author would like to thank Xinyu Li for communications about his forthcoming work, and the U.S. National Science Foundation for financial support under Award No.~2202677 and Award No.~2546353, as well as the IHES for its wonderful hospitality and intellectually stimulating environment during the final writing stage of this paper.  

\section{Preliminaries}\label{sec:prelim}

\subsection{Stable \texorpdfstring{$\infty$}{infty}-categories and derived Morita equivalences}\label{subsec:prelim-stable-infty}

The \textit{arithmetic} categorical local Langlands correspondence is formulated in terms of stable $\infty$-categories, and although it is beyond the scope of this paper to give a complete development of this theory, we give a brief summary of the aspects of stable $\infty$-categories most relevant to our purposes in this subsection. For a more in-depth discussion that still elides the technical foundations, we refer the reader to~\cite[Appendix A]{EGH}, which was our primary reference for the summary here. For a thorough and rigorous development of the theory of stable $\infty$-categories, we refer the reader to \cite[Chapter 1]{Lurie-HA}.

A stable $\infty$-category is, morally speaking, an $\infty$-categorical version of a triangulated category.  Indeed, given a stable $\infty$-category, its homotopy category admits a natural triangulated structure, and thus one may regard a stable $\infty$-category as an ``$\infty$-categorical enrichment'' of the corresponding homotopy category.  Moreover, the most natural examples of triangulated categories--namely, \textit{classical} derived categories, can be naturally enriched to stable $\infty$-categories; more precisely, given an abelian category $\CA$, there is a functorially associated stable $\infty$-category whose homotopy category is the \textit{classical} derived category of $\CA$.  In the following, we will refer to this stable $\infty$-category associated to $\CA$ as the ``derived category of $\CA$''.

A fundamental property of stable $\infty$-categories is that the $\Hom$-spaces between two objects are naturally spectra. A functor $\mathcal{F}: \CalC \rightarrow \CD$ of stable $\infty$-categories induces morphisms of spectra on the $\Hom$-spaces between any two objects; if $\mathcal{F}$ is fully faithful, then the induced morphisms on Hom-spaces are homotopy equivalences of spectra. In particular, for any object $V$ of a stable $\infty$-category $\CalC$, the space $\End_{\CalC}(V)$ of endomorphisms of $V$ is naturally an $E_1$-ring spectrum or, equivalently, a differential graded algebra (henceforth written ``dg-algebra'').  

Let $R$ be a dg-algebra. The category $\Rmod$ of right $R$-modules is naturally triangulated,
and this triangulated category admits a natural lift to a stable $\infty$-category. 
Henceforth the notation $\Rmod$ will refer to this $\infty$-categorical lift. 
A quasi-isomorphism $R \rightarrow R'$ of dg-algebras induces an equivalence between 
$\Rmod$ and $R'$-mod. 

Fix any object $V$ of a stable $\infty$-category $\CalC$. 
For any $W\in\mathrm{Obj}(\CalC)$, the $\Hom$-space $\Hom_{\CalC}(V,W)$ has the natural structure of a right $\End_{\CalC}(V)$-module. In this way, the functor $\Hom_{\CalC}(V,-)$ may be regarded as a functor (of stable $\infty$-categories) from $\CalC$ to $\End_{\CalC}(V)\mathrm{\text{-}mod}$.  This functor has a natural left adjoint, given by the functor $M \mapsto M \otimes_{\End_{\CalC}(V)} V$; as explained above $\otimes$ denotes the {\em derived} tensor product.

Let $\CalC$ be a stable $\infty$-category. For a full subcategory $\CD$ of $\CalC$, let $\CD^{\perp}$ denote the {\em right orthogonal} of $\CD$, i.e., the full subcategory of $\CalC$ whose objects are the objects $W$ of $\CalC$ such that $\Hom_{\CalC}(V,W) = 0$ for all objects $V$ of $\CD$. The {\em left orthogonal} $\tensor*[^{\perp}]{\CD}{}$ is the full subcategory whose objects $W$ satisfy $\Hom_{\CalC}(W,V) = 0$ for all objects $V$ of $\CD$.  If $V$ is a single object of $\CalC$, considered as a full subcategory, we let $\langle V \rangle$ denote the full subcategory $\tensor*[^{\perp}]{(V^{\perp})}{}$  of $\CalC$; it is the smallest full subcategory of $\CalC$ containing $V$ and closed under colimits. We shall refer to $\langle V\rangle$ as the full subcategory of $\CalC$ generated by $V$.

Recall that an object $V$ of an $\infty$-category $\CalC$ is called {\em compact} if the functor $\Hom_{\CalC}(V,-)$ commutes with colimits.~For a compact object $V$, the functor $W \mapsto \Hom_{\CalC}(V,W)$, together with its left adjoint $M \mapsto M \otimes_{\End_{\CalC}(V)}  V$, defines an equivalence $\langle V \rangle \rightarrow \End_{\CalC}(V)\mathrm{\text{-}mod}$ of stable $\infty$-categories.

In the following, we will consider certain stable $\infty$-categories, and functors between them, both in the context of representation theory (the ``automorphic side'') and in the context of coherent sheaves on certain stacks described in \S\ref{subsec:moduli-stack-Lpar} (the ``spectral side''). 

On the \textit{automorphic} side, the primary category of interest will be the
stable $\infty$-category $\Rep(G)$, 
where $G$ is the $F$-points of a $p$-adic group $\mathbf{G}$,
associated to the abelian category $\Rep(G)^{\heartsuit}$ of smooth complex representations of $G=\bG(F)$. The relevant functors in this setting are those of parabolic induction and restriction (these functors are exact so there is no ambiguity between a functor and its derived counterpart). 
On the \textit{spectral} side of the Langlands correspondence, we will be interested in derived categories of the form $\IndCoh(X)$, where $X$ is a derived Artin stack over $\Spec\;\CC$.  (In most cases the stacks we consider will be classical Artin stacks.)  Here $\IndCoh$ denotes the category of ind-coherent sheaves on an algebraic stack $X$ (or, more generally, a derived stack).  This is the ind-completion of the bounded derived category of coherent sheaves on $X$.  As such, its compact objects are precisely those complexes with bounded, coherent cohomology.  The relevant functors will be pullbacks or pushforwards along maps $f:X \rightarrow Y$; we will use the symbols $f^*$ and $f_*$ to denote the {\em derived} versions of these functors, as we never consider the underived versions.

\subsection{The Bernstein decomposition and Whittaker data} \label{subsec:repthy}

We now recall some basic facts in the representation theory of reductive groups over nonarchimedean local fields. 
Let $\bG$ be a connected reductive group over $F$ and $G:=\mathbf{G}(F)$. For any parabolic subgroup 
$\bP = \bM \bU_\bP$ of $\bG$, defined over $F$, we have (normalized) parabolic induction and restriction functors: 
$i_P^G: \Rep(M) \rightarrow \Rep(G)$ and $r^G_P: \Rep(M) \rightarrow \Rep(G)$. These functors are exact at the level 
of abelian categories and $i_P^G$ is right adjoint to $r^G_P$. The left adjoint of $r^G_P$ is $i_{\overline P}^G$,
where $\overline P$ denotes the parabolic subgroup of $G$ opposite to $(P,M)$.

Recall that a smooth complex representation $\sigma$ of $G$ is said to be {\em supercuspidal} if $\sigma$ is not isomorphic to a $G$-subquotient of $i_P^G \tau$, for any \textit{proper} parabolic $\bP = \bM\bU$ of $\bG$ defined over $F$ and any irreducible smooth representation $\tau$ of $M$.  Over the complex numbers, this is equivalent to requiring that the parabolic restrictions $r_P^G \sigma$ vanish for all proper parabolics $\bP$ as above.

A pair $(M,\sigma)$, where $\bM$ is a Levi subgroup of a parabolic $\bP$ of $\bG$ defined over $F$ and $\sigma$ is an irreducible supercuspidal representation of $M$, is called a {\em supercuspidal pair} for $G$.  A supercuspidal pair $(M,\sigma)$ is said to be a {\em supercuspidal support} 
for $\pi$ if $\pi$ is equivalent to a $G$-subquotient of $i_P^G \sigma$ (note that this condition is independent of the choice of a parabolic $\bP$ containing $\bM$).  Any two supercuspidal supports of $\pi$ are $G$-conjugate, so we will often refer to the $G$-conjugacy class of the pair $(M,\sigma)$ as \textit{the} supercuspidal support of $\pi$.

Let $(M,\sigma)$ and $(M',\sigma')$ be two supercuspidal pairs for $G$.  We say they are {\em inertially equivalent} if there is an element $g$ of $G$ that conjugates $M$ to $M'$ and $\sigma$ to a twist of $\sigma'$ by an unramified character of $M$ (henceforth referred to as an \textit{unramified twist} of $\sigma'$). This defines an equivalence relation on the set of all such pairs $(M,\sigma)$, and the equivalence classes are referred to as \textit{Bernstein inertial classes}. For a given supercuspidal pair $(M,\sigma)$, we denote its inertial equivalence class by $[M,\sigma]_G$, or simply $[M,\sigma]$ when the group $G$ is understood. This gives rise to a full (abelian) subcategory $\Rep(G)^{\heartsuit}_{[M,\sigma]}$ whose objects consists of all objects $\pi$ of $\Rep(G)^{\heartsuit}$ for which every irreducible $G$-subquotient of $\pi$ has supercuspidal support in the inertial class $[M,\sigma]_G$.

A fundamental theorem of Bernstein and Deligne~\cite{BD} gives us a direct product decomposition (called the \textit{Bernstein decomposition}):
$$\Rep(G)^{\heartsuit} \cong \prod_{[M,\sigma]} \Rep(G)^{\heartsuit}_{[M,\sigma]},$$
where $[M,\sigma]$ runs through the inertial equivalence classes of supercuspidal pairs for $G$.  This induces a corresponding decomposition of $\Rep(G)$ as a product of full subcategories $\Rep(G)_{[M,\sigma]}$, called \textit{Bernstein blocks} (or sometimes just ``blocks'' for short); each $\Rep(G)_{[M,\sigma]}$ is then the derived stable $\infty$-category associated to the abelian category $\Rep(G)_{[M,\sigma]}^{\heartsuit}$.  For any object $V$ of $\Rep(G)$, we will let $V^{[M,\sigma]}$ denote its projection to the direct factor $\Rep(G)_{[M,\sigma]}$ of $\Rep(G)$.

If $\bQ = \bL\bU_\bQ$ is a parabolic subgroup of $\bG$ over $F$ such that $\bL$ contains $\bM$, then the functor $i_Q^G$ restricts to a functor $\Rep(L)_{[M,\sigma]} \rightarrow \Rep(G)_{[M,\sigma]}$ for any inertial equivalence class of the form $[M,\sigma]$.

By a {\em Whittaker datum} for $G$, we mean a pair $(U,\psi)$, where $U:=\bU(F)$ for the unipotent radical $\bU$ of a Borel subgroup $\bB$ defined over $F$, and $\psi: U \rightarrow \CC^{\times}$ is a {\em generic} character of $U$, i.e.~a character whose orbit under conjugation by $T$, for some maximal torus $\bT$ of $\bB$, is of maximal dimension. 
The $G$-conjugacy classes of such data form a torsor for the group $G^{\ad}/G$, under conjugation.

Associated to $(U,\psi)$, we have the space $\mathcal{W}_{U,\psi}$ of {\em compact Whittaker functions} defined by the compact induction $\ind_U^G \psi$. This is a projective object of $\Rep(G)^{\heartsuit}$, and its projection to each direct factor $\Rep(G)_{[M,\sigma]}^{\heartsuit}$ of $\Rep(G)^{\heartsuit}$ is finitely generated as a $\CC[G]$-module. The representation $\cW_{U,\psi}$ depends only on the $G$-conjugacy class of $(U,\psi)$.

For any Levi subgroup $\bM$ of a parabolic subgroup $\bP$ of $\bG$ defined over $F$, the choice of a Whittaker datum $(U,\psi)$ for $G$ induces a corresponding choice of a Whittaker datum $(U_M,\psi_M)$ for $M$, up to $M$-conjugacy.  The datum $(U_M,\psi_M)$ is characterized, up to $M$-conjugacy, by the existence of an isomorphism: $r_{\overline P}^G \cW_{U,\psi} \cong \cW_{U_M,\psi_M}$. Upon fixing a Whittaker datum $(U,\psi)$, we let $\cW_G$ denote the space $\cW_{U,\psi}$, and $\cW_G^{[M,\sigma]}$ its projection to the direct factor $\Rep(G)^{\heartsuit}_{[M,\sigma]}$ of $\Rep(G)^{\heartsuit}$. Similarly, for any standard Levi $L$ of $G$ let $\cW_L$ denote $\cW_{U_L,\psi_L}$. 

Recall that an irreducible representation $\pi$ of $G$ is said to be {\em $(U,\psi)$-generic} (or simply {\em generic} when there is an implicit choice of Whittaker datum), if there exists a nonzero map $\cW_{U,\psi} \rightarrow \pi$.  We will say that a supercuspidal pair
$(M,\sigma)$ is generic if there exists a nonzero map $\cW_{U_M,\psi_M} \rightarrow \sigma$; note that genericity depends only on the inertial equivalence class $[M,\sigma]$ of $(M,\sigma)$. In particular, we will refer to the Bernstein block $\Rep(G)_{[M,\sigma]}$ as a \textit{generic} Bernstein block when $[M,\sigma]_G$ is generic. We have:

\begin{lemma}
Let $[M,\sigma]$ be an inertial class and fix a Whittaker datum $(U,\psi)$. Then $[M,\sigma]$ is generic if, and only if, the projection $\cW_G^{[M,\sigma]}$ to $\Rep(G)^{\heartsuit}_{[M,\sigma]}$ of $\cW_G$ is nonzero.  

Moreover, in this case, for any parabolic subgroup $\bP$ of $\bG$ defined over $F$ with Levi $\bM$, the representation $i_P^G \cW_M^{[M,\sigma]}$ is a progenerator of $\Rep(G)_{[M,\sigma]}^{\heartsuit}$.
\end{lemma}
\begin{proof}
First assume that $(M,\sigma)$ is generic.  Then we have a nonzero map 
$\cW_M \rightarrow \sigma$, 
and hence (via the isomorphism $r^G_{\overline P} W_G \cong \cW_M$) a nonzero map  $\cW_G\to i_{\overline P}^G \sigma$.  Since $i_{\overline P}^G \sigma$ lies in the block $\Rep(G)_{[M,\sigma]}^{\heartsuit}$, it follows that the projection of $\cW_G$ to this block is nonzero.

Conversely, assume by contradiction that $\Hom_M(\cW_M,\sigma)$ is zero.  Let $I(\sigma)$ denote the object $\Ind_{M^1}^M \sigma|_{M^1}$, where $M^1$ is the smallest subroup of $M$ containing all compact subgroups of $M$ (equivalently, $M^1$ is the intersection of all $\ker\chi$ as $\chi$ ranges over the group $\mathfrak{X}_{\mathrm{nr}}(M)$ of unramified characters of $M$).  Then $I(\sigma)$ is an injective object in $\Rep(G)_{[M,\sigma]}^{\heartsuit}$, and an easy application of the Mackey formula shows that $\Hom_M(\cW_M,I(\sigma))$ is zero.  Moreover, every unramified twist of $\sigma$ embeds in $I(\sigma)$.  The same argument as the previous paragraph shows that $\Hom_G(\cW_G,i_{\overline P}^G I(\sigma))$ is zero.  But every simple object of $\Rep(G)^{\heartsuit}_{[M,\sigma]}$ embeds in $i_{\overline P}^G I(\sigma)$, so this implies that the projection of $\cW_G$ to this block is zero.

To see the second claim, when $(M,\sigma)$ is generic, every unramified twist of $\sigma$ is a quotient of $\cW_M$, so that every simple object of $\Rep(G)^{\heartsuit}_{[M,\sigma]}$ is a quotient of $i_P^G \cW_M^{[M,\sigma]}$.  But the latter is projective and finitely generated as a $\CC[G]$-module, thus is a progenerator of $\Rep(G)_{[M,\sigma]}^{\heartsuit}$ as claimed.
\end{proof}

\begin{remark}
The generic case is more concrete than the general case of Conjecture \ref{conj:categorical-intro}, because we have explicit generators on both sides of \eqref{eqn:fully-faithful-intro}, which match via $\LocL_G$: on the automorphic side we have the part of the Whittaker module $\mathcal W_G$ in $\Rep_{[M,\sigma]}(G)$, while on the spectral side we have the parabolic induction of the structure sheaf on a component of $X_\LM$.
\end{remark}

\section{The categorical local Langlands correspondence}\label{sec:categorical-LLC}

As in the previous sections, let $F$ be a nonarchimedean local field, and $\bG$ a quasi-split, connected, reductive group over $F$.  We denote by ${\hG}$ the complex dual group of $\bG$, and by $\LG$ the L-group ${\hG} \rtimes W_F$.  We begin by recalling the conjectural categorical local Langlands correspondence for (the quasisplit form of) $\bG$.

\subsection{The moduli space of Langlands parameters}\label{subsec:moduli-stack-Lpar}

Let $X_{\LG}$ denote the moduli space of Langlands parameters for $\bG$ (see for example \cite{moduli}, where this stack is denoted $[Z^1(W_F,{\hat G})/{\hat G}]$, or \cite{Zhu-tame}, where this stack is denoted $\mathrm{Loc}_{\LG,F}$).  
Since we are working over $\CC$ throughout, we may regard $X_{\LG}$ as the quotient stack ${\widetilde X}_{\LG}/{\hG}$, where ${\widetilde X}_{\LG}$ is the moduli scheme parameterizing Langlands parameters $\phi$ for $\LG$ over $\CC$.  Such a parameter is given by a pair $(\rho,N)$, where $\rho: W_F \rightarrow \LG$ is an L-homomorphism with open kernel, and $N$ is a (necessarily nilpotent) element of $\Lie({\hG})$ such that for all $w$ in $W_F$, we have $\Ad_{\rho(w)} N = \| w \| \cdot N$. The scheme ${\widetilde X}_{\LG}$ is an infinite disjoint union of connected affine schemes, each of which is a reduced, finite type, local complete intersection over $\CC$.  The quotient $X_{\LG}$ is thus locally of finite type and a local complete intersection.

Let $\bP = \bM \bU_{\bP}$ be a parabolic subgroup of $G$ over $F$. Then we may consider the spaces of Langlands parameters $X_{\LM}$ and $X_{\LP}$ for $\bM$ and $\bP$, respectively; these are defined analogously as above. There are natural derived structures on the stacks $X_{\LG}$, $X_{\LP}$, and $X_{\LM}$, respectively. We refer the reader to~\cite{Zhu-tame} for details. Note that since $\LG$ and $\LM$ are reductive
\footnote{More precisely, the groups $\hG \rtimes W_F / K$ and $\hM \rtimes W_F / K$ 
as in \S\ref{par:notation} are reductive.}, 
the derived structures on these stacks are trivial; by contrast the derived structure on $X_{\LP}$ can be highly nontrivial in general. 
There are natural maps $r_P: X_{\LP} \rightarrow X_{\LM}$ and $\pi_P: X_{\LP} \rightarrow X_{\LG}$ 
induced by the maps $\LP \rightarrow \LM$ and $\LP \rightarrow \LG$. Note that the map $\pi_P$ is proper, whereas the map $r_P$ is neither flat nor proper.

\subsection{The conjectured correspondence}

Recall that $G$ denotes the group of $F$-points of $\bG$.  Let $\Rep(G)$ denote the derived category of complex representations of $G$, considered as a stable $\infty$-category as in \S \ref{subsec:prelim-stable-infty}, and let $\Rep(G)^{\omega}$ be the full subcategory of compact objects in $\Rep(G)$. More concretely, the objects of $\Rep(G)^{\omega}$ are those complexes with bounded cohomology, whose cohomology groups are each finitely generated as $\CC[G]$-modules. The categorical local Langlands correspondence is a conjectural relationship between the derived categories $\Rep(G)$ and $\IndCoh(X_{\LG})$.  

The categorical local Langlands correspondence depends on a ``normalizing'' choice of Whittaker datum $(U,\psi)$ for $G$, and a Borel pair $(\bB,\bT)$ of $\bG$; we assume that $\bT$ contains a maximal $F$-split torus of $\bG$.   We call a parabolic subgroup of $\bG$ \textit{standard} if it contains $\bB$, and a Levi subgroup \textit{standard} if it contains $\bT$.  For any standard Levi subgroup $\bM$ of $\bG$, we have the corresponding Whittaker datum $(U_M,\psi_M)$ as in \S\ref{subsec:repthy}, and the space $\cW_M$ of compact Whittaker functions $\ind_{U_M}^M \psi_M$.

There are several formulations of the categorical local Langlands correspondence in the literature (see for example \cite{Zhu-stack, Fargues-Scholze}), but when one only considers quasi-split groups $\bG$ (as we do here), all these various formulations has the following special form (which we will henceforth refer to as ``the'' categorical local Langlands correspondence for $\bG$ in this article):

\begin{conj}[Categorical Local Langlands Conjecture] \label{conj:categorical} \ \\
For each standard Levi subgroup $M$ of $G$, there is a fully faithful functor 
$$\LocL_M: \Rep(M) \rightarrow \IndCoh(X_{\LM})$$
such that
    \begin{enumerate}
        \item $\LocL_M(\cW_M)$ is isomorphic to the structure sheaf $\CO_{X_{\LM}}$; 
        \item for each standard parabolic subgroup $P = M U_P$ of $G$, defined over $F$, the diagram
        \begin{equation}\label{eqn:diagram-conjecture}
        \begin{tikzcd}\Rep(M)\arrow[]{r}{\LocL_M}\arrow[]{d}[swap]{i_P^G} & \IndCoh(X_{\LM})\arrow[]{d}{(\pi_P)_* r_P^*}\\ 
        \Rep(G)\arrow[]{r}{\LocL_G} &\IndCoh(X_{\LG})\end{tikzcd} .
        \end{equation}  
    \end{enumerate}
\end{conj}
The functor $(\pi_P)_* r_P^*$ is sometimes referred to as the ``spectral parabolic induction'' functor. We remind the reader that these functors are their derived versions, as we are regarding $\IndCoh$ as a stable $\infty$-category.

Note that the functor $\LocL_T$ exists unconditionally, and may be constructed in a straightforward way from local class field theory; in this case, $\LocL_T$ arises from a full embedding on the level of abelian categories. 
However, this is not true in general, not even for $\GL_2$.

\subsection{The unipotent correspondence}\label{subsec:unipotent-LLC}

Now suppose that $\bG$ is an unramified group; that is, that $\bG$ splits over an unramified extension of $F$. We fix a Borel subgroup $\bB$ of $\bG$ and a maximal torus $\bT$ of $\bB$.  For such groups, the categorical local Langlands correspondence is known on the {\em principal block} $\Rep(G)_{[T,1]}$ of $G$.  (Note that this is the block of $\Rep(G)$ containing all unramified principal series representations, and in particular the trivial representation.)  

For any standard Levi subgroup $\bM$ of $\bG$, we denote by $\cW_M^{[T,1]}$ the projection of $\cW_M$ to the block $\Rep(M)_{[T,1]}$ of $\Rep(M)$.

Let $T^1$ be the maximal compact subgroup of $T$; then one has the projective representation $\ind_{T^1}^T 1$ of $T$.  The endomorphism algebra of this representation is the group algebra $\CC[T/T^1]$, as can easily be seen via the Mackey formula; the action of $T$ is then given by the ``universal unramified character''
$$\chi^{\un}: T \rightarrow \CC[T/T^1]^{\times}$$
that takes an element $t$ of $T$ to its class in $T/T^1$.  For this reason we will often denote $\ind_{T^1}^T 1$ by $\chi^{\un}$ in what follows.
The representation $\chi^{\un}$ is a projective generator of the direct factor $\Rep(T)_{[T,1]}$ of $T$; indeed, it is isomorphic to $\cW_T^{[T,1]}$.

The parabolic induction $i_{B}^G \ind_{T^1}^T 1$ is a projective object of $\Rep(G)$, and it generates the direct factor $\Rep(G)_{[T,1]}$ of $\Rep(G)$, which we henceforth call the {\em principal block} of $G$. This is the block of $\Rep(G)$ containing the trivial representation of $G$.  This parabolic induction $i_{B}^G \chi^{\un}$ is isomorphic to the induction $\ind_I^G 1$, where $I$ is an Iwahori subgroup of $G$; in particular, its endomorphism algebra is the Iwahori--Hecke algebra $\mathcal{H}(G,I)$.   This is an affine Hecke algebra with, in general, \textit{unequal parameters} in the sense of Lusztig \cite{Lusztig-AHAgraded}.

The Langlands parameters of representations in $\Rep(G)_{[T,1]}$ are precisely the {\em unipotent} Langlands parameters, i.e., those parameters corresponding to pairs $(\rho,N)$ such that the restriction of $\rho$ to $I_F$ is trivial.  (Note that since $G$ is unramified, this condition is invariant under conjugation by $\hG$.) More generally, every irreducible unipotent $G$-representation has a unipotent Langlands parameter.
Denote by $X^1_{\LT}$, $X^1_{\LB}$ and $X^1_{\LG}$ the spaces of {\em unipotent} Langlands parameters for $T$, $B$ and $G$, respectively; these are closed substacks (in fact, connected components) of $X_{\LT}$, $X_{\LB}$, and $X_{\LG}$, respectively.  Note that the maps $r_B$ and $\pi_B$ restrict to maps from $X^1_{\LB}$ to $X^1_{\LT}$ and $X^1_{\LG}$, respectively.
An important role will be played by the sheaf
\begin{equation}\label{eqn:usual-coherent-Spring-sheaf}
\CS^1_G := (\pi_B)_* r_B^* \CO_{X^1_{\LT}} \cong (\pi_B)_* \CO_{X^1_{\LB}}
\end{equation}
on $X^1_{\LG}$, which we will call the {\em (usual) coherent Springer sheaf}. 

We then have the following theorem, due to Ben-Zvi, Chen, Nadler, and the first author~\cite{BZCHN} when $G$ is split and one restricts to the block consisting of Iwahori-spherical representations, and due to Zhu~\cite{Zhu-tame} in general:

\begin{thm}[Unipotent Categorical Local Langlands Correspondence] \label{thm:unipotent categorical} \ \\
Let $G$ be an unramified reductive group over $F$. Let $M$ be a standard Levi subgroup of $G$ and let 
$\Rep (M)_{[M',\sigma]}$ be a Bernstein block consisting of unipotent $M$-representations. 
There exists a fully faithful functor:
\begin{equation}\label{eqn:LLC1M}
    \LocL^1_M: \Rep(M)_{[M',\sigma]} \rightarrow \IndCoh(X^1_{\LM})
\end{equation}
    such that 
    \begin{enumerate}
        \item    
    $\LocL^1_M \big( \cW_M^{[T,1]} \big)$ is isomorphic to the structure sheaf $\CO_{X^1_{\LM}}$;
    \item \label{compatibility-Springer} $\LocL^1_M(\ind_{I_M}^M 1)$ is isomorphic to the coherent Springer sheaf $\CS^1_M$.
    \item\label{compability-parabolic-induction} The functors \eqref{eqn:LLC1M} are compatible with parabolic induction, in the following sense: for each standard parabolic subgroup $P=M U_P$ of $G$, the following diagram commutes:
    \begin{equation}\label{eqn:diagram-unipotent-LLC}
        \begin{tikzcd}\Rep(M)_{[T,1]}\arrow[]{r}{\LocL^1_M}\arrow[]{d}[swap]{i_{P}^G} & \IndCoh(X^1_{\LM})\arrow[]{d}{(\pi_P)_* r_P^*}\\ 
        \Rep(G)_{[T,1]}\arrow[]{r}{\LocL^1_G} &\IndCoh(X^1_{\LG})\end{tikzcd}
        \end{equation}
\item\label{compatibility-pinned-aut} The functor \eqref{eqn:LLC1M} is compatible with pinned automorphisms of reductive groups; 
\item\label{compatibility-unr-char} The functor 
$\LocL^1_G: \Rep(G)_{[M',\sigma]} \rightarrow \IndCoh(X^1_{\LG})$
is compatible with twisting by unramified characters. 
    \end{enumerate}
\end{thm}
\begin{proof}[Proof of Theorem \ref{thm:unipotent categorical} (1)--(2) and (4)--(5)]\textit{We defer the proof of Part (3) to a later paragraph; see Proof~\ref{proof of part (3)}.}
Most of the statements can be found in \cite[\S 5.2]{Zhu-tame}. We recall some of Zhu's notations and constructions,
and we provide arguments for the parts that are not explicitly addressed in \cite{Zhu-tame}. 

For the duration of this proof we work with coefficients in $\overline{\QQ}_\ell$ for a prime $\ell \neq p$. 
The final results can be translated back to representations and sheaves with coefficients in $\CC$, via a field 
isomorphism $\overline{\QQ}_\ell \cong \CC$. The notation $G$ in \cite{Zhu-tame} is
our notation $\bG$, and therein our functor $\LocL_G^1$ is denoted as $\mathbb L_\bG^\unip$, which arises from a commutative diagram
\begin{equation}\label{eq:3.11}
\begin{tikzcd}
\mathrm{IndShv}_{\mathrm{f.g.}} (\mathrm{Iw} \backslash L \bG / \mathrm{Iw}) \arrow[]{r}{\mathbb B^\unip} 
\arrow[]{d}[swap]{\mathrm{Ch}^\unip_{L\bG,\phi}} & 
\IndCoh (S^\unip_{\cG,\breve{F}}) \arrow[]{d}{\mathrm{Ch}^\unip_{\cG,\phi}} \\ 
\mathrm{IndShv}^{\unip} (\mathrm{Isoc}_{\bG})  \arrow[]{r}{\LocL^1_G} &\IndCoh (\Loc^\unip_{\cG,F})
\end{tikzcd} . 
\end{equation}
The group $\cG$ is a variant of $\LG$, and since we work over $\overline{\QQ}_\ell$ these two are equivalent
\cite[Remark 2.4]{Zhu-tame}. In this setting, the stack $\Loc^{\unip}_{\cG,F}$ from
\cite[\S 2.2.2]{Zhu-tame} is none other than our $X^1_\LG$.
The upper horizontal arrow in diagram \eqref{eq:3.11} is \cite[(5.1)]{Zhu-tame}, which generalizes Bezrukavnikov's 
equivalence \cite{Bezrukavnikov,Bando} using the Steinberg stack
\[
S^\unip_{\cG,\breve{F}} := 
\Loc_{{}^c \bB,\breve{F}}^\unip \times_{\Loc_{\cG,\breve{F}}^\unip} \Loc_{{}^c \bB,\breve{F}}^\unip .
\]
Since we work over $\overline{\QQ}_\ell$, the stack $\Loc^{\widehat \unip}_{\cG,F}$ from 
\cite[\S 2.2.2]{Zhu-tame} is the same as $\Loc^\unip_{\cG,F}$, and similarly for other cases
of $\widehat{\unip}$ v.s. unip. The symbol $\phi$ in \eqref{eq:3.11} indicates a pinned action 
of a Frobenius element on $\bG$ and $\hG$. Both vertical arrows in \eqref{eq:3.11} indicate 
taking a $\phi$-twisted categorical trace, see \cite[(2.61) and Theorem 4.125]{Zhu-tame}. 
The equivalence $\LocL_G^1$ and the commutativity of 
\eqref{eq:3.11} follow from \cite[Theorem 5.3]{Zhu-tame}.

For part (1), see \cite[Theorem 5.3.(2)]{Zhu-tame}. By naturality, both vertical
functors in \eqref{eq:3.11} are compatible with pinned automorphisms $\tau$ of $\bG$ and
$\hat \tau$ of $\hG$. Such a $\tau$ may change the Frobenius action, so it may send $\bG$ to
a group with a different $F$-rational structure. The precise statements of compatibility are
\[ 
\tau_* \circ \mathrm{Ch}^\unip_{L \bG, \phi} = \mathrm{Ch}^\unip_{L \bG, \tau \phi \tau^{-1}}
\circ \tau_* \quad \text{and} \quad
\hat \tau_* \circ \mathrm{Ch}^\unip_{\cG, \phi} = \mathrm{Ch}^\unip_{\cG, \tau \phi 
\hat \tau^{-1}} \circ \hat \tau_* .
\]
In view of the commutativity of the diagram \eqref{eq:3.11}, for part (4) it remains to check
that $\mathbb B^\unip$ is compatible with pinned automorphisms. This is already used in the 
main results of \cite{Zhu-tame}, and addressed in \cite[Remark 5.2 (3)]{Zhu-tame}. 
It is not entirely in the literature yet, but will appear in forthcoming work of Xinyu Li.

For part (5), we consider an unramified character $\chi \in \Xnr (G)$. Let $\phi_\chi$ be its Langlands parameter,
viewed as a 1-cocycle of $W_F/I_F$ with values in $Z(\hG)^{I_F}$, or equivalently as an
element $\phi_\chi (\Fr) \in (Z(\hG)^{I_F} )_\Fr^\circ$.
The compatibility \eqref{compatibility-unr-char} follows from compatibility with automorphisms of pinned groups, 
applied to the group $\bG' := \bG \times \bG/\bG^{\der}$.  Indeed, $\bG'$ admits an $F$-automorphism $\psi$ that takes 
a pair $(g,x)$ to $(g,[g]x)$, where $[g]$ denotes the image of $g$ in $\bG/\bG^{\der}$. For $\pi \in \Rep (G)_{[M',\sigma]}$,
the action of $\psi$ on $\pi \boxtimes \chi \in \Rep (G')$ is
\begin{equation}\label{eq:3.9}
(\pi \boxtimes \chi) \circ \psi = (\chi \otimes \pi) \boxtimes \chi .
\end{equation}
On the spectral side we have $\hG' = \hG \times Z(\hG)^\circ$ and $\hat \psi (g,z) = (zg,z)$. Furthermore 
$X^1_{\LG'} = X^1_\LG \times X^1_{\LG / \bG^\der}$ and $\LocL^1_{G'}$ can be identified with 
$\LocL^1_G \boxtimes \LocL^1_{(\bG / \bG^\der)(F)}$. Via the local Langlands correspondence for tori, $\chi$ 
corresponds to the skyscraper sheaf $\mathcal F_{\chi}$ at $\phi_\chi$ on $X^1_{\LG / \bG^\der}$. Recall that 
$(Z(\hG)^{I_F} )_\Fr^\circ$ acts naturally on $X^1_{\LG}$ by
adjusting only the values of L-parameters at $\Fr$. For any $\mathcal F \in \IndCoh(X^1_{\LG})$, we obtain
\begin{equation}\label{eq:3.10}
\hat \psi_* (\mathcal F \boxtimes \mathcal F_\chi) = (\phi_\chi)_* \mathcal F \boxtimes \mathcal F_\chi ,
\end{equation}
where $(\phi_\chi)_*$ means pushforward along the action of $\phi_\chi$. Comparing \eqref{eq:3.9} and \eqref{eq:3.10}
and using part (4) for $\bG'$, we deduce that 
\[
\LocL^1_G (\chi \otimes \pi) \boxtimes \mathcal F_\chi = \LocL^1_{G'}(\chi \otimes \pi \boxtimes \chi) =
((\phi_\chi)_* \boxtimes \mathrm{id}) \LocL^1_{G'}(\pi \boxtimes \chi) = (\phi_\chi)_* \LocL^1_G (\pi) \boxtimes 
\mathcal F_\chi .
\]
It follows that $\LocL^1_G \circ (\chi \otimes) = (\phi_\chi )_* \circ \LocL^1_G$, which proves (5). 
\end{proof} 

To prepare for proving part (3) of Theorem \ref{thm:unipotent categorical}, we study the relevant objects in the categories in \eqref{eq:3.11}. We have the standard 
sheaf $\Delta_e = \nabla_e$ on $\mathrm{Iw} \backslash L \bG / \mathrm{Iw}$, where $e$ denotes the unit element 
of $\bW (\bG,\bT) \ltimes X_* (\bT)$. 
The representation $\ind_I^G 1 = \ind_{\mathrm{Iw}(k_F)}^{L\bG (k_F)}\mathrm{triv}$
can be pushed forward to a sheaf on $\mathrm{Isoc}_\bG$. For a stack $S$, let $\omega_S$ denote the dualizing
(ind-)coherent sheaf on $S$. For $w \in W_0 = \bW (\bG,\bT)$, the stacks $\Loc^\unip_{\cdots}$ have a substack $\Loc^\unip_{\cdots, w}$ 
determined by the condition that a pair of Borel subgroups of $\hG$ is in relative position $w$. In particular, we have 
$S^\unip_{\cG,\breve F,1} \subset S^\unip_{\cG,\breve F}$. In \cite[(2.70)]{Zhu-tame}, 
Zhu introduces the coherent Springer sheaf 
\begin{equation}\label{eq:3.13}
\mathrm{CohSpr}^\unip_{\cG,F} := (\pi_B^\unip )_* \omega_{\Loc^\unip_{{}^c \bB, F}} = 
(\pi_B^\unip )_* \CO_{\Loc^\unip_{{}^c \bB, F}}, 
\end{equation}
where $\pi_B^\unip : \Loc^\unip_{{}^c \bB,F} \to \Loc^\unip_{\cG,F}$ is the restriction of
$\pi_B : X_{\LB} \to X_{\LG}$. Comparing the right-hand side of \eqref{eq:3.13} to 
\eqref{eqn:usual-coherent-Spring-sheaf}, we see that $\mathrm{CohSpr}^\unip_{\cG,F} = \CS^1_G$. 
These objects fit in the diagram \eqref{eq:3.11} as
\begin{equation}\label{eq:3.12}
\begin{tikzcd}
\Delta_e \arrow[|->]{r} \arrow[]{d} & \omega_{S^\unip_{\cG, \breve{F},1}} \arrow[|->]{d} \\ 
i_{1,*} (\ind_I^G 1) \arrow[|->]{r} & \mathrm{CohSpr}^\unip_{\cG,F} = \CS^1_G
\end{tikzcd} .
\end{equation}
For the left downward map, see \cite[(4.44) and Corollary 4.68]{Zhu-tame}; 
for the upper horizontal map, see 
\cite[(5.9)]{Zhu-tame}; for the right downward map, see \cite[Lemma 2.79]{Zhu-tame}. The lower horizontal map of
\eqref{eq:3.12} is stated in \cite[Theorem 5.3.(2)]{Zhu-tame}, and this gives our part (2). 

We now consider the endomorphism algebras of the objects in \eqref{eq:3.12}. These can be regarded as 
dg-algebras concentrated in degree zero. Since the horizontal functors in
\eqref{eq:3.11} are fully faithful, the two objects in the top (resp.~bottom) horizontal row of 
\eqref{eq:3.12} have isomorphic endomorphism algebras. By \cite{Bezrukavnikov,Bando}, in terms of the notations 
from \cite[\S 4.2.2 and (5.8)]{Zhu-tame}, we know that
\begin{equation}\label{eq:3.14}
\End (\Delta_e) \text{ has a } \overline{\QQ}_\ell \text{-basis } 
\{ \Delta_w * J_\lambda : w \in \bW (\bG,\bT), \lambda \in X_* (\bT) \} .
\end{equation}
Here the Wakimoto sheaf $J_\lambda$ equals $\Delta_\lambda$ when $\lambda$ is positive with respect to
$\bB$; see \cite[Remark 5.2.(2)]{Zhu-tame}. By \cite[(5.8)--(5.9) and \S 2.77--2.79]{Zhu-tame}, we see that via the 
functor $\mathbb B^\unip$, the statement \eqref{eq:3.14} becomes the statement that
\[
\End \big( \omega_{S^\unip_{\cG, \breve{F},1}} \big) \text{ has a } \overline{\QQ}_\ell \text{-basis } 
\big\{ \omega_{S^\unip_{\cG,\breve{F},w}}(\lambda)  : w \in \bW (\bG,\bT), \lambda \in X_* (\bT) \big\} .
\]
Moreover, by \cite{Bezrukavnikov,Bando} we know that $\End (\Delta_e) \cong \End \big( \omega_{S^\unip_{\cG, \breve{F},1}} \big) $ is an affine Hecke algebra with 
a single parameter $q_F$.
As in \cite[Corollary 5.5.(1)]{Zhu-tame}, one has  
\begin{equation}\label{eq:3.15}
\End \big( i_{1,*} (\ind_I^G 1) \big) \cong \End (\ind_I^G 1) \cong \CH (G,I) \cong \End (\CS^1_G) .
\end{equation}
We recall that 
\[
(\bW (\bG,\bT) \ltimes X_* (\bT))^\phi = \bW (\bG,\bT)^\phi \ltimes X_* (\bT)^\phi
\]
is the affine Weyl group of $G$ with respect to~$I$.
The algebra \eqref{eq:3.15} is an affine Hecke algebra (with unequal parameters if $\bG$ is not $F$-split), and
it has a $\overline{\QQ}_\ell$-basis 
\begin{equation}\label{eq:3.26}
\{ T_w \theta_\lambda : w \in \bW (\bG,\bT)^\phi , \lambda \in X_* (\bT)^\phi \} .
\end{equation}
More precisely, $T_w \in \CH (G,I)$ has support $I w I$ (and takes value 1 there if the volume of $I$ equals 1). 
Following \cite{Morris-Inventiones}, we use the bijection $X_* (\bT)^\phi \to T / (T \cap I)$ 
from evaluating cocharacters at the inverse of a uniformizing element $\varpi_F$ of $\mathfrak o_F$. This 
convention is necessary to preserve positivity on $X_* (\bT)^\phi$. For any $\lambda \in X_* (\bT)$ which is 
positive with respect to $\bB$, we have that $\theta_\lambda$ is supported on $I \lambda (\varpi_F^{-1}) I$, while the 
support of $\theta_{-\lambda} = \theta_\lambda^{-1}$ may contain more than one double $I$-coset.

The subalgebra of $\CH (G,I)$ spanned by the $\theta_\lambda$ is isomorphic to the subalgebra of 
$\End (\Delta_e)$ spanned by the $J_\lambda$ with $\lambda \in X_* (\bT)^\phi$, as both are isomorphic to
$\overline{\QQ}_\ell [X_* (\bT)^\phi]$. However, the subalgebra of $\CH (G,I)$ spanned by the $T_w$ with
$w \in \bW (\bG,\bT)^\phi$ is in general not naturally isomorphic to a subalgebra of $\End (\Delta_e)$.
The functor $\mathrm{Ch}^\unip_{L\bG,\phi}$ is $\overline{\QQ}_\ell$-linear and sends $\Delta_w * J_\lambda$ 
to $T_w \theta_\lambda$ \cite[(5.13)]{Zhu-tame}, but it is not multiplicative on these endomorphism algebras.

By \cite[Lemma 2.79.(3) and (5.13)]{Zhu-tame}, the algebra $\End (\mathrm{CohSpr}^\unip_{\cG,F}) \cong \CH (G,I)$ has a 
$\overline{\QQ}_\ell$-basis 
\begin{equation}\label{eq:3.27}
\big\{ \LocL_G^1 (T_w \theta_\lambda) = (\tilde{\pi}_{B,w}^\unip )_* \omega_{\widetilde{\Loc}^\unip_{\cG,F,w}} 
(\lambda) : w \in \bW (\bG,\bT)^\phi, \lambda \in X_* (\bT)^\phi \big\} .
\end{equation}
Here $\widetilde{\Loc}^\unip_{\cG,F,w} = \widetilde{\Loc}^{\mathrm{tame}}_{\cG,F,w} 
\times_{\Loc^{\mathrm{tame}}_{\cG,\breve{F}}} \Loc^\unip_{{}^c \bB,F}$ as in \cite[(2.47)]{Zhu-tame} and the map
\[
\tilde{\pi}^\unip_{B,w} : \widetilde{\Loc}^\unip_{\cG,F,w} \to \Loc^{\mathrm{tame}}_{\cG,F} ,
\]
is induced by $\pi_B : X_{\LB} \to X_{\LG}$. The commutativity of \eqref{eq:3.11}
guarantees that $\LocL_G^1 (T_w \theta_\lambda) \in \Coh (\Loc^{\unip}_{\cG,F})$.\\

\begin{proof}[Proof of Theorem \ref{thm:unipotent categorical}.(3)]\label{proof of part (3)}
We now study how parabolic induction interacts with the above constructions. There are natural isomorphisms of $G$-representations
\begin{equation}\label{eq:3.16}
i_P^G (\ind_{I_M}^M 1) \cong i_P^G \big( i_{B \cap M}^M (\ind_{T^1}^T 1) \big) \cong 
i_B^G (\ind_{T^1}^T 1) \cong \ind_I^G 1 .
\end{equation}
By \eqref{eq:3.16} and \cite[Lemma 5.1]{Solleveld2019}, $i_P^G$ induces an algebra homomorphism 
\begin{equation}\label{eq:3.17}
\begin{array}{ccc}
\CH (M,I_M) \cong \End_M \big( i_{B \cap M}^M (\ind_{T^1}^T 1 ) \big) & \longrightarrow & 
\End_G \big( i_B^G (\ind_{T^1}^T 1) \big) \cong \CH (G,I) \\
T_w \theta_\lambda & \longmapsto & T_w \theta_\lambda 
\end{array} 
\end{equation}
where $w \in \bW (\bM,\bT)^\phi$ and $\lambda \in X_* (\bT)^\phi$. 

We need to show that the analogues of \eqref{eq:3.16} and \eqref{eq:3.17} for coherent sheaves hold true.
Let $\hP = \hM \hU_\bP$ be the standard parabolic subgroup of $\hG$ with Levi factor $\hM$. 
Consider the following Cartesian diagram
\begin{equation}\label{eq:3.18}
\begin{tikzcd}
\hB = (\hB \cap \hM) \ltimes \hU_\bP \arrow[]{d}{r_P |_B} \arrow[]{rr}{\pi_{B \cap M,U_P}} && 
\hP = \hM \ltimes \hU_\bP \arrow[]{d}{r_P} \\
\hB \cap \hM \arrow[]{rr}{\pi_{B \cap M}} && \hM
\end{tikzcd} ,
\end{equation}
where $\pi_{B \cap M,U_P}$ means $\pi_{B \cap M}$ on $\hB \cap \hM$ and the identity on $\hU_\bP$. 
There is a similar Cartesian diagram for the stacks $X_{\LB}, X_{\LP}, X_{\LM}$ and $X_{{}^L \mathbf{B \cap M}}$.
This guarantees that 
\begin{equation}\label{eq:3.19}
r_P^* (\pi_{B \cap M})_* = (\pi_{B \cap M,U_P})_* (r_P |_B)^* ,
\end{equation}
as functors between (ind-)coherent sheaves on these stacks. Using this, we compute
\begin{equation}\label{eq:3.20}
\! \pi_{P*} r_P^* \CS^1_M = \pi_{P*} r_P^* (\pi_{B \cap M})_* r_{B \cap M}^* \CO_{X^1_\LT} \! =
\pi_{P*} (\pi_{B \cap M,U_P})_* (r_P |_B)^* r_{B \cap M}^* \CO_{X^1_\LT} \! =
\pi_{B*} r_B^* \CO_{X^1_\LT} \! = \CS^1_G . \hspace{-4mm} 
\end{equation}
For $\lambda \in X_* (\bT)^\phi$, by \cite[Remark 2.54]{Zhu-tame} and
\cite[Lemma 3.12]{BZCHN}, we have
\begin{equation}\label{eq:3.21}
\LocL^1_G (\theta_\lambda) = (\tilde{\pi}_{B,1}^\unip )_* \omega_{\widetilde \Loc^\unip_{\cG,F,1}} (\lambda) \quad \text{equals} \quad
(\pi_B^\unip )_* \omega_{\Loc^\unip_{{}^c \bB,F}} (\lambda)  .
\end{equation}
By \cite[Lemma 3.12]{BZCHN}, we can also express \eqref{eq:3.21} as
\begin{equation}\label{eq:3.22}
\LocL^1_G (\theta_\lambda) = (\pi_B)_* \CO_{\Loc^\unip_{{}^c \bB,F}} (\lambda) =
(\pi_B)_* r_B^* \CO_{\Loc^\unip_{{}^c \bT,F}} (\lambda) =
(\pi_B)_* r_B^* \omega_{\Loc^\unip_{{}^c \bT,F}} (\lambda) =
\pi_{B*} r_B^* \LocL^1_T (\theta_\lambda) .
\end{equation}
With \eqref{eq:3.22} at hand, the same computation as in \eqref{eq:3.22} shows that
\begin{equation}\label{eq:3.24}
\pi_{P*} r_P^* \LocL^1_M (\theta_\lambda) = \LocL^1_G (\theta_\lambda) .
\end{equation}
Now we take $w \in \bW (\bM,\bT)^\phi$ and use the conventions from \cite[\S 2.50--\S 2.55]{Zhu-tame}.
The canonical map
\[
r_P |_B : \widetilde{\Loc}^\unip_{\cG,F,w} \to \widetilde{\Loc}^\unip_{{}^c \bM,F,w} 
\]
arises from $r_P |_B : \Loc^\unip_{{}^c \bB,\breve{F}} \to \Loc^\unip_{{}^c \bB \cap \bM,
\breve{F}}$ by the same operations on both sides. Using \cite[Lemma 3.12]{BZCHN}, we have
\begin{equation}\label{eq:3.28}
(r_P |_B )^* \omega_{\widetilde{\Loc}^\unip_{{}^c \bM,F,w}} = 
(r_P |_B )^* \CO_{\widetilde{\Loc}^\unip_{{}^c \bM,F,w}} =
\CO_{\widetilde{\Loc}^\unip_{\cG,F,w}} =
\omega_{\widetilde{\Loc}^\unip_{\cG,F,w}} .
\end{equation}
By \eqref{eq:3.19} and \eqref{eq:3.28}, we have
\begin{align}\label{eq:3.23}
\begin{split}
\pi_{P*} r_P^* \LocL^1_M (T_w) = \pi_{P*} r_P^* (\pi_{B \cap M})_* \omega_{\widetilde{\Loc}^\unip_{{}^c \bM,F,w}} &
= 
\pi_{P*} (\pi_{B \cap M,U_P})_* (r_P |_B )^* \omega_{\widetilde{\Loc}^\unip_{{}^c \bM,F,w}} \\
&
=
\pi_{B*} \omega_{\widetilde{\Loc}^\unip_{\cG,F,w}} = \LocL^1_G (T_w) .
\end{split}
\end{align}
In the diagram \eqref{eqn:diagram-unipotent-LLC} we may restrict the right-hand side to the essential images
of $\LocL^1_M$ and $\LocL^1_G$, which by part (2)  are generated by $\CS^1_M$ and $\CS^1_G$, respectively.
By \eqref{eq:3.16} and \eqref{eq:3.20} and \cite[Theorem 3.8]{Keller2006DGCategories}, we can rewrite diagram \eqref{eqn:diagram-unipotent-LLC}--upon this restriction--as the derived module categories
of the following algebras:
\begin{equation}\label{eq:3.25}
\begin{tikzcd}
\End_M ( \ind_{I_M}^M 1) \arrow[]{r}{\LocL^1_M}\arrow[]{d}[swap]{i_{P}^G} & 
\End (\CS^1_M) \arrow[]{d}{(\pi_P)_* r_P^*}\\ 
\End_G (\ind_I^G 1) \arrow[]{r}{\LocL^1_G} & \End (\CS^1_G) 
\end{tikzcd}.
\end{equation}
By \eqref{eq:3.17}, \eqref{eq:3.24} and \eqref{eq:3.23}, the diagram \eqref{eq:3.25} commutes.  
Hence \eqref{eqn:diagram-unipotent-LLC} commutes as well. 
\end{proof}

\begin{remark}\label{rem:3.3}
Part (3) of Theorem \ref{thm:unipotent categorical} also holds in the generality 
of \eqref{eqn:LLC1M}, but the proof is more complicated. It will appear in 
forthcoming work of Xinwen Zhu and Xinyu Li.
\end{remark}
Using \cite[Theorem 1.7]{BZCHN}, we can describe the functors in Theorem 
\ref{thm:unipotent categorical} more explicitly. For $\pi \in \Rep (G)_{[T,1]}$,
we can view $\Hom_G (i_{B}^G \ind_{T^1}^T 1, \pi)$ as a right module
over any of the algebras in \eqref{eq:3.15}. Then
\[
\LocL^1_G (\pi) = \Hom_G (i_{B}^G \ind_{T^1}^T 1, \pi) 
\otimes_{\End(\CS^1_G)} \CS^1_G .
\]

\section{Spectral side reduction}\label{sec:spectral-side-reduction}

As the unipotent case of the categorical local Langlands correspondence is well-understood, we turn our attention to non-unipotent parameters.  We will see that there is in fact a very close geometric relationship between an arbitrary connected component of $X_{\LG}$ and the unipotent component of a related unramified group.  Our approach is very closely related to that of~\cite[\S 3]{moduli}, where a reduction to {\em tame} parameters is constructed over more general coefficient rings; the reduction to unipotent parameters that we construct here is possible only over fields of characteristic zero.  Indeed, our approach makes essential use of the fact that for Langlands parameters $(\rho,N)$ over a field of characteristic zero, the restriction of $\rho$ to $I_F$ is constant, up to ${\hat G}$-conjugacy, on connected components of the space of parameters; this is false over more general base rings.

\subsection{Geometry of the component \texorpdfstring{$X_{\LG}^{\phi}$}{XLGphi}}\label{subsec:geometry-XLGphi}

Fix a Langlands parameter $\phi = (\rho,N)$. We call $\phi$ \textit{Frobenius-semisimple} if $\rho (\Fr)$ is 
semisimple (which we do not require in the body of this paper). The L-parameters used in classical versions of
the local Langlands correspondence are Frobenius-semisimple.

Let $X_{\LG}^{\phi}$ denote the connected component of $X_{\LG}$ containing $\phi$. 
As $(\rho,N)$ and $(\rho,0)$ lie on the same component of $X_{\LG}$, we may assume without 
loss of generality that $N=0$.  We will refer to $\phi$ with $N=0$ as L-parameters \textit{with
trivial monodromy}. For such a $\phi$, and $w \in W_F$, we will often use the notation 
$\phi(w)$ to denote $\rho(w)$, and for a subgroup $I$ of $W_F$, we will often use the 
phrase ``restriction of $\phi$ to $I$'' to mean the restriction of $\rho$ to $I$. Let 
\begin{equation}
\nu:=\phi|_{I_F}
\end{equation}
denote the restriction of $\phi$ to the inertial subgroup $I_F$ of $W_F$, and let $C_{\hG}(\nu)$ denote the centralizer of $\nu(I_F)$ in ${\hG}$.  This centralizer is in general disconnected; its identity component is a reductive group that we shall denote by ${\hG}_{\nu}$.  (Note that although this notation suggests that $\hG_{\nu}$ is the complex dual group of a $p$-adic group $\bG_{\nu}$--and this will indeed turn out to be the case--for the moment we do not define such a group.)

For any Langlands parameter $(\rho',N')$ on $X_{\LG}^{\phi}$, the restriction $\rho'|_{I_F}$ of $\rho'$ is ${\hG}$-conjugate to $\nu$; moreover, if $g$ is an element of ${\hG}$ conjugating $\rho'|_{I_F}$ to $\nu$, then the conjugate parameter $(\rho')^g$ agrees with $\rho$ on $I_F$, and on a Frobenius element $\Fr$ satisfies $(\rho')^g(\Fr) = g'\rho(\Fr)$ for some $g' \in {\hG}_{\nu}$.  

There is an action of $W_F/I_F$ on ${\hG}_{\nu}$ where $w \in W_F$ acts via conjugation by $\rho(w)$.  As in the proof of~\cite[Theorem 3.4]{moduli}, we may, upon replacing $\phi$ by another Langlands parameter on the same component of $X_{\LG}$, assume that this action has finite order and preserves a Borel pair of ${\hG}_{\nu}$.  Let $\LG_{\nu}$ denote the semidirect product $\hG_{\nu} \rtimes W_F$.  In the cases we consider below, this will turn out to be the $L$-group of a quasi-split reductive group $\bG_{\nu}$ over $F$, but for the moment we treat $\LG_{\nu}$ as an abstract group. 

We have an L-homomorphism $\iota_{\phi}: \LG_{\nu} \rightarrow \LG$, defined by 
\begin{equation}\label{eqn:L-hom-iota-phi}
\iota_{\phi}(g \rtimes w) = g \rho(w) ,
\end{equation}
which sends Langlands parameters for $\LG_{\nu}$ to Langlands parameters for $\LG$.  Given a {\em unipotent} Langlands parameter $(\tau,N)$ for $G_{\nu}$, (that is, one for which the restriction of $\tau$ to $I_F$ is trivial) its image under $\iota_{\phi}$ is a Langlands parameter $(\rho',N')$ such that $\rho'|_{I_F}=\nu$.  

Let ${\tilde X}^1_{\LG_{\nu}}$ denote the moduli {\em scheme} of unramified Langlands parameters $(\tau,N)$ over $\CC$ for $G_{\nu}$, and let ${\tilde X}^{\nu}_{\LG}$ denote the moduli {\em scheme} of Langlands parameters $(\rho',N')$ over $\CC$ for $G$ such that $\rho'|_{I_F}=\nu$.  The $L$-homomorphism $\iota_{\phi}$ from \eqref{eqn:L-hom-iota-phi} then induces a closed immersion
\begin{equation}\label{eqn:iota-phi-Xtilde}
\tilde{\iota}_{\phi}:{\tilde X}^1_{\LG_{\nu}} \rightarrow {\tilde X}^{\nu}_{\LG}.
\end{equation}
On the other hand, we have a map:
\begin{equation}\label{eqn:map-tildeXnu-TnuFr}
{\tilde X}^{\nu}_{\LG} \rightarrow T_{\widehat G}(\nu,\nu^{\Fr}),
\end{equation}
where $\nu^{\Fr}: I_F \rightarrow \LG$ is defined by $\nu^{\Fr}(w) = \nu(\Fr w \Fr^{-1})^{\Fr}$, and $T_{\widehat G}(\nu,\nu^{\Fr})$ is the subset of ${\widehat G}$ consisting of elements
that conjugate $\nu$ to $\nu^{\Fr}$.  This map takes a parameter $(\rho',N')$ to the element $g$ of ${\widehat G}$ such that $\rho'(\Fr) = g \rtimes \Fr$. 
Consider the following map induced from \eqref{eqn:map-tildeXnu-TnuFr}:
\begin{equation}\label{eqn:map-tildeXnu-pi0TnuFr}
{\tilde X}^{\nu}_{\LG} \rightarrow \pi_0 \big( T_{\hat G}(\nu,\nu^{\Fr}) \big).
\end{equation}
The fiber of this map \eqref{eqn:map-tildeXnu-pi0TnuFr} containing the parameter $\phi$ is a connected component of ${\tilde X}^{\nu}_{\LG}$; this connected component is on the one hand equal to the image of $\tilde{\iota}_{\phi}$ under \eqref{eqn:iota-phi-Xtilde}, and on the other hand is the preimage of the component $X^{\phi}_{\LG}$
under the natural map:
\begin{equation}\label{eqn:natural-map-tildeXnu-to-X}
{\tilde X}^{\nu}_{\LG} \rightarrow X_{\LG}.
\end{equation}
We denote this preimage by ${\tilde X}^{\nu,\phi}_{\LG}$; it is isomorphic to ${\tilde X}^1_{\LG_{\nu}}.$

Let $x$ and $y$ be two arbitrary points of ${\tilde X}^{\nu}_{\LG}$; their images, under the natural map \eqref{eqn:natural-map-tildeXnu-to-X}, in $X_{\LG}$ agree if and only if their corresponding Langlands parameters $(\rho_x,N_x)$ and $(\rho_y,N_y)$ are ${\hG}$-conjugate; since the restrictions of $\rho_x$ and $\rho_y$ to $I_F$ are both equal to $\nu$, an element of $\hG$ that conjugates one to the other must centralize $\nu$.  The map \eqref{eqn:natural-map-tildeXnu-to-X} from ${\tilde X}^{\nu}_{\LG}$ to $X_{\LG}$ thus identifies the quotient ${\tilde X}^{\nu}_{\LG}/C_{\hG}(\nu)$ with a union of connected components of $X_{\LG}$ (more precisely, with the union of those components on which the restrictions to $I_F$ of the Langlands parameters are conjugate to $\nu$).

Note that ${\tilde X}^{\nu,\phi}_{\LG}$ is not stable under the action of $C_{\hG}(\nu)$ on ${\tilde X}^{\nu}_{\LG}$; indeed, the component group $\pi_0(C_{\hG}(\nu))$ permutes the connected components of the scheme ${\tilde X}^{\nu}_{\LG}$.  An element of $\pi_0(C_{\hG}(\nu))$ preserves the component ${\tilde X}^{\nu,\phi}_{\LG}$ of ${\tilde X}^{\nu}_{\LG}$ if and only if it is fixed under the conjugation action of $\rho(\Fr)$ on $C_{\hG}(\nu)$.  Let $\pi_{0,G,\phi}$ denote the subgroup of $\pi_0(C_{\hG}(\nu))$ fixed under this action, and let $C_{\hG}(\nu)_{\phi}$ be the preimage of $\pi_{0,G,\phi}$ in $C_{\hG}(\nu)$.  We then have an isomorphism:
\begin{equation}\label{eqn:isom-Xtilde}
{\tilde X}^{\nu,\phi}_{\LG}/C_{\hG}(\nu)_{\phi} \cong X_{\LG}^{\phi}.
\end{equation}

Note that the image of $\LG_{\nu}$ under $\iota_{\phi}$ is stable under the conjugation action of $C_{\hG}(\nu)_{\phi}$; in this way we obtain an action of $C_{\hG}(\nu)_{\phi}$ on $\LG_{\nu}$, and hence also on $X^1_{\LG_{\nu}}$.  The latter action factors through the quotient $\pi_{0,G,\phi}$ of $C_{\hG}(\nu)_{\phi}$.

Combining \eqref{eqn:isom-Xtilde} with the identification $\tilde{\iota}_{\phi}: {\tilde X}^1_{\LG_{\nu}}\cong {\tilde X}^{\nu,\phi}_{\LG}$ from \eqref{eqn:iota-phi-Xtilde}, we  then have the following:

\begin{thm} \label{thm:moduli reduction}
The isomorphism $\tilde{\iota}_{\phi}: {\tilde X}^1_{\LG_{\nu}} \xrightarrow{\sim} {\tilde X}^{\nu,\phi}_{\LG}$ descends to an isomorphism:
\begin{equation}\label{eqn:moduli-reduction-thm}
X^1_{\LG_{\nu}}/\pi_{0,G,\phi} \cong X^{\phi}_{\LG}.
\end{equation}
that takes the trivial $L$-parameter for $\LG_{\nu}$ to the parameter $\phi$.
\end{thm}
\begin{proof}
This follows by identifying $X^1_{\LG_{\nu}}$ with the quotient of ${\tilde X}^1_{\LG_\nu}$ by ${\hG}_\nu$, and by identifying $X^{\phi}_{\LG}$ with the quotient of ${\tilde X}^{\nu,\phi}_{\LG}$ by $C_{\hG}(\nu)_{\phi}$ as in \eqref{eqn:isom-Xtilde}.
\end{proof}

\subsection{Centralizers of L-parameters }\label{subsec:centralizers-defn-Gnu}

Fix $\phi$ as in \S \ref{subsec:geometry-XLGphi}.  We will be interested in understanding the Levi subgroups of $\LG$ 
through which $\phi$ factors.  (Here we are using the term ``Levi subgroup'' in the sense of Levi subgroups of
disconnected reductive groups; c.f.~\cite[\S 3]{Borel}; in particular, a Levi subgroup $\LM$ of $\LG$ factors as 
a semidirect product ${\hM} \rtimes W_F$ for a unique Levi subgroup ${\hM}$ of ${\hG}$.) Recall from 
\S \ref{subsec:geometry-XLGphi} that we denote ${\hG}_{\nu}:=C_{\hG}(\nu(I_F))^{\circ}$. 
Likewise we can define $\hM_\nu$. 

Recall that a Langlands parameter in $X_\LG$ is called \emph{discrete} if it does not factor through 
any proper Levi subgroup of $\LG$.

\begin{lemma} \label{lemma:irreducible}
Suppose that $\phi$ is discrete.
\begin{enumerate}
    \item ${\hG}_{\nu}$ is a torus.
    \item The map $\big( Z({\hG})^{I_F} \big)^\circ \rightarrow {\hG}_{\nu}$ induces a surjection on $\phi(\Fr)$-coinvariants.
    \item Every parameter $\phi'$ in $X_{\LG}^{\phi}$ is discrete.
    \item Every parameter $\phi' = (\rho',N')$ on $X_{\LG}^{\phi}$ satisfies $N' = 0$.
\end{enumerate}
\end{lemma}
\begin{proof}
We first prove (1). 
Recall that, without loss of generality, we have chosen $\phi$ so that conjugation by $\phi(\Fr)$ preserves a Borel pair of ${\hG}_{\nu}$; in particular, it preserves a maximal torus ${\hT}_{\nu}$ of ${\hG}_{\nu}$. We must show that in fact ${\hT}_{\nu} = {\hG}_{\nu}$.  Suppose otherwise, then there is a non-empty orbit of $\Fr$ on the set of simple roots for ${\hG}_{\nu}$ and a proper, $\phi(\Fr)$-stable Levi subgroup ${\hM}_{\nu}$ of ${\hG}_{\nu}$.  Let $Z$ be the center of this Levi subgroup ${\hM}_{\nu}$, and let ${\hM}$ be the Levi subgroup $C_{\hG}(Z)$. The union 
\[
\LM := \bigcup\nolimits_{w \in W_F} {\hM} \phi(w) 
\]
is then a proper Levi subgroup of $\LG$ through which $\phi$ factors, contradicting our starting assumption on $\phi$. Thus $\LG_{\nu}$ must be a torus.

Claim (2) is \cite[(6.6)]{moduli}, and is proven \textit{loc.cit}.

We now prove (3). Suppose $\phi' = (\rho',N')$ is another Langlands parameter on $X_{\LG}^{\phi}$.  If $\phi'$ factors through a proper Levi subgroup, then so does the parameter $(\rho',0)$; thus, without loss of generality, we may assume that $N'=0$. Upon conjugating $\phi'$ by an element of ${\hG}$, we may assume that $\rho'$ agrees with $\rho$ on $I_F$.  By claim (2), after further conjugation by an element of ${\hG}$ centralizing $\nu$, we can assume that $\phi'(\Fr)$ and $\phi(\Fr)$ differ by an element of $Z(\hG)^{I_F}$.  But such an element is contained in every Levi subgroup of $\LG$, so if $\phi'$ factors through a proper Levi subgroup then so does $\phi$.

Claim (4) follows from the isomorphism $X^1_{\LG_{\nu}}/\pi_{0,G,\phi} \cong X^{\phi}_{\LG}$ given in 
\eqref{eqn:moduli-reduction-thm} from Theorem \ref{thm:moduli reduction}, noting that by (1) no Langlands parameter 
for $\LG_{\nu}$ has nonzero monodromy.
\end{proof}
\begin{definition}
We will henceforth refer to a discrete L-parameter for $\LG$ with trivial monodromy as 
a {\em supercuspidal} L-parameter.
\end{definition}
Now let $\LM$ be a standard Levi subgroup of $\LG$. Let $\phi$ be a supercuspidal L-parameter for $\LM$, and let $\nu$ still denote its restriction to the inertia subgroup.  Then $\hM_\nu$ is a torus by Lemma \ref{lemma:irreducible}.  Let $\LP$ be the standard parabolic subgroup of $\LG$ with Levi subgroup $\LM$. Let ${\hP}$ and ${\hM}$ be the identity components of $\LP$ and $\LM$, respectively. We first prove the following:
\begin{lemma} \label{lemma:levis}
Let $\LL$ be a Levi subgroup of $\LG$ containing $\LM$. Let $\LQ$ be any parabolic subgroup of $\LG$ with Levi $\LL$; let ${\hL}$ and ${\hQ}$ denote their identity components.  Then the intersections ${\hQ} \cap \hG_{\nu}$ and ${\hL} \cap \hG_{\nu}$ are a parabolic subgroup and a Levi subgroup of $\hG_{\nu}$, respectively. In particular, both are connected, and thus equal to ${\hQ}_{\nu}$ and ${\hL}_{\nu}$, respectively.
\end{lemma}
\begin{proof}
Since ${\hQ}$ is a parabolic subgroup of ${\hG}$ with Levi ${\hL}$, we can find a central element $z$ of ${\hL}$ such that ${\hL}$ is the centralizer of $z$ in ${\hG}$ and ${\hQ}$ is the subgroup of ${\hG}$ consisting of all elements $g\in \hG$ such that $z^m g z^{-m}$ approaches a well-defined limit as $m$ goes to infinity.  Since $\phi$ factors through ${\hL}$, any conjugate of $z$ by $\phi(w)$, as $w$ ranges over $W_F$, will have the same property; upon replacing $z$ with a finite product of such conjugates, we may assume without loss of generality that $z$ lies in ${\hL}_{\nu}$ and is stable under the action of $W_F/I_F$ on ${\hL}_{\nu}$. Then ${\hQ} \cap \hG_{\nu}$ consists of the elements of ${\hG}_{\nu}$ such that $z^m g z^{-m}$ approaches a well-defined limit as $m$ goes to infinity; this is a parabolic subgroup as desired.  The elements that arise as limits are precisely those in ${\hL} \cap \hG_{\nu}$, so it is a Levi subgroup of ${\hQ} \cap \hG_{\nu}$.
\end{proof}

From the above Lemma \ref{lemma:levis} we immediately have the following:

\begin{cor} \label{cor:pinning}
The pair $({\hM}_{\nu},{\hP}_{\nu})$ is a Borel pair for ${\hG}_{\nu}$, stable under the action of $\rho (\Fr)$.  Moreover, there exists an element $m$ of ${\hM}_{\nu}$, and a pinning $({\hM}_{\nu}, {\hP}_{\nu}, \{\mu_{\alpha}\})$ of ${\hG}_{\nu}$ extending the pair $({\hM}_{\nu},{\hP}_{\nu})$, that is preserved by conjugation by $m \rho (\Fr)$.
\end{cor}
\begin{proof}
The only new claim here is the existence of a pinning preserved by $m \rho (\Fr)$ for some $m\in \hM_\nu$.  Fix a pinning $({\hM}_{\nu}, {\hP}_{\nu}, \{\mu_{\alpha}\})$ of ${\hG}_{\nu}$ extending the Borel pair $({\hM}_{\nu}, {\hP}_{\nu})$.  Then $\rho (\Fr)$ takes this pinning to a pinning of the form $({\hM}_{\nu}, {\hP}_{\nu}, \{\mu'_{\alpha}\})$. Here $\{\mu_{\alpha}\}$ and $\{\mu'_{\alpha}\}$ are two collections of pinning maps (into root subgroups) indexed by a basis of simple roots of $\Phi(\hG_\nu,\hM_\nu)$ determined by $\hP_\nu$. Since any two pinnings extending the Borel pair $({\hM}_{\nu},\hP_{\nu})$ are conjugate by a $\CC$-point of $\hM_{\nu}$, we can find an $m$ that conjugates $\{\mu'_{\alpha}\}$ to $\{\mu_{\alpha}\}$. Then $m \rho (\Fr)$ conjugates $\{\mu_{\alpha}\}$ to $\{\mu_{\alpha}\}$ as desired.
\end{proof}

\begin{cor}\label{coro:pi0-injection-Q-to-G}
The map $\pi_0(C_{\hQ}(\nu)) \rightarrow \pi_0(C_{\hG}(\nu))$ is injective. 
\end{cor}
\begin{proof}
The kernel of this map is the component group of $C_{\hQ}(\nu) \cap {\hG}$, but this space is equal to ${\hQ} \cap C_{\hG}(\nu)$ and is thus connected by Lemma~\ref{lemma:levis}.
\end{proof}

\begin{cor}\label{coro:pi0-isom-Q-to-L}
The map $\pi_0(C_{\hQ}(\nu)) \rightarrow \pi_0(C_{\hL}(\nu))$ is an isomorphism.
\end{cor}
\begin{proof}
The kernel of the map $C_{\hQ}(\nu) \rightarrow \pi_0(C_{\hL}(\nu))$ is a subgroup of the unipotent radical of ${\hQ}$; since any subgroup of a unipotent group is connected, the induced map on component groups is injective. On the other hand, the inclusion of ${\hL}$ in ${\hQ}$ induces a section of this map on component groups, proving surjectivity.
\end{proof}

In light of Corollary~\ref{cor:pinning}, we shall henceforth assume that the Langlands parameter $\phi$ is chosen so that $\rho (\Fr)$ preserves a pinning of ${\hG}_{\nu}$; note that for any supercuspidal L-parameter $W_F \rightarrow \LM$, there will be such a $\phi$ on its connected component of $X_{\LM}$.  In particular, the automorphism of $\hG_{\nu}$ is induced by an automorphism of the root datum $(X^*_{\nu},\Sigma_{\nu},(X_*)_{\nu},\Sigma^{\vee}_{\nu})$ associated to $\hG_{\nu}$; this gives an action of $W_F/I_F$ on this root datum.

Let $(\bG_{\nu})_{\overline{F}}$ denote the dual group of $\hG_{\nu}$, considered as an algebraic group over $\overline{F}$, i.e.~$(\bG_{\nu})_{\overline{F}}$ is the reductive group over $\overline{F}$ associated to the root datum $((X_*)_{\nu}, \Sigma^{\vee}_{\nu}, X^*_{\nu}, \Sigma_{\nu})$ dual to that of $\hG_{\nu}$.  The action of $W_F/I_F$ on this root datum described in the previous paragraph induces an action of $W_F$ on $(\bG_{\nu})_{\overline{F}}$; this action preserves the Borel pair $((\bM_{\nu})_{\overline{F}},(\bP_{\nu})_{\overline{F}})$ of $(\bG_{\nu})_{\overline{F}}$
dual to the Borel pair $(\hM_{\nu}, \hP_{\nu})$ of $\hG_{\nu}$.

The action of $W_F/I_F$ on $(\bG_{\nu})_{\overline{F}}$ allows us to descend this group to a reductive group 
$\bG_{\nu}$ over $F$; similarly $(\bM_{\nu})_{\overline{F}}$ and $(\bG_{\nu})_{\overline{F}}$ descend to a maximal 
torus $\bM_{\nu}$ and a Borel subgroup $\bP_{\nu}$ of $\bG_{\nu}$.  In particular, $\bG_{\nu}$ is quasi-split.  
Moreover, as the action of $W_F$ on $(\bG_{\nu})_{\overline{F}}$ factors through $W_F / I_F$, the group 
$\bG_{\nu}$ splits over an unramified extension of $F$.

By construction, the L-group of $\bG_{\nu}$ is precisely $\LG_{\nu}$.  More generally, for any parabolic subgroup $\LQ$ of $\LG$ containing $\LP$, with Levi subgroup $\LL$, the group $\hQ_{\nu}$ is dual to a parabolic subgroup $\bQ_{\nu}$ of $\bG_{\nu}$ with Levi subgroup $\bL_{\nu}$ dual to $\hL_{\nu}$.  In particular ${\LL}_{\nu}$ and $\LQ_{\nu}$ will be the $L$-groups of the quasi-split algebraic groups $\bL_{\nu}$, and $\bQ_{\nu}$ over $F$.

\subsection{The coherent Springer sheaf} \label{subsec:sheaf}
We are now in a position to define the fundamental object of study on the spectral side, i.e. the 
\textit{coherent Springer sheaf} $S_G^{\phi}$. This is an analogue of the usual coherent Springer sheaf 
$\CS^1_G$ in \eqref{eqn:usual-coherent-Spring-sheaf}, supported now on $X_{\LG}^{\phi}$, instead of $X_{\LG}^1$. 

\begin{definition}\label{defn:phi-Springer-sheaf}
The {\em coherent Springer sheaf} is defined as $\CS^{\phi}_G:=(\pi_P)_* \CO_{X_{\LP}^{\phi}}$. 
\end{definition}
\noindent Theorem~\ref{thm:moduli reduction} then allows us to relate the endomorphism algebra of $\CS^{\phi}_G$ to that of $\CS^1_{G_{\nu}}.$ Indeed, let $\pi_{0,M,\phi}$ and $\pi_{0,G,\phi}$ be the $\phi(\Fr)$-fixed subgroups of $\pi_0(C_{\hM}(\nu))$ and $\pi_0(C_{\hG}(\nu))$, respectively. By Corollaries \ref{coro:pi0-injection-Q-to-G} and \ref{coro:pi0-isom-Q-to-L}, we have an injection $\pi_{0,M,\phi}\hookrightarrow\pi_{0,G,\phi}$, whose cokernel we denote by ${\widehat R}_{G,\phi}$, thus we have an exact sequence:
\begin{equation}\label{eqn:pi-0-G-SES}
0 \rightarrow \pi_{0,M,\phi} \rightarrow \pi_{0,G,\phi} \rightarrow {\widehat{R}}_{G,\phi} \rightarrow 0.
\end{equation}
We then have the following: 

\begin{thm} \label{thm:Springer reduction}
There is an isomorphism 
\begin{equation}\label{eqn:thm-Springer-reduction}
\End(\CS^{\phi}_G) \cong \End(\CS^1_{G_{\nu}})^{\pi_{0,M,\phi}} \rtimes \CC[{\widehat R}_{G,\phi}],
\end{equation}
compatible with the action of $\pi_{0,G,\phi}$ on $\End(\CS^1_{G_{\nu}})$ induced by 
$\tilde{\iota}_\phi$ Moreover, the sheaf $\CO_{X^{\phi}_{\LG}}$ is a direct summand of 
$\CS^{\phi}_G$, and in particular lies in the full subcategory of $\IndCoh(X^{\phi}_{\LG})$
generated by $\CS^{\phi}_G$. 

Furthermore, there is an isomorphism:
\begin{equation}
\Hom(\CS^{\phi}_G, \CO_{X^{\phi}_{\LG}}) \cong \Hom(\CS^1_{G_{\nu}}, \CO_{X^1_{\LG_{\nu}}})^{\pi_{0,M,\phi}}
\end{equation}
compatible with the actions of $\End(\CS^1_{G_{\nu}})^{\pi_{0,M,\phi}}$ on the source and the target.
\end{thm}
\begin{proof}
We have a commutative diagram:
\begin{equation}\label{eqn:comm-diag-spectral-side}
\begin{tikzcd}
X^1_{\LP_{\nu}} \arrow[]{r}{}\arrow[]{d}{\pi_{P_\nu}}
&  X^1_{\LP_{\nu}}/\pi_{0,M,\phi}\arrow[]{r}{\cong}[swap]{\iota_{\phi}}\arrow[]{d}{\pi_{P_\nu}/\pi_{0,M,\phi}}
& X^{\phi}_{\LP} \\
X^1_{\LG_{\nu}} \arrow[]{r}{\mathrm{pr}}
& X^1_{\LG_{\nu}}/\pi_{0,M,\phi}\arrow[]{d}{}&\\
&   X^1_{\LG_{\nu}}/\pi_{0,G,\phi}\arrow[]{r}{\cong}
& X^{\phi}_{\LG}
\end{tikzcd}
\end{equation}
in which the square is Cartesian.  The Springer sheaf $\CS^1_{G_{\nu}}$ is the pushforward of the structure sheaf of $X^1_{\LP_{\nu}}$ along the left-hand vertical map, whereas $\CS^{\phi}_G$ is the pushforward of $\iota_\phi^* \CO_{X^\phi_\LP}$ to $X^\phi_\LG$ along the composition of the right-hand vertical maps.

Let $\CT$ be the pushforward of $\iota_{\phi}^*\CO_{X^{\phi}_{\LP}}$ along the upper right-hand vertical map $\pi_{P_\nu}/\pi_{0,M,\phi}$.  Then $\CS^1_{G_{\nu}}$ is the pullback of $\CT$ along the lower horizontal map $\mathrm{pr}$, thus we have an isomorphism 
\begin{equation}\label{eqn:isom-End-CT-End-SGnu1-inv}
\End(\CT) \cong \End(\CS^1_{G_{\nu}})^{\pi_{0,M,\phi}}.
\end{equation}
On the other hand, $\CS^{\phi}_G$ is the pushforward of $\CT$ along the lower right-hand vertical map, which, by \eqref{eqn:pi-0-G-SES}, is a quotient by the group ${\widehat R}_{G,\phi}$. Thus we have an isomorphism
$\End(\CS^{\phi}_G) \cong \End(\CT) \rtimes \CC[{\widehat R}_{G,\phi}]$ and the first claim follows.

Now consider $\Hom(\CS^{\phi}_G,  \CO_{X^{\phi}_{\LG}})$.  The sheaf $\CS^{\phi}_G$ is the pushforward of $\CT$ along the bottom vertical map, which is finite {\'e}tale.  Thus this $\Hom$-space is isomorphic to $\Hom(\CT, \CO_{X^1_{\LG_{\nu}/\pi_{0,M,\phi}}})$, by $!$-adjunction. Again since the square in \eqref{eqn:comm-diag-spectral-side} is Cartesian, we see that this is identified with $\Hom(\CS^1_{G_{\nu}}, \CO_{X^1_{\LG_{\nu}}})^{\pi_{0,M,\phi}}$ as claimed.

It remains to show that $\CO_{X^{\phi}_{\LG}}$ is a direct summand of $\CS^{\phi}_G$.~Firstly, it follows easily from Theorem \ref{thm:unipotent categorical} that $\CO_{X^1_{\LG_{\nu}}}$ is a direct summand of $\CS^1_{G_{\nu}}$: under the correspondence these sheaves correspond, respectively, to the space $\cW_{G_{\nu},[M_{\nu},1]}$ of compact Whittaker functions on the principal block, and the parabolic induction $i_{P_{\nu}}^{G_{\nu}} \chi^{\un}$.  Both of these can be interpreted as compact inductions: the former is isomorphic to $\ind_{K_{\nu}}^{G_{\nu}} \St$ where $K_{\nu}$ is a certain hyperspecial subgroup $K_{\nu}$ of $G_{\nu}$ and $\St$ denotes the inflation to $K_{\nu}$ of the Steinberg representation; the latter is isomorphic to $\ind_{I_{\nu}}^{G_{\nu}} 1$, where $I_{\nu}$ is an Iwahori subgroup of $G_{\nu}$ that we may take to be contained in $K_{\nu}$.  It is then clear that the former is a direct summand of the latter, so the unipotent local Langlands correspondence tells us that $\CO_{X^1_{\LG_{\nu}}}$ is a direct summand of $\CS^1_{G_{\nu}}$.

There is a natural way to realize this splitting on the spectral side.  We have a natural isomorphism of $\pi_{P_{\nu}}^! \CO_{X^1_{\LG_{\nu}}}$ with $\CO_{X^1_{\LP_{\nu}}}$, and the counit of the adjunction between $\pi_{P_{\nu}}^!$ and $(\pi_{P_{\nu}})_*$ thus gives a map from $\CS^1_{G_{\nu}}$ to $\CO_{X^1_{\LP_{\nu}}}$.  This map must be surjective: any other map from $\CS^1_{G_{\nu}}$ to $\CO_{X^1_{\LP_{\nu}}}$ is given by precomposing this counit with an endomorphism of $\CS^1_{G_{\nu}}$, and so the counit must have maximal image among all such maps.  Since $X^1_{\LP_{\nu}}$ is the quotient of an affine scheme by a reductive group, any surjection to the structure sheaf is necessarily split.

Now consider the structure sheaf $\CO$ on the quotient $X^1_{\LG_\nu}/\pi_{0,M,\phi}$.  Since the square in diagram \eqref{eqn:comm-diag-spectral-side} is Cartesian, the shriek pullback of $\CO$ along the upper right vertical map is the structure sheaf on $X^1_{\LP_\nu}/\pi_{0,M,\phi}$.  Thus the counit of the $!$-adjunction gives a map from $\CT$ to $\CO$, and the pullback of this counit along the lower horizontal map is the split surjection of $\CS^1_{G_{\nu}}$ onto $\CO_{X^1_{\LG_{\nu}}}$ described above.  Thus $\CT$ surjects onto $\CO$, and this surjection is also necessarily split.

Thus $\CO$ is a direct summand of $\CT$.  Pushing forward along the lower right-hand vertical map (which is finite {\'e}tale) we see that $\CO_{X_{\LG}^{\phi}}$ is a direct summand of $\CS^{\phi}_G$, as desired.
\end{proof}

Our next step is an analysis of the action of $\pi_{0,M,\phi}$ on $\End(\CS^1_{G_{\nu}})$.  Recall that this endomorphism algebra is an Iwahori--Hecke algebra.  It has a natural subalgebra $\End(\CS^1_{M_{\nu}})$ that is a Laurent polynomial ring. Since ${\hM}_{\nu}$ is a torus by Lemma \ref{lemma:irreducible}, $\CS^1_{M_{\nu}}$ is simply the structure sheaf on $X^1_{\LM_{\nu}}$. The algebra $\End(\CS^1_{M_{\nu}})$ acts on the Springer sheaf $\CS^1_{G_\nu}$ by endomorphisms, via the isomorphism
\begin{equation}\label{eq:4.SGM}
\CS^1_{G_{\nu}} \cong (\pi_{P_{\nu}})_* r_{P_{\nu}}^* \CS^1_{M_{\nu}}
\end{equation}
and functoriality of $(\pi_{P_{\nu}})_*$ and $r_{P_{\nu}}^*$.

The action of $\pi_{0,M,\phi}$ on $\End(\CS^1_{M_{\nu}})$ is easy to describe: an endomorphism of this sheaf is a function on $X^1_{\LM_{\nu}}$, so it suffices to describe the action of $\pi_{0,M,\phi}$ on the underlying coarse moduli scheme of $X^1_{\LM_{\nu}}$, which may be identified (via the map that evaluates a parameter at $\Fr$) with the torus ${\hM}_{\nu}/(\Fr - 1){\hM}_{\nu}$ of $W_F/I_F$-coinvariants of ${\hM}_{\nu}$. We have a natural group homomorphism:
\begin{equation}\label{eqn:map-defining-KMphi}
\pi_{0,M,\phi} \rightarrow {\hM}_{\nu}/(\Fr - 1){\hM}_{\nu} =: \hM_{\nu,\Fr} .
\end{equation}
It takes an element $x$ of $\pi_{0,M,\phi}$, chooses a lift ${\tilde x}$ to $C_{\hM}(\nu)$ (note that such a lift is well-defined up to an element of ${\hM}_{\nu}$), and then sends $x$ to $(1 - \Fr){\tilde x} = \tilde x \Fr (\tilde x)^{-1}$.  
\begin{definition}\label{defn:KMphi}
Let $K_{M,\phi}$ denote the image of $\pi_{0,M,\phi}$ in ${\hM}_{\nu,\Fr}$ under the map \eqref{eqn:map-defining-KMphi}. 
\end{definition}

\begin{lemma}\label{lem:4.11}
The action of $\pi_{0,M,\phi}$ on $\End(\CS^1_{M_{\nu}})$ factors through $K_{M,\phi}$. 
\end{lemma}
\begin{proof}
Note that if $(\rho, N): W_F \rightarrow \LM_{\nu}$ is a parameter on $X^1_{\LM_{\nu}}$
(and thus $\rho|_{I_F}=1$), then the conjugate of $\rho$ by ${\tilde x}$ is the unique
unramified parameter $\rho'$ with $\rho'(\Fr) = (1 - \Fr){\tilde x}\rho(\Fr)$, so the
action of $x$ on the coarse moduli space of $X^1_{\LM_{\nu}}$ is via translation by 
$(1 - \Fr){\tilde x}$. Thus the claim follows. 
\end{proof}

Recall the sheaf $\CT$ defined above \eqref{eqn:isom-End-CT-End-SGnu1-inv}. 
\begin{lemma} \label{lem:K-action}
The action of $\pi_{0,M,\phi}$ on $\End(\CS^1_{G_{\nu}})$ factors through $K_{M,\phi}$. In particular, we have:
\begin{equation}\label{eqn:isom-EndT-EndSGnu1-KMinv}
    \End(\CT)\cong \End(\CS_{G_\nu}^1)^{K_{M,\phi}}.
\end{equation}
\end{lemma}
\begin{proof}
The maps $\pi_{P_\nu} : \hP_\nu \to \hG_\nu$ and $r_{P_\nu} : \hP_\nu \to \hM_\nu$ 
are equivariant for the conjugation action of $C_{\hG}(\nu)_\phi$. Therefore, by 
\eqref{eq:4.SGM}, the action of $\pi_{0,M,\phi}$ on $\CS^1_{G_\nu}$ is obtained via
the functoriality of $(\pi_{P_\nu})_*$ and $(r_{P_\nu})^*$ from the action
on $\CS^1_{M_\nu}$. In other words,
the actions of $\pi_{0,M,\phi}$ on $\CS^1_{M_\nu}$ and $\CS^1_{G_\nu}$ are related by 
parabolic induction in the geometric sense. It thus follows from Lemma \ref{lem:4.11} 
that both actions factor through $K_{M,\phi}$.
\end{proof}

Since $\hM_\nu$ is a torus (Lemma \ref{lemma:irreducible}), we may identify
$X^1_{\LM_\nu}$ with $\hM_{\nu,\Fr}$. Moreover, $\CS^1_{M_\nu} = \CO_{X^1_{\LM_\nu}}$ 
and there is a natural isomorphism
\begin{equation}\label{eq:4.EndSMnu}
\End (\CS^1_{M_\nu}) \cong \CC [X^1_{\LM_\nu}] = \CC [\hM_{\nu,\Fr}] .
\end{equation}
In the proof of Lemma \ref{lem:4.11} we 
saw that $\pi_{0,M,\phi}$ and $K_{M,\phi}$ act on $X^1_{\LM_\nu}$ via translations by the finite subgroup $K_{M,\phi}$ of $\hM_{\nu,\Fr}$. This also describes how
$K_{M,\phi}$ acts on \eqref{eq:4.EndSMnu}.

We now turn our attention from considering the action of $\pi_{0,M,\phi}$ to that of the quotient ${\widehat R}_{G,\phi}$ on $\End(\CS^1_{G_{\nu}})^{K_{M,\phi}}$.  Our approach here largely follows the proof of~\cite[Lemma 6.14]{moduli}, though our objectives are somewhat different.

Let ${\widetilde \Omega}_{\phi,G}$ denote the normalizer, in $C_{\hG}(\nu)$, of the maximal torus $\hM_{\nu}$ of $\hG_{\nu}$. Set $\Omega_{\phi,G} := {\widetilde \Omega}_{\phi,G}/\hM_{\nu}$.\footnote{This quotient is denoted by $\Omega_{\phi}$ in~\cite{moduli}, and the torus $\hM_{\nu}$ is denoted by $T_{\phi}$ in that context.}  Let ${\widetilde \Omega}_{\phi,G,\phi(\Fr)}$ be the subgroup of ${\widetilde \Omega}_{\phi,G}$ consisting of all $g \in {\widetilde \Omega}_{\phi,G}$ such that $(g \phi(\Fr) g)^{-1}$ lies in $\hM_{\nu}$, and set 
\[
\Omega_{\phi,G,\phi(\Fr)}:={\widetilde \Omega}_{\phi,G,\phi(\Fr)}/\hM_{\nu}.\footnote{The subgroup $\Omega_{\phi,G,\phi(\Fr)}$ of $\Omega_{\phi,G}$ is denoted by $\Omega_{\phi}^{\Ad \beta}$ in~\cite{moduli}.} 
\]
We write $\bW ({\hG} ,\hM) = N_{\hG} (\LM ) / \hM$. Let $\bW (\hM,\phi)$ be the stabilizer of 
$X^\phi_{\LM}$ in $\bW ({\hG}, \hM)$. It consists of the cosets $g \hM$ such that
$\Ad(g) \phi$ lies on the component $X^{\phi}_{\LM}$ of $X_{\LM}$. We record these groups for later use:
\begin{equation}\label{eqn:defn-W-Mhat-phi}
\bW(\hM,\phi) \subset \bW ({\hG} ,\hM) = N_{\hG} (\LM ) / \hM .
\end{equation}
The proof of \cite[Lemma 6.14]{moduli} (and particularly the final paragraph containing (6.7) \textit{loc.cit.}) shows that we have the following short exact sequence:\footnote{Note that in \cite{moduli} the group $\bW(\hM,\phi)$ is denoted by $W_{\varphi,{\mathcal M}}$, and the group $\pi_{0,M,\phi}$ is denoted by $K$.}
\begin{equation}0 \rightarrow \pi_{0,M,\phi} \rightarrow \Omega_{\phi,G,\phi(\Fr)} \rightarrow \bW(\hM,\phi) \rightarrow 0.
\end{equation}
Now consider the sequence of maps: 
\begin{equation}\label{eqn:SES-Omega-pi0G-R}
\Omega_{\phi,G,\phi(\Fr)} \rightarrow \pi_{0,G,\phi} \rightarrow {\widehat R}_{G,\phi}.
\end{equation}
The kernel of the composition \eqref{eqn:SES-Omega-pi0G-R} contains the subgroup $\pi_{0,M,\phi}$ of $\Omega_{\phi,G,\phi(\Fr)}$, and thus the composed map \eqref{eqn:SES-Omega-pi0G-R} factors through a map $\bW(\hM,\phi) \rightarrow {\widehat R}_{G,\phi}$, which fits into a commutative diagram:
\begin{equation}\label{eqn:big-diagram-pi0M-pi0G}
\begin{tikzcd}
0 \arrow[]{r}{}
& \pi_{0,M,\phi} \arrow[]{r}{}\arrow[]{d}{}
& \Omega_{\phi,G,\phi(\Fr)} \arrow[]{r}{}\arrow[]{d}{}
& \bW(\hM,\phi) \arrow[]{r}{}\arrow[]{d}{}
& 0\\
0 \arrow[]{r}{}
& \pi_{0,M,\phi}\arrow[]{r}{} 
& \pi_{0,G,\phi}\arrow[]{r}{} 
& {\widehat R}_{G,\phi}\arrow[]{r}{} 
& 0
\end{tikzcd}.
\end{equation}

\begin{lemma} \label{lemma:spectral Weyl sequence}
The maps $\Omega_{\phi,G,\phi(\Fr)} \rightarrow \pi_{0,G,\phi}$ and $\bW(\hM,\phi) \rightarrow {\widehat R}_{G,\phi}$ in diagram \eqref{eqn:big-diagram-pi0M-pi0G} are surjective.  In particular, we have an exact sequence:
\begin{equation}\label{eqn:exact-sequence-WGnuFr}
0 \rightarrow \bW(\hG_{\nu},\hM_\nu)^{\Fr} \rightarrow \bW(\hM,\phi) \rightarrow 
{\widehat R}_{G,\phi} \rightarrow 0,
\end{equation}
where $\bW(\hG_{\nu}, \hM_\nu)^{\Fr}$ denotes $\Fr$-invariant elements in $\bW(\hG_{\nu},\hM_\nu)$.  
\end{lemma}
\begin{proof}
Consider the intersection ${\widetilde \bW}(\hM,\phi) \cap C_{\hG}(\nu)$. We will show that this intersection maps surjectively onto $\pi_{0,G,\phi}$. 

Indeed, fix an element of $\pi_{0,G,\phi}$, and choose an element $g$ on the corresponding connected component of $C_{\hG}(\nu)$.  It suffices to show that we may choose $g$ inside ${\widetilde \bW}(\hM,\phi)$. Upon changing $g$ by an element of $\hG_{\nu}$, we may assume without loss of generality that $g$ normalizes the Borel pair $(\hM_{\nu},\hP_{\nu})$ of $\hG_{\nu}$.

We now show that such a $g$ lies in ${\widetilde \bW}(\hM,\phi)$, that is, that $g$ normalizes $\LM$ in $\LG$ and conjugates $\phi$ to a parameter $\phi^g$ on $X_{\LM}^{\phi}$.  Since $g$ centralizes $\nu$, the restriction of $\phi^g$ to $I_F$ is equal to $\nu$. Thus $\phi^g(\Fr)$ and $\phi(\Fr)$ differ by an element of $C_{\hG}(\nu)$.  Moreover, the connected component of $C_{\hG}(\nu)$ containing $g$ is fixed by $\phi(\Fr)$, since this component is fixed in $\pi_{0,G,\phi}$ from the beginning.  Thus $g \phi(\Fr) g^{-1} \phi^{-1}(\Fr)$ lies on the identity component of $C_{\hG}(\nu)$, so $\phi(\Fr)$ and $\phi^g(\Fr)$ lie on the same component of $C_{\hG}(\nu)$.

In particular, $\phi(\Fr)$ and $\phi^g(\Fr)$ differ by an element $m$ of $\hG_{\nu}$.  Moreover, since $g$ and $\phi(\Fr)$ both normalize $(\hM_{\nu},\hP_{\nu})$, so does $\phi^g(\Fr)$.  Thus $m$ normalizes this Borel pair, so $m$ lies in $\hM_{\nu}$.  Since $\phi$ and $\phi^g$ agree on $I_F$ and their values at $\Fr$ differ by an element of $\hM_{\nu}$, we know that $\phi^g$ lies in $X^{\phi}_{\LM}$. In particular, $\LM$ is the minimal Levi of $\LG$ containing $\phi^g$. Since it is also the minimal Levi containing $\phi$, we must have that $g$ normalizes $\LM$, thus $g$ lies in ${\widetilde \bW}(\hM,\phi)$ as claimed.

Now, given an element $r$ of ${\widehat R}_{G,\phi}$, one can choose a lift to $\pi_{0,G,\phi}$, and that lift is represented by an element ${\widetilde r}$ of ${\widetilde \bW}(\hM,\phi)$.  The image of ${\widetilde r}$ in $\bW(\hM,\phi)$ maps to $r$, proving surjectivity of the right-hand vertical map in diagram \eqref{eqn:big-diagram-pi0M-pi0G}; the surjectivity of the middle vertical map then follows by the five lemma.

The kernel of the middle vertical map in diagram \eqref{eqn:big-diagram-pi0M-pi0G} is the
quotient $({\widetilde \Omega}_{\phi,G,\phi(\Fr)} \cap \hG_{\nu})/\hM_{\nu}$.  This is
precisely the subgroup of the Weyl group $\bW(\hG_{\nu},\hM_\nu)$ fixed by $\phi(\Fr)$; 
it can equivalently be regarded as the quotient $N_{\hG_{\nu}}(\LM_{\nu})/{\hM_{\nu}}$.

Since the left-hand vertical map is the identity, the kernel of the middle vertical map is isomorphic to the kernel of the right-hand vertical map, yielding the desired exact sequence \eqref{eqn:exact-sequence-WGnuFr}. 
\end{proof}

It turns out that the extension \eqref{eqn:exact-sequence-WGnuFr} splits. To see this, we first need
the following lemma as preparation. 

\begin{lemma} \label{lemma:Weyl splitting}
Let $\LP'_{\nu}$ be a parabolic subgroup of $\LG_{\nu}$ with Levi $\LM_{\nu}$, and suppose that $\LP'_{\nu}$  
contains the element $1 \rtimes \Fr$ of $\LG_{\nu}$.  Then there exists a unique element of 
$\bW(\hG_{\nu}, \hM_\nu)^{\Fr}$ conjugating $\LP'_{\nu}$ to $\LP_{\nu}$.
\end{lemma}
\begin{proof}
The identity components $\hP_{\nu}$ and $\hP'_{\nu}$, of $\LP_{\nu}$ and $\LP'_{\nu}$, respectively, are Borel subgroups containing the maximal torus $\hM_{\nu}$, so there is a unique element $w$ of $\bW(\hG_{\nu},\hM_\nu)$ conjugating $\hP'_{\nu}$ to $\hP_{\nu}$.  Since $\LP'_{\nu}$ and $\LP_{\nu}$ are the normalizers in $\LG_{\nu}$ of $\hP'_{\nu}$ and $\hP_{\nu}$, respectively, $w$ also conjugates $\LP'_{\nu}$ to $\LP_{\nu}$.

Thus $w^{\Fr}$ conjugates $(\LP'_{\nu})^{\Fr}$ to $\LP_{\nu}^{\Fr}$.  But since both $\LP_{\nu}$ and $\LP'_{\nu}$ contain $1 \rtimes \Fr$, these parabolics are equal to $\LP'_{\nu}$ and $\LP_{\nu}$.  By the uniqueness of $w$, we must have $w = w^{\Fr}$.
\end{proof}

Now we construct the splitting of \eqref{eqn:exact-sequence-WGnuFr}. It depends on $\LP_\nu$, but $\LP_\nu$ is
determined by $\LP$, which is part of our data, thus this construction is canonical.

\begin{lemma}\label{lem:4.splittingR}
The parabolic subgroup $\LP_{\nu}$ of $\LG_\nu$ determines a splitting of 
\eqref{eqn:exact-sequence-WGnuFr}, which gives a canonical group isomorphism $\bW (\hM,\phi) \cong 
\bW (\hG_\nu, \hM_\nu)^\Fr \rtimes \widehat{R}_{G,\phi}$.
\end{lemma}
\begin{proof}
Given $w$ in $\bW(\hM,\phi)$, we may lift it to an element of $\Omega_{\phi,G,\phi(\Fr)}$. Such an element acts on $\LG_{\nu}$ by conjugation, preserving $\hM_{\nu}$ and thus $\LM_{\nu}$.  The ambiguity in such a lift is given by an element of $\pi_{0,M,\phi}$, which normalizes the maximal torus $\LM_{\nu}$ of $\LG_{\nu}$.  This gives an action of such $w$ on the set of parabolic subgroups of $\LG_{\nu}$ that contain $\LM_{\nu}$.

Given $w \in {\widehat R}_{G,\phi}$, by Lemma \ref{lemma:Weyl splitting} we know that $w$ admits a unique lift to an element of $\bW(\hM,\phi)$ that preserves the parabolic subgroup $\LP_{\nu}$ of $\LG_{\nu}$. This identifies ${\widehat R}_{G,\phi}$ with the subgroup of $\bW(\hM,\phi)$ preserving this parabolic, thus providing a splitting of the exact sequence \eqref{eqn:exact-sequence-WGnuFr}. 
\end{proof}

We henceforth regard ${\widehat R}_{G,\phi}$ as a subgroup of $\bW(\hM,\phi)$ via 
Lemma \ref{lem:4.splittingR}. Let $\Omega_{\phi,G,\phi (\Fr),P}$ be the stabilizer of $\LP_\nu = 
\LP \cap \LG_\nu$ in $\Omega_{\phi,G,\phi(\Fr)}$, or equivalently the preimage of ${\widehat R}_{G,\phi}$ in $\Omega_{\phi,G,\phi(\Fr)}$. By Lemmas 
\ref{lemma:spectral Weyl sequence} and \ref{lem:4.splittingR} and the diagram \eqref{eqn:big-diagram-pi0M-pi0G}, we 
see that \eqref{eqn:SES-Omega-pi0G-R} restricts to a group isomorphism
\[
\Omega_{\phi,G,\phi (\Fr),P} \cong \pi_{0,G,\phi} .
\]
Recall from Corollary \ref{cor:pinning} that $\phi$ (or more precisely its image)
preserves a pinning $(\hM_\nu, \hP_\nu, \{\mu_\alpha\} )$ of $\hG_\nu$. As in the 
proof of Corollary \ref{cor:pinning}, every representative in $C_{\hG}(\nu)$ of an
element of $\Omega_{\phi,G,\phi (\Fr),P}$ can be adjusted by an element of $\hM_\nu$,
such that it preserves this pinning. In other words, $\pi_{0,G,\phi}$ is isomorphic to
the subgroup of $\Omega_{\phi,G,\phi (\Fr)}$ that stabilizes $\LP_\nu$, and it can be
represented by elements of $C_{\hG} (\nu)$ that stabilize $(\hM_\nu, \hP_\nu, 
\{\mu_\alpha \})$. For $x \in \pi_{0,G,\phi}$, let $x_P \in C_{\hG} (\nu)$ be such a
representative -- it is unique up to $Z(\hG_\nu)$. 

Since $x_P$ and $\phi$ preserve the same pinning, 
\begin{equation}\label{eq:4.xP}
x_P \phi (\Fr) x_P^{-1} = x_P \phi (\Fr) x_P^{-1} \phi (\Fr)^{-1} \phi (\Fr) =
(1 - \Fr)(x_P) \phi (\Fr) 
\end{equation}
also preserves that pinning. This implies that
\begin{equation}\label{eq:K-central}
(1 -\Fr)(x_P) \text{ is central in } \hG_\nu \text{ for every } x \in \pi_{0,G,\phi} .
\end{equation}
We write $Z(\hG_\nu)_\Fr := Z(\hG_\nu) / (\Fr - 1)Z(\hG_\nu)$ and we define a map
\begin{equation}\label{eqn:defn-zeta-x}
\zeta : \pi_{0,G,\phi} \to Z(\hG_\nu)_\Fr \quad \text{by}\quad \zeta (x) = (1 - \Fr)(x_P) .
\end{equation}
By the uniqueness of $x_P Z(\hG_\nu)$, $\zeta$ is well-defined. It is not necessarily 
a group homomorphism, rather a 1-cocycle with respect to the action of $\pi_{0,\G,\phi}$
on $Z(\hG)_\Fr$ by $\Ad_{x_P}$.

We recall that $(Z(\hG)^{I_F})_\Fr$ acts naturally on $X_\LG$ as follows: for $\phi' = (\rho',N') \in X_\LG$ and 
$z \in (Z(\hG )^{I_F})_\Fr$, the L-parameter $z \phi' = (z \rho',N') \in X_\LG$ is defined by
\begin{equation}\label{eq:4.actionZ}
(z \rho') |_{I_F} := \rho' |_{I_F} \;\text{and}\; ( z \rho') (\Fr) := z \rho' (\Fr) .
\end{equation}

\begin{prop}\label{prop:4.actionXonSGnu}\ 
\begin{enumerate}
\item The map $\zeta$ defined in \eqref{eqn:defn-zeta-x} has image in $Z(\hG_\nu)_\Fr^\circ \cong \Xnr (G_\nu)$ 
and extends the map \eqref{eqn:map-defining-KMphi}.
\item An element $x \in \pi_{0,G,\phi}$ acts on $\CS^1_{G_\nu}$ and on $\End (\CS^1_{G_\nu})$ 
as the pinned group automorphism $\Ad_{x_P}$ of $G_\nu$ followed by the twist by 
$\zeta (x) \in \Xnr ( G_\nu )$.
\item The action of $\widehat{R}_{G,\phi} \cong \pi_{0,G,\phi} / \pi_{0,M,
\phi}$ on $\End (\CS^1_{G_\nu})^{\pi_{0,M,\phi}} = \End (\CS^1_{G_\nu})^{K_{M,\phi}}$ 
in Theorem \ref{thm:Springer reduction} can be described in the same way as part (2).
\end{enumerate}
\end{prop}
\begin{proof}
To show (2), we investigate the actions of $\pi_{0,G,\phi}$ on $X^\phi_{\LG}$ and $X^1_{\LG_\nu}$.
For $g \in C_{\hG}(\nu)$ we define a Langlands parameter $g \phi \in X_{\LG}$ by 
\[
(g \phi ) |_{I_F} = \phi |_{I_F} ,\; (g \phi)(\Fr) = g \phi (\Fr) .
\]
Every object of $X^\phi_{\LG}$ with trivial monodromy can be written as $g \phi$ for some 
$g \in \hG_\nu$. Using \eqref{eq:4.xP}, we compute:
\[
x \cdot g \phi = x_P g \phi x_P^{-1} = (x_P g x_P^{-1}) (x_P \phi x_P^{-1}) =
(x_P g x_P^{-1}) (1 - \Fr)(x_P) \phi =: \Ad_{x_P}(g) (1 - \Fr)(x_P) \phi \in X^\phi_{\LG} .
\]
Every object of $X^1_{\LG_\nu}$ with trivial monodromy can be expressed by a single element 
$g \in \hG_\nu$. The action of $\pi_{0,G,\phi}$ on $X^1_{\LG_\nu}$ is induced from the
action on $X^\phi_\LG$, via $\tilde{\iota}_\phi$:
\begin{align}\label{eq:4.xg}
\begin{split}
x \cdot g &= \tilde{\iota}_\phi^{-1} ( x \cdot \tilde{\iota}_\phi (g) ) = 
\tilde{\iota}_\phi^{-1} ( x \cdot g \phi ) 
= \tilde{\iota}_\phi^{-1} ( \Ad_{x_P}(g) (1-\Fr)(x_P) \phi ) \\
&= \Ad_{x_P}(g) (1-\Fr)(x_P) = (1-\Fr)(x_P) \Ad_{x_P}(g) \in X^1_{\LG_\nu}.
\end{split}
\end{align}
The action of $\pi_{0,G,\phi}$ on $X^1_{\LG_\nu}$ induces, by pushforward of sheaves,
an action on $\CS^1_{G_\nu}$, namely the action in the statement of (2).\\
To see (3): This follows directly from (2).\\
To show (1): It follows from \eqref{eq:4.xg} that $x \cdot 1 = (1 - \Fr)(x_P)$ lies in $X^1_{\LG_\nu}$.
Hence the image of $(1 - \Fr)(x_P)$ in $Z(\hG_\nu)_\Fr$ lies in  $Z(\hG_\nu)_\Fr^\circ$. 
Via the LLC, $Z(\hG_\nu)_\Fr^\circ$ is isomorphic to the group of unramified characters 
of $G_\nu$ \cite{Haines}. Restriction of characters from $G_\nu$ to its minimal Levi 
subgroup $M_\nu$ is a faithful functor, thus
\[
Z(\hG_\nu)_\Fr^\circ \cong \Xnr (G_\nu) \to \Xnr (M_\nu) \cong \hM_{\nu,\Fr} \text{ is injective.}
\]
Now we can also regard $\zeta |_{\pi_{0,M,\phi}}$ as a map with range $\hM_{\nu,\Fr}$, and then 
it recovers the map \eqref{eqn:map-defining-KMphi}.
\end{proof}

Recall from \eqref{eq:3.26} and \eqref{eq:3.27} that $\End (\CS^1_{G_\nu}) \cong \CH (G_\nu,I_\nu)$
has a basis consisting of elements $T_w \theta_\lambda$, where $w \in \bW (G_\nu,M_\nu) 
\cong \bW(\hG_\nu, \hM_\nu)^\Fr$ and $\lambda \in X_* (M_\nu) \cong X^* (\hM_{\nu,\Fr})$. 
In the following lemma we make the action of $\pi_{0,G,\phi}$ on $\End (\CS^1_{G_\nu})$ 
from Proposition \ref{prop:4.actionXonSGnu} explicit.

\begin{lemma}\label{lem:4.actionXonHGI}
For $x \in \pi_{0,G,\phi}$, $w \in \bW(\hG_\nu, \hM_\nu)^\Fr$ and 
$\lambda \in X^* (\hM_{\nu,\Fr})$, we have
\[
x \cdot T_w \theta_\lambda = 
\zeta (x)^{-1}(\lambda) \, T_{\Ad_{x_P}(w)} \theta_{\Ad_{x_P}(\lambda)} .
\]
\end{lemma}
\begin{proof}
In the proof of Theorem \ref{thm:unipotent categorical}, we showed that
\[
\Ad_{x_P} \cdot T_w \theta_\lambda = T_{\Ad_{x_P}(w)} \theta_{\Ad_{x_P}(\lambda)} .
\]
By Proposition \ref{prop:4.actionXonSGnu}.(2), it remains to prove that
\[
\zeta (x) \cdot T_w \theta_\lambda = \zeta (x)^{-1}(\lambda) \, T_w \theta_\lambda .
\]
By Theorem \ref{thm:unipotent categorical}, we may just as well check this in
$\CH (G_\nu, I_\nu)$, where the action of $\zeta (x) \in \Xnr (G_\nu)$ on 
$\CH (G_\nu ,I_\nu) \cong \End_{G_\nu}(i_{P_\nu}^{G_\nu} \chi^{\un})$ arises from the
action of $\zeta (x)$ on 
\[
\CH (M_\nu ,I_{M_\nu}) \cong \End_{M_\nu}(\chi^{\un}) \cong \CC [\Xnr (M_\nu)] 
\]
via parabolic induction. The group $\Xnr (G_\nu)$ acts on $\Xnr (M_\nu)$ by 
multiplication (of characters of $M_\nu$), thus for $\theta_\lambda$ regarded as an element
of $\CC [\Xnr (M_\nu)]$ we have
\[
\zeta (x) \cdot \theta_\lambda = \theta_\lambda \circ [\chi \mapsto \zeta (x)^{-1} \chi] =
\theta_\lambda (\zeta (x)^{-1}) \theta_\lambda = \zeta (x)^{-1}(\lambda) \theta_\lambda .
\]
The action of $\zeta (x) \in \Xnr (G_\nu)$ on $\CH (G_\nu,I_\nu)$ is also induced by 
the action $\zeta (x) \cdot \pi= \zeta (x) \otimes \pi$ on $G_\nu$-representations $\pi$.
This means that for any $f \in \CH (G_\nu,I_\nu)$ we may identify $\zeta (x) \cdot f$
with the pointwise product $\zeta (x)^{-1} f$ of $I_\nu$-biinvariant functions on $G_\nu$.
The element $T_w \in \CH (G_\nu,I_\nu)$ has support in a compact subgroup of $G_\nu$, hence
in the kernel of $\zeta (x)$. Therefore $\zeta (x) \cdot T_w = T_w$.
\end{proof}

Note that although the group ${\widehat R}_{G,\phi}$ does not act on $X^1_{\LG}$, the above identification of ${\widehat R}_{G,\phi}$ with a subgroup of $\bW(\hM,\phi)$ yields an action of ${\widehat R}_{G,\phi}$ on $X^1_{\LG}/\pi_{0,M,\phi}$. Given an element of ${\widehat R}_{G,\phi}$, we can lift the corresponding element of $\bW(\hM,\phi)$ to an element of $\Omega_{\phi,G,\phi(\Fr)}$, well-defined up to an element of $\pi_{0,M,\phi}$; the action on this element on $X^1_{\LG}/\pi_{0,M,\phi}$ is then well-defined.  Moreover, for any parabolic subgroup $\LQ$ of $\LG$ containing $\LP$, with Levi $\LL$, and any $w \in {\widehat R}_{G,\phi}$, we obtain well-defined isomorphisms 
\begin{equation}
X^1_{\LQ}/\pi_{0,M,\phi} \xrightarrow{\sim} X^1_{\LQ^w}/\pi_{0,M,\phi},\quad \text{and}\quad X^1_{\LL}/\pi_{0,M,\phi} \xrightarrow{\sim} X^1_{\LL^w}/\pi_{0,M,\phi}.
\end{equation}

\subsection{Spectral parabolic induction}

The isomorphism in Theorem~\ref{thm:Springer reduction} is compatible with certain maps on endomorphism algebras of Springer sheaves induced by the spectral parabolic induction functors.  Indeed, as in the previous section, let $\LQ$ be a parabolic subgroup of $\LG$ with Levi $\LL$, and assume that $\LQ$ contains the parabolic $\LP$ of the previous section; associated to these we have a pair ${\hL}_{\nu}, {\hQ}_{\nu}$ of a standard Levi subgroup and a standard parabolic subgroup of ${\hG}_{\nu}$, respectively.

\begin{prop} \label{prop:Springer induction}
We have isomorphisms:
$\CS^{\phi}_G \cong (\pi_Q)_* r_Q^* S^{\phi}_L$ and 
$\CS^1_{G_{\nu}} \cong (\pi_{Q_{\nu}})_* r_{Q_{\nu}}^* \CS^1_{L_{\nu}}$.
\end{prop}
\begin{proof}
When $G=\GL_n$, this is proven in~\cite[\S 5.3.2]{BZCHN}.  The argument in the general case is identical, hence we omit it.
\end{proof}

These isomorphisms give rise, via functoriality, to natural maps:
\begin{equation}\label{eqn:maps-between-End-parabind}
\End(\CS^1_{L_{\nu}}) \rightarrow \End(\CS^1_{G_{\nu}})\quad\text{and}\quad 
\End(\CS^{\phi}_L) \rightarrow \End(\CS^{\phi}_G).
\end{equation}

On the other hand, we have a natural inclusion of $\bW(\hL,\phi)$ in $\bW(\hG,\phi)$, and we verify:
\begin{lemma}\label{lem:injection-Rhat}
The inclusion of $\bW(\hL,\phi)$ in $\bW(\hG,\phi)$ restricts to an injection of ${\widehat R}_{L,\phi}$ into ${\widehat R}_{G,\phi}$.
\end{lemma}
\begin{proof}
On the one hand, any element of $\bW({\widehat L},\phi)$ normalizes $\hQ$, and thus also $\hQ_{\nu}$ and its unipotent radical.  On the other hand, an element of ${\widehat R}_{L,\phi}$ also normalizes $\hP_{\nu} \cap \hL_{\nu}$.  Since $\hP_{\nu}$ is generated by $\hP_{\nu} \cap \hL_{\nu}$ and the unipotent radical of $\hQ_{\nu}$, any element of ${\widehat R}_{L,\phi}$ normalizes $\hP_{\nu}$, and thus is an element of ${\widehat R}_{G,\phi}$.
\end{proof}

It is then straightforward to check that the isomorphism \eqref{eqn:thm-Springer-reduction} of Theorem~\ref{thm:Springer reduction} 
is compatible with the maps in \eqref{eqn:maps-between-End-parabind}. 
\begin{prop} \label{prop:spectral induction compatibility}
We have a commutative diagram:
\begin{equation}
\begin{tikzcd}
\End(S^{\phi}_L)\arrow[]{d}{} \arrow[]{r}{\simeq} & \End(\CS^1_{L_{\nu}})^{\pi_{0,M,\phi}} \rtimes \CC[{\widehat R}_{L,\phi}]\arrow[]{d}{}\\
\End(\CS^{\phi}_G) \arrow[]{r}{\simeq} & \End(\CS^1_{G_{\nu}})^{\pi_{0,M,\phi}} \rtimes \CC[{\widehat R}_{G,\phi}]
\end{tikzcd}
\end{equation}
where the vertical maps on the right-hand side are the maps \eqref{eqn:maps-between-End-parabind} on the first factor and the natural injection of ${\widehat R}_{L,\phi}$ into ${\widehat R}_{G,\phi}$ on the second factor, and the horizontal maps are given by \eqref{eqn:thm-Springer-reduction}. 
\end{prop}
\begin{proof}
   This follows from the compatibility of diagram~\ref{eqn:comm-diag-spectral-side} with the corresponding diagram with $\LL$ in place of $\LG$, under the appropriate spectral parabolic induction correspondences.
\end{proof}

\section{Reduction to unipotent representations}\label{sec:conjecture}

The conjectural categorical Langlands correspondence suggests that Theorem~\ref{thm:Springer reduction} should have a natural analogue in representation theory.  The goal of this section is to precisely formulate such an analogue.

\subsection{Langlands compatibility}

We will formulate our analogue of Theorem~\ref{thm:Springer reduction} in terms of a pair $(\sigma,\phi)$, where $\sigma$ is an irreducible generic supercuspidal representation of a Levi subgroup $M$ of $G$, and $\phi: W_F \rightarrow \LM$ is a supercuspidal Langlands parameter.  Morally, we will think of $\phi$ as being the Langlands parameter associated to $\sigma$.  However, as there are situations in which the local Langlands correspondence is not known, we will instead assume certain compatibilities between $\sigma$ and $\phi$, that we now make precise.

The first compatibility we define is a compatibility with twisting by unramified characters.  Let $M^1$ be the intersection of the kernels of all unramified characters of $M$; equivalently, $M^1$ is the subgroup of $M$ generated by all open compact subgroups of $M$.  The complex torus $\Hom(M/M^1,\GG_m)$ is the group $\mathfrak{X}_{\mathrm{nr}}(M)$ of unramified characters of $M$.  Let 
\begin{equation}\label{eqn:Stab-sigma}
\Stab(\sigma):=\{\chi\in \mathfrak{X}_{\mathrm{nr}}(X)|\sigma \otimes \chi\simeq \sigma\}
\end{equation}
be the set of unramified characters $\chi$ of $M$ such that $\sigma \otimes \chi$ is isomorphic to $\sigma$.  It is a subgroup of $\Hom(M/M_1,\GG_m)$.  The quotient $\Hom(M/M^1,\GG_m)/\Stab(\sigma)$ is the ``torus of unramified twists of $\sigma$'', i.e. $\{\sigma\otimes\chi|\chi\in\mathfrak{X}_{\nr}(M)\}$; we will denote this torus by $\CO_{\sigma}$.

The Langlands correspondence for characters identifies $\Hom(M/M^1,\GG_m)$ with the torus 
$(Z({\hM})^{I_F})^{\circ}_{\Fr}$ on the spectral side \cite{Haines}. This group acts naturally on
$X_{\LM}$ according to \eqref{eq:4.actionZ}, and in particular on $X^\phi_{\LM}$ and $\overline{X}^\phi_{\LM}$.
By Lemma~\ref{lemma:irreducible}, the action of $(Z({\hM})^{I_F})^{\circ}_{\Fr}$ on  $\overline{X}^\phi_{\LM}$
is transitive.

\begin{definition}\label{defn:twist-compatible}
We say $\sigma$ and $\phi$ are {\em twist compatible} if the stabilizer of $\phi \in \overline{X}_\LM$ in
$(Z({\hM})^{I_F})^{\circ}_{\Fr}$ corresponds to Stab$(\sigma) \subset \Xnr (M)$ via the Langlands correspondence 
for characters.
\end{definition}

When $\sigma$ and $\phi$ are twist compatible, the actions of $\Xnr (M)$ provide a bijection
\begin{equation}
\CO_\sigma \to \overline{X}^\phi_\LM : \chi \otimes \sigma \mapsto \hat \chi \phi ,    
\end{equation}
where $\hat \chi \in (Z({\hM})^{I_F})^{\circ}_{\Fr}$ corresponds to $\chi \in \Xnr (M)$.
This identifies $\CO_{\sigma}$ with the set of ($\CC$-valued) Langlands parameters on $X^{\phi}_{\LM}$ in exactly 
the manner predicted by the expected property ``compatibility with twisting by unramified characters'' 
of the local Langlands correspondence.

The second compatibility
we impose is related to the action of Weyl groups on $\CO_{\sigma}$ and $\overline{X}^{\phi}_{\LM}$.  Let $\bW(G,M)$ denote the quotient $N_G(M)/M$; then $\bW(G,M)$ acts on the set of isomorphism classes of irreducible $\CC$-representations of $M$ by conjugation.  Let $\bW(M,\sigma)$ denote the subgroup of $\bW(G,M)$ consisting of those $w$ such that $\sigma^w$ is a twist of $\sigma$ by an unramified character of $M$, i.e.
\begin{equation}\label{eqn:defn-WMsigma}
    \bW(M,\sigma):=\{w\in \bW(G,M)|\sigma^w\simeq\sigma\otimes\chi\text{ for some }\chi\in\mathfrak{X}_{\nr}(M)\}. 
\end{equation}
Then $\bW(M,\sigma)$ acts on $\CO_{\sigma}$.

On the spectral side, we have a group $\bW(\hM,\phi)$, defined as in \eqref{eqn:defn-W-Mhat-phi}, analogous to 
$\bW(M,\sigma)$.  The groups $\bW(M,\sigma)$ and $\bW(\hM,\phi)$ may be naturally identified with subquotients of 
the groups $\bW(G,M)$ and $\bW(\hG, \hM)$, respectively. There is a natural isomorphism \cite[Proposition 3.1]{ABPS}
\begin{equation}\label{eq:5.WGM}
\bW (G,M) \cong \bW (\hG,\hM) .
\end{equation}

\begin{definition}\label{defn:Weyl-compatible}
We say $\sigma$ and $\phi$ are {\em Weyl compatible} if they are twist compatible, and \eqref{eq:5.WGM} 
induces an isomorphism of $\bW(M,\sigma)$ with $\bW(\hM,\phi)$, and if moreover 
the isomorphism: $\CO_{\sigma} \cong \overline{X}^{\phi}_{\LM}$ is compatible with the actions of $\bW(M,\sigma)$ 
on the former and $\bW(\hM,\phi)$ on the latter via this isomorphism.
\end{definition}

Note that when $\sigma$ and $\phi$ are related by a local Langlands correspondence that is {\em compatible with isomorphisms} in the sense of~\cite[5.2.4]{Haines}, then $\sigma$ and $\phi$ will be Weyl compatible in the sense of Definition~\ref{defn:Weyl-compatible}.

The last compatibility we impose is more subtle, and is motivated by Langlands' conjecture on Plancherel measures.  
Let $\LL$ be a Levi subgroup (not necessarily proper) of $\LG$. We say that $\LL$ minimally contains $\LM$ if
$\LL \supsetneq \LM$ and $\LL$ is minimal with respect to this property.  Then $\LL$ is the L-group of some 
Levi subgroup $L$ of $G$ containing $M$.

Associated to $L$ and $\sigma$, we have the Harish-Chandra $\mu$-function $\mu_{L,\sigma}$ \cite{Wal}, which is 
a $W(M,\sigma)$-invariant rational function on $\CO_{\sigma}$. This function is an ingredient of an explicit
formula for the Plancherel density for irreducible tempered L-representations \cite{Wal}.
For any $\chi \in \Xnr (M)$, the representation $i_{P_L}^L (\sigma\otimes\chi)$ has length at most two. By the
properties of the intertwining operators used to construct $\mu_{L,\sigma}$, the representation 
$i_{P_L}^L (\sigma\otimes\chi)$ is indecomposable if and only if $\mu_{L,\sigma}(\sigma \otimes \chi) = \infty$.
Let $\mathrm{Pol}_{L,\sigma}$ be the set of $\CC$-points $\sigma'$ of $\CO_{\sigma}$ such that $\mu_{L,\sigma}$ 
has a pole at $\sigma'$.

On the spectral side, for any $\CC$-point of the coarse moduli space $\overline{X}^{\phi}_{\LL}$, corresponding to 
some Langlands parameter $\varphi$, we may consider 
the adjoint L-function $L(s, \Ad_{\hL} \varphi)$ and the adjoint $\gamma$-factor $\gamma (s,\Ad_{\hL} \varphi)$.  
The poles of $L(s=1, \Ad_{\hL} \varphi)$ are the poles of $\gamma (s=0,\Ad_{\hL} \varphi)$, and these functions of
$\varphi$ relate to Plancherel densities and formal degrees \cite{HII}. Let $\widehat{\mathrm{Pol}}_{L,\phi}$ denote 
the set of $\varphi \in \overline{X}^\phi_{\LM}$ such that $L(1,\Ad_{\hL} \varphi) = \infty$.

\begin{definition}\label{defn:Plancherel-compatibility}
We say $\sigma$ and $\phi$ are {\em Plancherel compatible} if they are twist compatible and the isomorphism 
$\CO_{\sigma} \cong \overline{\X}^{\phi}_{\LM}$ identifies $\mathrm{Pol}_{L,\sigma}$ with 
$\widehat{\mathrm{Pol}}_{L,\phi}$, for all $\LL$ minimally containing $\LM$.
\end{definition}

In Appendix \ref{appendix}, we study many equivalent formulations of Plancherel comptability, and moreover prove Plancherel compatibility for a large class of reductive $p$-adic groups.

\begin{definition}\label{defn:LLC-compatible}
        We will say $\sigma$ and $\phi$ are {\em Langlands compatible}
    if they are twist compatible, Weyl compatible, and Plancherel compatible.
\end{definition}

\begin{remark}
\rm We note that although twist compatibility depends only on the pairs $(M,\sigma)$ and
$(\LM,\phi)$, the notions of Weyl compatibility and Plancherel compatibility (and thus also 
the notion of Langlands compatibility) depend also on the ambient group $G$ and its L-group 
$\LG$.  In particular when $M = G$, the Weyl and Plancherel compatibilities are vacuous.
\end{remark}

In situations where we are considering more than a single block we will need a way systematically assigning Langlands parameters to generic supercuspidals that satisfy the above compatibilities.  The following definition formalizes what we will need:

\begin{definition} \label{def:wgsc}
By a {\em weak generic supercuspidal correspondence} for $G$ and its Levi's, we mean for each standard Levi 
subgroup $M$, a map 
$\sigma \mapsto \Phi_M(\sigma)$ from irreducible generic supercuspidal representations of $M$ to $\hM$-conjugacy 
classes of supercuspidal Langlands parameters $\Phi: W_F \rightarrow \LM$, such that the following conditions hold:
\begin{enumerate}
\item For any unramified character $\chi$ of $M$, we have 
\begin{equation}\label{eq:5.unramified twists}
\Phi_M(\sigma \otimes \chi)|_{I_F} = \Phi_M(\sigma)|_{I_F}  \;\text{ and }\;
\Phi_M(\sigma \otimes \chi)(\Fr) = \hat \chi \Phi_M(\sigma),
\end{equation}
where $\hat \chi$ is a $\CC$-point of $Z({\hM})^{I_F}$ representing the class in 
$(Z({\hM})^{I_F})^{\circ}_{\Fr}$ that corresponds to $\chi$ under the Haines' isomorphism
\begin{equation}
\Xnr (M) \cong (Z({\hM})^{I_F})^{\circ}_{\Fr} .
\end{equation}
\item For any element $w$ of $\bW(G,M) \cong \bW (\hG,\hM)$, we have 
$\Ad (w) \Phi_M(\sigma) = \Phi_{M}(w \cdot \sigma)$.
\item For any pair $(M,\sigma)$, the parameter $\Phi_M(\sigma)$ is Plancherel compatible with $\sigma$.
\end{enumerate}
\end{definition}

It is straightforward to show that for such a correspondence, the parameter $\Phi_M(\sigma)$ is Langlands 
compatible with $\sigma$ for every $M$ and $\sigma$.

\begin{remark} \rm The properties of a weak generic supercuspidal correspondence do {\em not} uniquely characterize such a correspondence.  For instance, for $\GL_n$ any 
twist of a weak generic correspondence by a character $\chi \circ \det$ yields a different 
weak generic supercuspidal correspondence.

We do not demand, neither in the formulation of Langlands compatibility nor in that of a weak generic supercuspidal
correspondence, that tempered representations are matched with bounded L-parameters. For
what we do in this paper, this is simply not necessary. One could add this additional property by requiring
in addition that $\sigma$ is tempered and $\phi$ (or $\Phi (\sigma)$) is bounded. 
\end{remark}

\begin{remark}[The non-generic case]\label{rem:wsc} \rm
All of the above compatibilities (1)--(3) in Definition \ref{def:wgsc} also make sense for \textit{non-generic} supercuspidal representations $\sigma$. 
Then $\phi$ must be discrete but may have nontrivial mo\-no\-dromy $N$.
In this case, one must replace $\overline{X}^\phi_\LM$ by $(Z(\hM)^{I_F})_\Fr^\circ \cdot\phi$
and regard $\Ad_{\hL} \varphi$ as a $W_F$-representation on $Z_{\Lie (\hL)}(N)$. Furthermore, 
$\phi$ must be equipped with an enhancement $\epsilon$ such that the enhanced $L$-parameter 
$(\phi,\epsilon)$ is cuspidal.~(In the above, we are implicitly taking the trivial 
enhancement, see also Lemma \ref{lem:5.10}.) Then all the conditions in Definition 
\ref{def:wgsc} can be reformulated with $(\phi,\epsilon)$ instead of $\phi$. This formulates 
a weak (not necessarily generic) supercuspidal correspondence, generalizing Definition
\ref{def:wgsc} to the non-generic case. 
\end{remark}

\begin{remark}\label{remark:classical-LLC} \rm
There are many reductive groups over non-archimedean local fields that are known to admit a weak generic supercuspidal correspondence in the sense of Definition \ref{def:wgsc}. For instance, in~\cite[\S 7]{conjecture} it is shown that the known local Langlands correspondences for general linear groups, symplectic groups, odd special orthogonal groups and unitary groups satisfy the requirements. Moreover, as expected, the Langlands parameter associated to an irreducible generic supercuspidal representation is supercuspidal. 
We refer the reader to Appendix \ref{appendix} for more details and further examples.
\end{remark}

\subsection{An automorphic analogue of Theorem~\ref{thm:Springer reduction}}
We are now an a position to formulate a representation-theoretic analogue of 
Theorem~\ref{thm:Springer reduction}.  Fix a standard Levi subgroup $\bM$ of $\bG$, a 
Whittaker datum $(U,\psi)$ for $\bG$, and an irreducible $(U_M,\psi_M)$-generic supercuspidal
representation $\sigma$ of $M$.  We then have the spaces of compact Whittaker functions 
$\mathcal{W}_L$ for any standard Levi subgroup $\bL$ of $\bG$, associated to the choice of 
$(U,\psi)$.  If $\bL$ contains $\bM$, we let $\cW_L^{[M,\sigma]}$ denote the summand of 
$\cW_L$ corresponding to the Bernstein block $\Rep(L)_{[M,\sigma]}$ of $L$ corresponding 
to the inertial equivalence class of $(M,\sigma)$.  Let $\phi: W_F \rightarrow \LM$ be 
a supercuspidal Langlands parameter that is Langlands compatible with $\sigma$.

The conjectural categorical local Langlands correspondence then predicts that we should have 
an isomorphism of $\End_G \big( i_{P}^G \cW_M^{[M,\sigma]} \big)$ with the endomorphism 
algebra of the coherent Springer sheaf $\CS_G^{\phi}$.  Combining this prediction with 
Theorem~\ref{thm:Springer reduction} we arrive at the following, purely 
representation-theoretic statement:

\begin{thm} \label{thm:automorphic reduction}
For each standard Levi $L$ of $G$, and standard parabolic $P$ of $G$ with Levi $M$, we have isomorphisms:
\[
\End_L \big( i_{P_L}^L \cW_M^{[M,\sigma]} \big) \cong 
\End_{L_{\nu}}(i_{P_L}^{L_{\nu}} \chi^{un})^{K_{M,\phi}} \rtimes \CC[{\widehat R}_{L,\phi}],
\]
where $P_L$ denotes the intersection of $P$ with $L$. Here the action of $K_{M,\phi}$ is
induced from its action on $\chi^{un} \cong \CC[\Xnr(M_\nu)]$, which comes from translations 
by elements of $\Xnr (M_\nu)$. The action of $\widehat R_{L,\phi}$ on 
$\End_{L_{\nu}}(i_{P_L}^{L_{\nu}} \chi^{un})^{K_{M,\phi}}$ is induced, via Weyl compatibility,
from its action on $\CO_\sigma \cong \overline{X}^\phi_{\LM}$. 
When $\sigma$ and $\phi$ are given, these algebra isomorphisms are canonical.

Moreover, let $Q$ denote the standard parabolic of $G$ with Levi $L$.  
Then the above isomorphisms fit into a commutative diagram:
\begin{equation}
\begin{tikzcd}
\End_L \big( i_{P_L}^L \cW_M^{[M,\sigma]} \big) \arrow[]{r}{\cong}\arrow[]{d}{}
& \End_{L_{\nu}} \big( i_{P_L}^{L_{\nu}} \chi^{un} \big)^{K_{M,\phi}} \rtimes \CC[{\widehat R}_{L,\phi}] \arrow[]{d}{}\\
\End_G \big( i_{P}^G \cW_M^{[M,\sigma]} \big) \arrow[]{r}{\cong}  
& \End_{G_{\nu}} \big( i_{P}^{G_{\nu}} \chi^{\un} \big)^{K_{M,\phi}} \rtimes \CC[{\widehat R}_{G,\phi}],
\end{tikzcd}
\end{equation}
where the left-hand vertical arrow is induced by parabolic induction $i_{Q}^G$, and the right-hand vertical 
arrow is induced by the crossed product of the maps induced by $i_{Q_{\nu}}^{G_{\nu}}$ on the first factor 
and the injection of $\widehat{R}_{L,\phi}$ into $\widehat{R}_{G,\phi}$ on the second factor.

Finally, under these identifications, the $\End_G \big( i_{P}^G \mathcal{W}_M^{[M,\sigma]} \big)$-module 
$\Hom_G \big( i_{P}^G \cW_M^{[M,\sigma]}, \cW_G^{[M,\sigma]} \big)$ is identified with 
the $\End_{G_{\nu}} \big( i_{P_{\nu}}^{G_{\nu}} \chi^\un \big)^{K_{M,\phi}}$ module 
$\Hom_{G_{\nu}} \big( i_{P_{\nu}}^{G_{\nu}} \chi^{\un}, \cW_{G_\nu}^{[M_{\nu},1]} \big)^{K_{M,\phi}}$.  
\end{thm}

\section{Proof of Theorem \ref{thm:automorphic reduction}}\label{sec:Proof-main1} 

As in previous sections, let $\nu:=\phi|_{I_F}$. By Corollary \ref{cor:pinning}, we can 
twist $\phi$ by an unramified character of $M$ such that the resulting L-parameter preserves 
a pinning of $\hG_\nu$ associated to the Borel pair $({\hM}_{\nu},{\hP}_{\nu})$ 
of ${\hG}_{\nu}$. Therefore we may and will assume that $\phi$ itself preserves such a 
pinning. Note that this involves adjusting $\sigma$ as well, via twist compatibility.
We use this $\phi$ to define $\LG_\nu$ and $\LM_\nu$, and hence $G_\nu$ and $M_\nu$.

\subsection{Automorphic side endomorphism algebras}\label{subsec:End-alg-auto-side}
One key ingredient for proving Theorem~\ref{thm:automorphic reduction}
comes from results of Heiermann \cite{Heiermann} and the second author \cite{solleveld-endomorphisms} describing the endomorphism algebras of certain progenerators for Bernstein blocks on the automorphic side. We give a brief review in this section.

Let $\bM$ be a standard Levi subgroup of $\bG$ and let $\sigma$ be an irreducible, $(U_M,\psi_M)$-generic supercuspidal representation of $G$.  Recall that $M^1$ denotes the intersection of the kernels of all unramified characters of $M$ (and thus in particular contains $U_M$.) The genericity of $\sigma$ implies that the restriction of $\sigma$ to $M^1$ is multiplicity-free, i.e. the $\Hom$-space
\begin{equation}\Hom_M(\cW_M,\sigma) \cong \Hom_{M^1}(\ind_{U_M}^{M^1} \psi_M, \sigma)
\end{equation}
is one-dimensional.  Thus there is a unique irreducible $M^1$-subrepresentation $\sigma^1$ of $\sigma$ that admits a nonzero map from $\ind_{U_M}^{M^1} \psi$, and this subrepresentation $\sigma^1$ occurs with multiplicity one in $\sigma$.  The remaining $M^1$-subrepresentations of $\sigma$ are $M$-conjugate to $\sigma^1$ and thus also occur with multiplicity one.

\begin{lemma}\label{lem:6.3}
For any nonzero map from $\ind_{U_M}^{M^1} \psi_M$ to $\sigma^1$, the induced map
\begin{equation}\cW_M \rightarrow \ind_{M^1}^M \sigma^1\end{equation}
gives an isomorphism $\mathcal{W}_M^{[M,\sigma]} \cong \ind_{M^1}^M \sigma^1$ of $M$-representations. 
\end{lemma}
\begin{proof}
Both $\mathcal{W}_M^{[M,\sigma]}$ and $\ind_{M^1}^M \sigma^1$ are projective generators of $\Rep(G)_{[M,\sigma]}$, and the Mackey formula shows that $\Hom_M(W_M,\ind_{M^1}^M \sigma^1)$ is free of rank one over $\End_M(\ind_{M^1}^M \sigma^1)$, generated by the map $\cW_M \rightarrow \ind_{M^1}^M \sigma^1$ described above.  
\end{proof}

By Lemma \ref{lem:6.3}, the multiplicity-one property of $\sigma^1$ and 
\cite[\S 10.1]{solleveld-endomorphisms}, there is a natural isomorphism 
$\End_M (\mathcal W_M^{[M,\sigma]}) \cong \CC [ \CO_\sigma]$. 
By twist compatibility, tensoring with $\sigma$ 
and using $\tilde \iota_{\phi}$ in \eqref{eqn:moduli-reduction-thm} induce bijections
\begin{equation}\label{eqn:isom-tori}
\mathfrak{X}_\nr (M) / \mathrm{Stab}(\sigma) \to \CO_\sigma \to \overline{X}^\phi_{\LM} \cong 
\overline{X}^1_{\LM_\nu} / K_{M,\phi} \cong \Xnr (M_\nu) / K_{M,\phi} . 
\end{equation}
It follows that $\sigma$ and $\phi$ also induce isomorphisms
\begin{equation}\label{eq:6.4}
\End_M \big( \mathcal W_M^{[M,\sigma]} \big) \cong \CC [ \CO_\sigma] \cong 
\CC [ \Xnr (M_\nu) / K_{M,\phi} ] 
= \CC[ \Xnr (M_\nu) ]^{K_{M,\phi}} = \End_{M_\nu}(\chi^{un})^{K_{M,\phi}} ,
\end{equation}
where $K_{M,\phi}$ acts via translations on $\Xnr (M_\nu)$. 

By \cite[Theorem F]{solleveld-endomorphisms}, one can describe the $G$-endomorphism algebra of 
$i_{P}^G \ind_{M^1}^M \sigma^1 \cong i_{P}^G \mathcal W_M^{[M,\sigma]}$ in terms of a 
Hecke algebra  with unequal parameters attached to a root datum, which we now describe. Let $\bA_M$ be the maximal $F$-split torus in the center $Z(\bM)$ of $\bM$, and let $X_*(\bA_M)$ denote its cocharacter lattice. 
Denote by $\Sigma$ the set of nonzero weights occurring in the adjoint representation of $\bA_M$ on $\Lie(G)$, and let $\Sigma_{\red}$ be the set of indivisible elements of $\Sigma$. For each $\alpha\in\Sigma_{\red}$, there is a Levi subgroup $\bM_{\alpha}$ of $\bG$ containing $\bM$, and minimal with respect to the property that $\alpha$ is a nonzero weight of the action of $A_M$ on $\Lie(\bM_{\alpha})$. (Explicitly, $\bM_{\alpha}$ is the centralizer, in $\bG$, of the kernel of $\alpha$ on $\bA_M$.)  Note that $\bM_{\alpha}$ only depends on $\alpha$ up to sign, so that there is a bijection between $\Sigma_{\red}/{\pm 1}$ and the set of Levi's of $\bG$ of the form $\bM_{\alpha}$.

For each $\alpha \in \Sigma_{\red}$, consider the Harish-Chandra $\mu$-function $\mu_{M_{\alpha},\sigma}$ from \cite[\S 1]{Silberger-79} attached to the Levi $\bM_{\alpha}$ of $\bG$ and the torus $\CO_{\sigma}$; it is a rational function on $\CO_{\sigma}$.  Let $\Sigma_{\CO_{\sigma},\mu}$ denote the subset of $\Sigma_{\red}$ consisting of those $\alpha$ for which $\mu_{M_{\alpha},\sigma}$ has a zero on $\CO_{\sigma}$.

The map $X_*(\bA_M) \rightarrow M/M^1$ given by evaluation at the inverse of a uniformizer of $F$ is injective with finite cokernel, so we can regard $M/M^1$ as a sublattice of $X_*(\bA_M) \otimes \QQ$. In particular, there is a bilinear pairing:
$$M/M^1 \times X^*(\bA_M) \rightarrow \ZZ,$$
restricting the pairing between $X^*(\bA_M)$ and $X_*(\bA_M)$.

For $\alpha \in \Sigma_{\CO_{\sigma},\mu}$, note that $(M \cap M_{\alpha}^1)/M^1$ is a rank-one sublattice in $M/M^1$, and is equal to the kernel of the map $M/M^1 \rightarrow M_{\alpha}/M_{\alpha}^1$. Recall from \eqref{eqn:Stab-sigma} that $\Stab(\sigma)$ is the group of unramified characters of $M$ that preserve the isomorphism class of $\sigma$ under twisting by characters in $\mathfrak{X}_{\nr}(M)$; in particular, we can evaluate elements of $\Stab(\sigma)$ on $M/M^1$.  Let $(M/M^1)^{\Stab(\sigma)}$ denote the common kernel, on $M/M^1$, of all elements of $\Stab(\sigma)$; it is a finite index sublattice of $M/M^1$.  We then have a sequence of maps:
\begin{equation}\label{eqn:composition-MmodM1}
(M/M^1)^{\Stab(\sigma)} \hookrightarrow M/M^1 \rightarrow M_{\alpha}/M_{\alpha}^1,
\end{equation}
and the kernel of the composition \eqref{eqn:composition-MmodM1} is free of rank-one over $\ZZ$.  Let $h_{\alpha}^{\vee}$ be the generator of this kernel that pairs positively with $\alpha$. 
We then set $\alpha^{\sharp}$ to be the $\QQ$-multiple of $\alpha$ in $X^*(\bA_M)$ such that $\langle\alpha^{\sharp}, h_{\alpha}^{\vee}\rangle = 2$; it is an element of $\Hom((M/M^1)^{\Stab(\sigma)},\ZZ).$  Let $\Sigma^{\vee}_{\CO_{\sigma}}$ denote the set of $h_{\alpha}^{\vee}$ for $\alpha \in \Sigma_{\CO_\sigma,\mu}$ and $\Sigma_{\CO_{\sigma}}$ the set of $\alpha^{\sharp}$ for $\alpha \in \Sigma_{\CO_\sigma,\mu}$. 
We then have \cite[Proposition 3.1]{solleveld-endomorphisms}:
\begin{prop} \label{prop:solleveld}
The tuple $\left((M/M^1)^{\Stab(\sigma)}, \Sigma^{\vee}_{\CO_{\sigma}}, \Hom((M/M^1)^{\Stab(\sigma)}, \ZZ), \Sigma_{\CO_{\sigma}}\right)$ is a root datum.
\end{prop}
Let $\bW_{\CO_{\sigma}}$ denote the Weyl group of the root datum in Proposition \ref{prop:solleveld} (this is denoted $\bW(\Sigma_{\CO,\mu})$ in~\cite{solleveld-endomorphisms}). 
As in \cite[\S 3]{solleveld-endomorphisms}, we let $\Sigma_{\red}(P)$ be the subset of $\Sigma_{\red}$ consisting of nonzero weights that appear in the adjoint action of $\bA_M$ on $\Lie(P)$, and let $\Sigma_{\CO_{\sigma}}(P)$ be the intersection of $\Sigma_{\red}(P)$ with $\Sigma_{\CO_{\sigma},\mu}$.  Then $\Sigma_{\CO_{\sigma}}(P)$ determines a set of positive roots for the root datum of Proposition~\ref{prop:solleveld}, and thus a set of simple roots; we denote the latter by $\Delta_{\CO_{\sigma}}$.

Recall that $\bW(M,\sigma)$, as defined in \eqref{eqn:defn-WMsigma}, denotes the subgroup of $N_G(M)/M$ that conjugates $\sigma$ to an unramified twist of $\sigma$.  We have an action of $\bW(M,\sigma)$ on the root datum of Proposition \ref{prop:solleveld} preserving $\Sigma_{\CO_{\sigma},\mu}$.  The discussion surrounding \cite[(3.1)--(3.2)]{solleveld-endomorphisms} shows that $\bW(M,\sigma)$ splits as a semidirect product:
\begin{equation}\label{eqn:WMO-semidirect-product}
\bW(M,\sigma) = \bW_{\CO_{\sigma}} \rtimes R(\CO_{\sigma})\end{equation}
where $R(\CO_{\sigma})$ is the subgroup of $\bW(M,\sigma)$ that stabilizes 
$\Delta_{\CO_{\sigma}}$.  Moreover, under the identification \eqref{eqn:WMO-semidirect-product}, $\bW_{\CO_{\sigma}}$ is identified with the subgroup of $\bW(M,\sigma)$ generated by the simple reflections $s_{\alpha}$ for $\alpha \in \Sigma_{\CO_{\sigma},\mu}$. 
To each element $\alpha \in \Delta_{\CO_{\sigma}}$, as in \cite{solleveld-endomorphisms}, one assigns a pair of Hecke algebra q-parameters $q_{\alpha}$, $q^*_{\alpha}$ as follows. We can fix a $\sigma'\in\CO_{\sigma}$ such that 
$\mu_{M_{\alpha},\sigma}(\sigma') = 0$ for all $\alpha \in \Delta_{\CO_{\sigma}}$.  
Define a function $X_{\alpha}$ on $\CO_{\sigma}$ by setting 
\begin{equation}\label{eq:6.7}
X_{\alpha}(\sigma' \otimes \chi) = \chi(h_{\alpha}^{\vee}),
\end{equation}
where $\chi$ is a character of $M/M^1$.  Then, up to a constant factor (c.f.~\cite[(3.7)]{solleveld-endomorphisms}; see also \cite{Silberger-79,Heiermann}) the function $\mu_{M_{\alpha},\sigma}$ may be written as
\begin{equation}\label{eqn:Silberger-Solleveld}
\frac{(1 - X_{\alpha}^2)(1 - X_{\alpha}^{-2})}{(1 - q_{\alpha}^{-1} X_{\alpha})(1 - q_{\alpha}^{-1} X_{\alpha}^{-1})(1 + (q^*_{\alpha})^{-1} X_{\alpha})(1 + (q^*_{\alpha})^{-1}X_{\alpha}^{-1})}
\end{equation}
for real numbers $q_{\alpha},q^*_{\alpha} \geq 1$. 
With an appropriate choice of $\sigma'$, we may further arrange that $q_{\alpha} \geq q^*_{\alpha}$
for all $\alpha \in \Delta_{\CO_\sigma}$. (We will see later, in Lemma \ref{lem:6.basepoints}, that our $\sigma$ already has these properties.)
Note that as a function on $\CO_{\sigma}$, which is a torsor over the torus
$\Xnr (M)/\Stab(\sigma)$, the function $X_{\alpha}$ is {\em monomial}, i.e.~there exists a trivialization of this torsor that identifies $X_{\alpha}$ with a character of $\Xnr (M) /\Stab(\sigma).$

Let $\CH_{\CO_{\sigma}}$ denote the Hecke algebra, with (possibly unequal) parameters $q_{\alpha},q^*_{\alpha}$ from \eqref{eqn:Silberger-Solleveld}, associated to the based root datum:
\begin{equation}\label{eqn:root-system-for-HOsigma}
((M/M^1)^{\Stab(\sigma)}, \Sigma^{\vee}_{\CO_{\sigma}}, \Hom((M/M^1)^{\Stab(\sigma)}, \ZZ), \Sigma_{\CO_{\sigma}}, \Delta_{\CO_\sigma}).
\end{equation}
Via the bijection 
\[
\CO_\sigma \to \Hom \big( (M/M^1)^{\Stab(\sigma)}, \CC^\times \big) \;\;\text{given by}\;\;
\sigma' \otimes \chi \mapsto \chi,
\]
we identify $\CC \big[ X^* \big( (M/M^1)^{\Stab(\sigma)} \big) \big] \subset \CH_{\CO_\sigma}$
with $\CC[\CO_{\sigma}]$. The sub-algebra $\CC[\CO_{\sigma}]$ of $\CH_{\CO_\sigma}$ is also the image of the map
\begin{equation}\CC[\CO_{\sigma}] \cong \End_M \big( \mathcal{W}_M^{[M,\sigma]} \big) \rightarrow 
\End_G \big( i_{P}^G \mathcal{W}_M^{[M,\sigma]} \big).
\end{equation}

\begin{thm} \label{thm:solleveld}
The base-point $\sigma'$ of $\CO_\sigma$ determines an isomorphism: 
\begin{equation}
\End_G \big( i_{P}^G \mathcal{W}_M^{[M,\sigma]} \big) \cong \CH_{\CO_{\sigma}} \rtimes \CC[R(\CO_{\sigma})].
\end{equation}
\end{thm}
\begin{proof}
By \cite[Theorem F]{solleveld-endomorphisms}, the endomorphism algebra 
$\End_G \big( i_{P}^G \mathcal{W}_M^{[M,\sigma]} \big)$ decomposes as a twisted crossed product of 
$\CH_{\CO_{\sigma}}$ with $\CC[R(\CO_{\sigma})]$, where the multiplication is twisted by a $2$-cocycle 
$\natural: \bW(M,\sigma)^2 \rightarrow \CC[\CO_{\sigma}]^{\times}.$  By \cite[Theorem A.1]{OSgeneric}, we may 
take this cocycle $\natural$ to be trivial. Here canonicity is guaranteed by \cite[Theorem 2.7]{solleveld-principal} 
with respect to the chosen Whittaker datum, which is used to renormalize the operators $T_w'$ in \cite{OSgeneric}. 
\end{proof}

In Theorem \ref{thm:solleveld}, the group $R(\CO_\sigma)$ acts on $\CH_{\CO_{\sigma}}$ in
the following way. Every $r \in R(\CO_\sigma)$ acts naturally on the based root datum 
\eqref{eqn:root-system-for-HOsigma}, say via a map $\psi_r : (M/M^1)^{\Stab(\sigma)} \to 
(M/M^1)^{\Stab(\sigma)}$. Then $\psi_r$ also gives rise to an algebra automorphism 
\begin{equation}\label{eq:6.psir}
\psi_r :  \CH_{\CO_{\sigma}} \to  \CH_{\CO_{\sigma}} ,\quad
\theta_x T_w \mapsto \theta_{\psi_r (x)} T_{\psi_r w \psi_r^{-1}} .
\end{equation}
Moreover, there is a unique $\chi_r \in \Hom \big( (M/M^1)^{\Stab(\sigma)}, \CC^\times \big)$
such that 
\begin{equation}\label{eq:6.psirchir}
r \cdot (\sigma' \otimes \chi) \cong \sigma' \otimes \tilde \chi_r \otimes \psi_r (\chi) 
\end{equation}
for all $\chi \in \Xnr (M)$ and for any lift $\tilde \chi_r$ of $\chi_r$ to $M / M^1$. 
This determines an automorphism of $\CH_{\CO_\sigma}$ by
\begin{equation}\label{eq:6.chir}
\chi_r \cdot (\theta_x T_w) = [\chi \mapsto x (\chi \chi_r^{-1})] T_w = 
x(\chi_r)^{-1} \theta_x T_w .
\end{equation}
In terms of \eqref{eq:6.psir} and \eqref{eq:6.chir}, the action of $r \in R(\CO_\sigma)$ on
$\CH_{\CO_\sigma}$ in Theorem \ref{thm:solleveld} is via $\chi_r \circ \psi_r$.

\subsection{Reinterpretation in terms of the spectral side}\label{subsec:spectral-side-lattices}

We will prove Theorem~\ref{thm:automorphic reduction} by applying the results of \S \ref{subsec:End-alg-auto-side} to the parabolic inductions $i_{P}^G \mathcal{W}_M^{[M,\sigma]}$ and $i_{P_{\nu}}^{G_{\nu}} \chi^{\un}$, and comparing the respective outputs. On the other hand, since the relation between the groups $\bG$ and $\bG_{\nu}$ come from their L-groups, to carry out this comparison it is helpful to rephrase these results in terms of data on the spectral side.

Passing to character groups, \eqref{eqn:isom-tori} yields isomorphisms:
\begin{equation}\label{eqn:lattices-of-tori-isom}
(M/M^1)^{\Stab(\sigma)} \cong X^*({\hM}_{\nu,\Fr} / K_{M,\phi}) \quad \text{and} \quad
\Hom((M/M^1)^{\Stab(\sigma)}, \ZZ) \cong X_*({\hM}_{\nu,\Fr} / K_{M,\phi} ).
\end{equation}
We would like to characterize the images of the subsets $\Sigma^{\vee}_{\CO_{\sigma}}$ and $\Sigma_{\CO_{\sigma}}$ 
of the two left-hand lattices, respectively, under isomorphisms \eqref{eqn:lattices-of-tori-isom}, in terms of $\phi$.
Since these subsets $\Sigma^{\vee}_{\CO_{\sigma}}$ and $\Sigma_{\CO_{\sigma}}$ are stable under negation, we work 
up to sign in what follows.

Recall that there is a bijection of $\Sigma_{\red}/ \{\pm 1\}$ with the set of Levi's $M_{\alpha}$ of $G$
minimally containing $M$; passing to L-groups, we have a bijection of this set with the set of Levi subgroups 
$\LM_{\alpha}$ of $\LG$ minimally containing $\LM$ (i.e.~centralizers of $W_F$-fixed tori in ${\hG}$ that are 
properly contained in $Z({\hM})^{W_F}$ and have maximal rank among such).

Note that the structure of $\mu_{M_{\alpha},\sigma}$ discussed near \eqref{eqn:Silberger-Solleveld} shows that 
$\mu_{M_{\alpha},\sigma}$ has a zero on $\CO_{\sigma}$ if, and only if, $\mu_{M_{\alpha},\sigma}$ has a pole on 
$\CO_{\sigma}$, and is constant otherwise.  Thus $\Sigma_{\CO_{\sigma},\mu}$ is the subset of $\Sigma_{\red}$ 
consisting of those $\alpha$ for which $\mu_{M_{\alpha},\sigma}$ has a pole.  By Plancherel compatibility 
(Definition \ref{defn:Plancherel-compatibility}), $\Sigma_{\CO_{\sigma},\mu}/ \{\pm 1\}$ is then in bijection with 
the set of minimal Levi's $\LM_{\alpha}$ of $\LG$ such that there exists a Langlands parameter $\varphi$ on 
$X^{\phi}_{\LM}$ with $L(1,,\Ad_{{\hM}_{\alpha}} \varphi) = \infty$. We denote this set by 
${\hat \Sigma}_{\CO_{\sigma},\mu}/ \{\pm 1\}$ (note that the symbol ${\hat \Sigma}_{\CO_{\sigma},\mu}$ has no 
independent meaning here; the notation is merely intended to indicate that this is a spectral side analogue of, 
and naturally in bijection with, $\Sigma_{\CO_{\sigma},\mu}/ \{\pm 1\}$).

Recall that for a fixed $\alpha \in \Sigma_{\CO_{\sigma},\mu}$, there are two generators $\pm h_{\alpha}^{\vee}$ 
of the kernel of the map:
\begin{equation}\label{eqn:MmodM1stab-to-sourceofkottwitzmap}
(M/M^1)^{\Stab(\sigma)} \rightarrow M_{\alpha}/M_{\alpha}^1.
\end{equation}
The isomorphism $M_{\alpha}/M_{\alpha}^1 \cong X^*(Z({\hM}_{\alpha})^{I_F})^{\circ}_{\Fr})$ 
from the local Langlands correspondence for characters fits into a commutative diagram:
\begin{equation}\label{diagram:Kottwitz-isom-square-diagram}
    \begin{tikzcd}
        (M/M^1)^{\Stab(\sigma)} \arrow[]{r}{\cong}\arrow[]{d}{}&X^*({\hM}_{\nu,\Fr} / K_{M,\phi})
        \arrow[]{d}{}\\
        M_{\alpha}/M^1_{\alpha} \arrow[]{r}{\cong}&X^*(Z({\hM}_{\alpha})^{I_F})^{\circ}_{\Fr})
    \end{tikzcd}
\end{equation}
in which the right-hand vertical map is induced by the natural map 
$(Z({\hM}_{\alpha})^{I_F})^{\circ} \to {\hM}_{\nu,\Fr} / K_{M,\phi}$.

Let $\pm {\widehat h}_{\alpha}^{\vee}$ denote the two generators of the kernel of the right-hand vertical map in 
diagram \eqref{diagram:Kottwitz-isom-square-diagram}. Then the isomorphism of $(M/M^1)^{\Stab(\sigma)}$ with 
$X^*({\hM}_{\nu,\Fr} / K_{M,\phi})$ identifies $\pm h_{\alpha}^{\vee}$ with $\pm {\widehat h}_{\alpha}^{\vee}$.  
Thus the image ${\hat \Sigma}_{\CO_{\sigma}}^{\vee}$ of $\Sigma_{\CO_{\sigma}}^{\vee}$ under the isomorphism is 
equal to the collection of $\pm h_{\alpha}^{\vee}$ as $\alpha$ runs over the minimal Levi's $\LM_{\alpha}$ in 
${\hat \Sigma}_{\CO_{\sigma},\mu}/ \{\pm 1\}$.

On the dual side, we let ${\hat \Sigma}_{\CO_{\sigma}}$ denote the image of $\Sigma_{\CO_{\sigma}}$ under the 
isomorphism of $\Hom((M/M^1)^{\Stab(\sigma)}, \ZZ)$ with $X_*({\hM}_{\nu,\Fr})$ constructed in 
\eqref{eqn:lattices-of-tori-isom}; for $\alpha^{\sharp} \in \Sigma_{\CO_{\sigma}}$, we let 
${\hat \alpha}^{\sharp}$ denote its image in ${\hat \Sigma}_{\CO_{\sigma}}$.

Let $V_{\alpha}$ be the one-dimensional subspace of $X^*(\bA_M) \otimes \QQ$ spanned by $\alpha$, and recall that $\{\pm \alpha^{\sharp}\}$ is the set of elements of $V_{\alpha}$ that pair to $\pm 2$ with $h_{\alpha}^{\vee}$.  To obtain a spectral side characterization of ${\widehat \alpha}^{\sharp}$, we must give a description of $V_{\alpha}$ on the spectral side.

The group $X^*(\bA_M)$ is the largest torsion-free, $W_F$-invariant quotient of $X^*(Z(\bM))$.  We have a sequence of isomorphisms:
\[
X_*({\hM}_{\nu,\Fr}) \otimes \QQ \cong X_* \big( (Z({\hM})^{I_F})_{\Fr} \big) \otimes \QQ \cong 
X_*(Z({\hM})_{W_F}) \otimes \QQ \cong X^*(\bA_M) \otimes \QQ,
\]
where the first isomorphism comes from Lemma~\ref{lemma:irreducible}, the second from the inclusion of $Z(\hM)^{I_F}$ in $Z(\hM)$, and the third from the natural isomorphism $X^*(Z(\bM)) \otimes \QQ \cong X_*(Z({\hM})) \otimes \QQ$.  We let ${\hat V}_{\alpha}$ be the subspace of $X_*(\hM_{\nu,\Fr}) \otimes \QQ$ corresponding to $V_{\alpha}$ under this sequence of identifications.  This allows us to characterize $\pm {\hat \alpha}^{\sharp}$ as the elements of ${\hat V}_{\alpha}$ that pair to $\pm 2$ with ${\hat h}_{\alpha}^{\vee}$.

We would still like a purely spectral-side description of ${\hat V}_{\alpha}$. We observe that for 
a surjective homomorphism $\bM \rightarrow \bM'$, whose kernel is contained in the center of $\bM_{\alpha}$, we have a 
corresponding morphism of split tori $\bA_M \rightarrow \bA_{M'}$, where $\bA_{M'}$ is the maximal $F$-split torus 
in the center of $\bM'$.  Then the subspace $V'_{\alpha}$ of $X^*(\bA_{M'}) \otimes \QQ$ spanned by $\alpha$ 
is identified with $V_{\alpha}$ under the natural map $X^*(\bA_{M'}) \rightarrow X^*(\bA_M)$.

On the dual side, we find that the subspace ${\hat V}_{\alpha}$ of $X_*(\hM_{\nu,\Fr})$ is the image of 
${\hat V}'_{\alpha}$ under the map:
\begin{equation}\label{eqn:lattice-map-Mnu'-to-Mnu}
X_*(\hM'_{\nu,\Fr}) \otimes \QQ \rightarrow X_*(\hM_{\nu,\Fr}) \otimes \QQ
\end{equation}
induced by the homomorphism $\hM'_{\nu} \rightarrow \hM_{\nu}$.  We thus have:

\begin{lemma} \label{lemma:spectral V}
Let $\bM'$ be the quotient of $\bM$ by $Z(\bM_{\alpha})$. Then the subspace ${\hat V}_{\alpha}$ 
is the image of \eqref{eqn:lattice-map-Mnu'-to-Mnu}. 
\end{lemma}
\begin{proof}
In this case, $X_*(\hM'_{\nu,\Fr}) \otimes \QQ$ is one-dimensional, hence equal to ${\hat V}'_{\alpha}$, 
so the claim follows by the discussion in the previous paragraph.
\end{proof}

We next obtain a spectral-side interpretation of the Hecke algebra q-parameters $q_{\alpha}$ and $q^*_{\alpha}$.  Let ${\hat X}_{\alpha}$ be the function on $X_{\LM}^{\phi}$ corresponding to $X_{\alpha}$ under the isomorphism \eqref{eqn:isom-tori} of $\CO_{\sigma}$ with $\overline{X}_{\LM}^{\phi}$. 
We may describe ${\hat X}_{\alpha}$ explicitly as follows.  Let $\phi'$ be the parameter on 
$X_{\LM}^{\phi}$ corresponding to $\sigma'$ under the identification \eqref{eqn:isom-tori} of 
$\CO_{\sigma}$ with $X_{\LM}^{\phi}$.  Then for any unramified character $\chi$ of $M$, we have 
a corresponding element $z$ of $\hM_{\nu,\Fr}$ associated to $\chi$ by the isomorphism~\ref{eqn:isom-tori}. 
The representation $\sigma' \otimes \chi$ then corresponds to the parameter $z \phi'$ that agrees with 
$\phi'$ on $I_F$ but satisfies $(z \phi') (\Fr) = {\tilde z} \phi'(\Fr)$, for some lift ${\tilde z}$ of 
$z$ to ${\hM}_{\nu}$.  (The parameter $\phi'_z$ is independent of the choice of ${\tilde z}$ up to 
${\hM}$-conjugacy.)  Then ${\hat X}_{\alpha}$ is explicitly characterized by the equation:
\begin{equation}
{\hat X}_{\alpha}(z \phi') = {\hat h}_{\alpha}^{\vee}(z).
\end{equation}
By Plancherel compatibility, the locus in $X^{\phi}_{\LM_{\alpha}}$ on which 
$L(1,\Ad_{M_{\alpha}} \varphi ) = \infty$ is given by:
\begin{equation}\label{eqn:Plancherel-cases}
\begin{cases}
\{\varphi: {\hat X}_{\alpha}(\varphi) \in \{q_{\alpha}^{\pm 1},-(q^*_{\alpha})^{\pm 1}\}\} &
\text{if} \;q^*_{\alpha} \neq 1,\\
\{\varphi: {\hat X}_{\alpha}(\varphi) \in \{q_{\alpha}^{\pm 1}\}\} &\text{otherwise.}
\end{cases}
\end{equation}
We observe that this completely determines the pair $q_{\alpha},q^*_{\alpha}$.

\subsection{Comparison of endomorphism algebras}

We are now in a position to compare the endomorphism algebra of $i_{P}^G \mathcal{W}_M^{[M,\sigma]}$ with that of $i_{P_{\nu}}^{G_{\nu}} \chi^{\un}$.  We begin by applying Theorem~\ref{thm:solleveld} and the various ideas of \S\ref{subsec:spectral-side-lattices} to the endomorphism algebra of $i_{P_{\nu}}^{G_{\nu}} \chi^{\un}$, whose affine part is associated to a root datum:
\begin{equation}\label{eqn:root-datum-1comma-nu}
\left(X^*({\hM}_{\nu,\Fr}), {\hat \Sigma}^{\vee}_{1,\nu}, X_*({\hM}_{\nu,\Fr}), {\hat \Sigma}_{1,\nu}\right)
\end{equation}
for some subsets ${\hat \Sigma}^{\vee}_{1,\nu}$ and ${\hat \Sigma}_{1,\nu}$ of $X^*({\hM}_{\nu,\Fr})$ and 
$X_*({\hM}_{\nu,\Fr})$, respectively.

The lattices appearing in this root datum \eqref{eqn:root-datum-1comma-nu} differ only by $K_{M,\phi}$ from
the lattices in the root datum
\begin{equation}\label{eqn:root-datum-Osigma}
\big( X^*({\hM}_{\nu,\Fr} / K_{M,\phi} ), {\hat \Sigma}^{\vee}_{\CO_{\sigma}}, X_*({\hM}_{\nu,\Fr} / K_{M,\phi}), 
{\hat \Sigma}_{\CO_{\sigma}} \big)
\end{equation}
giving the affine part of the endomorphism algebra of $i_{P}^G \mathcal{W}_M^{[M,\sigma]}$. 
The comparison is facilitated by \eqref{eq:K-central} and Proposition
\ref{prop:4.actionXonSGnu} (1), which tell us that $K_{M,\phi}$ is central in $\hG_\nu$ and 
in particular lies in the kernel of every element of ${\hat \Sigma}^{\vee}_{1,\nu}$.

We want to show 
that ${\hat \Sigma}^{\vee}_{1,\nu} = {\hat \Sigma}^{\vee}_{\CO_{\sigma}}$ and 
$\hat \Sigma_{1,\nu} = {\hat \Sigma}_{\CO_{\sigma}}$, 
that the group $R(\CO_1)$ is trivial, and that the Hecke 
algebra q-parameters attached to simple roots of the datum \ref{eqn:root-datum-1comma-nu} are the same as those 
attached to the root datum of Proposition \ref{prop:solleveld}. 
Then the $K_{M,\phi}$-invariant part of the endomorphism algebra of $i_{P_{\nu}}^{G_{\nu}} \chi^{\un}$ 
is isomorphic to the affine part of the endomorphism algebra of $i_{P}^G \mathcal{W}_M^{[M,\sigma]}$; 
that is, we would then have an isomorphism: 
\begin{equation}
\End_G(i_{P}^G \mathcal{W}_M^{[M,\sigma]}) \cong 
\End_{G_{\nu}}(i_{P_{\nu}}^{G_{\nu}} \chi^{\un})^{K_{M,\phi}} \rtimes \CC[R(\CO_{\sigma})],
\end{equation}
and the isomorphism of Theorem~\ref{thm:automorphic reduction} would then follow if we can identify $R(\CO_{\sigma})$ with 
${\widehat R}_{G,\phi}$ (Lemma \ref{lem:identifying-RO-with-hatR}).

\begin{lemma}\label{lem:min-Levi-lemma}
Let $\alpha \in {\hat \Sigma}_{\red}/ \{\pm 1\}$ correspond to a minimal Levi for $G$.  The following are equivalent:
\begin{enumerate}
\item The root $\alpha$ is an element of ${\hat \Sigma}_{\CO_{\sigma},\mu}/ \{\pm 1\}$.
\item The quotient $\Lie({\hM}_{\alpha})/\Lie(\hM)$ contains an $\Ad \, \phi(I_F)$-invariant subspace.
\item The group $(\hM_{\alpha})_{\nu}:=Z_{\widehat{M}_{\alpha}}(\phi(I_F))$ strictly contains $\hM_{\nu}$.
\end{enumerate}
\end{lemma}
\begin{proof}
``(2)$\Longleftrightarrow$(3)'': 
Consider the space $\big( \Lie({\hM}_{\alpha})/\Lie(\hM) \big)^{\Ad \, \phi(I_F)}$. Since 
$\Lie({\hM}_{\alpha})$ and $\Lie(\hM)$ are vector spaces of characteristic zero and $\phi(I_F)$ is finite, we may view 
$\big( \Lie({\hM}_{\alpha})/\Lie(\hM) \big)^{\Ad \, \phi(I_F)}$ as a quotient of $\Lie({\hM}_{\alpha})^{\Ad \, \phi(I_F)}$ 
by $\Lie(\hM)^{\Ad \, \phi(I_F)}$; note that these two spaces are isomorphic to $\Lie((\hM_{\alpha})_{\nu})$ and 
$\Lie(\hM_{\nu})$, respectively.  Thus, if (2) holds then there is an element of $\Lie((\hM_{\alpha})_{\nu})$ 
not contained in $\Lie(\hM_{\nu})$, proving (3).  The reverse implication is also clear.

``(1)$\Longrightarrow$(2)'': 
Assume (1) holds.  Then there exists a parameter $\varphi$ in $X_{\LM}^{\phi}$ such that 
$L(1,\Ad_{\hM_{\alpha}} \varphi ) = \infty$. Then the subspace $H^0(W_F,\Ad_{\hM_{\alpha}} \varphi(1))$ of 
$\Lie(\hM_{\alpha})$ on which $\Ad \, \varphi$ acts via the cyclotomic character is nonzero. Since $\varphi$ and 
$\phi$ have the same restriction to $I_F$, this linear subspace
is fixed under the action of $\Ad \, \phi(I_F)$.  We may thus regard it as a subspace of $\Lie((\hM_{\alpha})_{\nu})$. 
Since $\hM_{\nu}$ is a torus, and $\varphi(\Fr)$ differs from $\phi(\Fr)$ by an element of $\hM_{\nu}$, the action of 
$\Ad \, \varphi(\Fr)$ on $\hM_{\nu}$ has finite order.  Thus $H^0(W_F,\Ad_{\hM_{\alpha}} \varphi(1))$ cannot be 
contained in $\Lie(\hM_{\nu})$, so its image in the quotient $\Lie((\hM_{\alpha})_{\nu})/\Lie(\hM_{\nu})$ is nonzero. 
But this quotient embeds in the quotient $\Lie(\hM_{\alpha})/\Lie(\hM)$, hence (2) follows.

``(1)$\Longleftarrow$(2)'':
Conversely, suppose that the space $(\Lie({\hM}_{\alpha})/\Lie(\hM))^{\Ad \, \phi(I_F)}$ is nonzero.  We will construct 
a $\varphi$ such that $\Ad \, \varphi$ acts on this space via the cyclotomic character; this is equivalent to 
requiring that $\Ad \, \varphi$ acts on a subspace of $\Lie(({\hM}_{\alpha})_{\nu})$ via the cyclotomic character.  
Note that the action of $\Ad \, \varphi$ on the latter factors through $W_F/I_F$, and that $\Ad \, \varphi(\Fr)$ 
acts via $m \phi(\Fr)$ for an element $m$ of $\hM_{\nu}$ that we are free to choose.  

The group $\hM_{\nu}$ is a torus that acts on $\Lie(({\hM}_{\alpha})_{\nu})/\Lie(\hM_{\nu})$; 
all of the weights of this action are nontrivial.  Pick a weight $\gamma$ for this action, let 
$\gamma_1, \dots, \gamma_r$ be the orbit of $\gamma$ under the action of $\phi(\Fr)$, and let 
$\Lie((\hM_{\alpha})_{\nu})_{[\gamma]}$ denote the direct sum of the weight spaces 
$\Lie((\hM_{\alpha})_{\nu})_{\gamma_i}$ for this action.  Then $\Lie((\hM_{\alpha})_{\nu})_{[\gamma]}$ is stable 
under the adjoint action of $\phi(\Fr)$, and, since $\phi(\Fr)$ preserves a pinning on ${\hG}_{\nu}$, it follows 
that this space is isomorphic to $\Ind_{W_{F_r}}^{W_F} 1$ when considered as a $W_F$-representation via $\Ad \, \phi$, 
where $F_r$ is the unramified extension of $F$ of degree $r$.  If we instead consider this space as a 
$W_F$-representation via $\Ad \, (m \phi)$, it is isomorphic to $\Ind_{W_{F_r}}^{W_F} \chi_{\gamma,m}$, where 
$\chi_{\gamma,m}$ is the unramified character of $W_{F_r}$ sending $\Fr$ to the product 
$\gamma_1(m)\gamma_2(m) \dots \gamma_r(m)$.  By Frobenius reciprocity, this representation contains a copy of 
the cyclotomic character if, and only if, the character $\chi_{\gamma,m}$ is the cyclotomic character; 
since we may choose $m$ freely we can certainly arrange for this to be so.
\end{proof}

Lemma \ref{lem:min-Levi-lemma} allows us to compare minimal Levi's for $G$ and $G_{\nu}$.  In particular, we have:

\begin{prop} \label{prop:levis}
Let $\alpha \in {\hat \Sigma}_{\red}/\{\pm 1\}$ correspond to a Levi $\LM_{\alpha}$ of $\LG$ minimally containing
$\LM$. Then $(\LM_{\alpha})_{\nu} := ({\hM}_{\alpha})_{\nu} \rtimes W_F$ is a Levi of $\LG_{\nu}$ minimally 
containing $\LM_\nu$ if $\alpha$ 
lies in ${\hat \Sigma}_{\CO,\mu}/\pm 1$, and is equal to ${\LM}_{\nu}$ otherwise.\\
Let ${\hat \Sigma}_{\nu,\red}/ \{\pm 1\}$ denote the set of Levi's of $\LG_{\nu}$ minimally containing $\LM_{\nu}$.  
The induced map ${\hat \Sigma}_{\CO_{\sigma},\mu}/ \{\pm 1\} \rightarrow {\widehat \Sigma}_{\nu,\red}/ \{\pm 1\}$ 
is a bijection.
\end{prop}
\begin{proof}
It is clear that $({\LM}_{\alpha})_{\nu}$ contains $\LM_{\nu}$ and that it is a Levi subgroup of $\LG_{\nu}$.  
Moreover, the center of $({\LM}_{\alpha})_{\nu}$ contains $Z({\hM}_{\alpha})^{W_F}$, which is isogenous to 
$Z((\hM_{\alpha})^{I_F})^{\circ}_{\Fr}$ and thus has corank one inside ${\hM_{\nu}}^{\Fr} = Z(\LM_{\nu})$.  
As $({\LM}_{\alpha})_{\nu}$ is the centralizer of its center we find that it is equal to either $\LM_{\nu}$ 
or a Levi of $\LG_{\nu}$ minimally containing $\LM_{\nu}$.
By Lemma \ref{lem:min-Levi-lemma}, the latter holds if, and only if, $\alpha$ lies in ${\hat \Sigma}_{\CO,\mu}$.

Now let $\alpha,\beta \in {\hat \Sigma}_{\CO,\mu}/ \{\pm 1\}$, and suppose that 
$(\LM_{\alpha})_{\nu} = (\LM_{\beta})_{\nu}.$ Then we have equalities:
$$Z(\LM_{\alpha}) = Z((\LM_{\alpha})_{\nu}) = Z((\LM_{\beta})_{\nu}) = Z(\LM_{\beta}),$$
so that $\LM_{\alpha} = \LM_{\beta}$.

Finally, let $\alpha$ be an element of ${\hat \Sigma}_{\nu,\red}/ \{\pm 1\}$, so that we have a minimal Levi
$(\LM_{\nu})_{\alpha}$ of $\LM_{\nu}$.  Let $\LM_{\alpha}$ be the centralizer, in $\LG$, of 
$Z((\LM_{\nu})_{\alpha})$.  Since this torus has corank one in $Z(\LM_{\nu}) = Z(\LM)$, we find that $\LM_{\alpha}$ 
is either equal to $\LM$ or is a minimal Levi containing it.  But since $(\LM_{\alpha})_{\nu}$ clearly contains
$(\LM_{\nu})_{\alpha}$, we must have that $\LM_{\alpha}$ is a minimal Levi of $\LG$ containing $\LM$, with
$(\LM_{\alpha})_{\nu} = (\LM_{\nu})_{\alpha}$.  Thus $\alpha$ lies in ${\hat \Sigma}_{\CO,\mu}/ \{\pm 1\}$ by 
Lemma \ref{lem:min-Levi-lemma}.
\end{proof}

From Proposition \ref{prop:levis} it is not hard to deduce the following desired equalities: 
\begin{prop} \label{prop:root equality}
Let $\alpha$ be an element of ${\hat \Sigma}_{\CO_{\sigma},\mu}/ \{\pm 1\}$ and let $\beta$ be the corresponding 
element of ${\hat \Sigma}_{\nu,\red}/ \{\pm 1\}$.  Then we have 
$\pm {\hat h}_{\alpha}^{\vee} = \pm {\hat h}_{\beta}^{\vee}$ and 
$\pm {\hat \alpha}^{\sharp} = \pm {\hat \beta}^{\sharp}.$  In particular, 
${\hat \Sigma}^{\vee}_{1,\nu} = {\hat \Sigma}^{\vee}_{\CO_{\sigma}}$ and 
 ${\hat \Sigma}_{1,\nu} = {\hat \Sigma}_{\CO_{\sigma}}$.
\end{prop}
\begin{proof}
The group $Z(\hM_\alpha)^{I_F} = Z(\hM_\alpha) \cap (\hM_\alpha)_\nu$ is contained in $Z((\hM_\alpha)_\nu) =
Z( (\hM_\nu)_\beta)$. Both sides have the same dimension (i.e.~$\dim \hM_\nu - 1$), hence they have the same 
neutral components. This gives a surjection $(Z(\hM_{\alpha})^{I_F})^{\circ}_{\Fr} \to 
Z((\hM_{\nu})_{\beta})_{\Fr}^{\circ}$ with finite kernel.
We have the following commutative diagram:
\begin{equation}\label{eq:6.diagonal}
\begin{tikzcd}
X^*(\hM_{\nu,\Fr} / K_{M,\phi})\arrow[]{r} \arrow[]{d}{} 
& X^*(\hM_{\nu,\Fr})\arrow[]{d}{} \arrow[]{dl}{}\\
X^*((Z(\hM_{\alpha})^{I_F})^{\circ}_{\Fr})
& X^*(Z((\hM_{\nu})_{\beta})_{\Fr}^{\circ}) \arrow[]{l}
\end{tikzcd}
\end{equation}
where the horizontal maps are injective. The elements $\pm {\hat h}_{\alpha}^{\vee}$ and 
$\pm {\hat h}_{\beta}^{\vee}$ are the generators of the kernels of the left-hand and 
right-hand vertical maps in \eqref{eq:6.diagonal}, respectively. By the injectivity of 
the lower horizontal map, $\pm {\hat h}_{\beta}^{\vee}$ are also the generators of the 
diagonal map.

By Proposition \ref{prop:4.actionXonSGnu}(1), $K_{M,\phi}$ can be represented by elements 
of $Z(\hG_\nu)^\circ_\Fr$, and such elements lie in the kernel of every character of
$\hM_{\nu,\Fr}$ which is a rational multiple of a root. Hence the elements 
$\pm {\hat h}_{\alpha}^{\vee}$ can also be characterized as the generators of the  
diagonal map in \eqref{eq:6.diagonal}, and 
$\pm {\hat h}_{\alpha}^{\vee} = \pm {\hat h}_{\beta}^{\vee}$.

It follows easily from Lemma \ref{lemma:spectral V} that $\hat V_\alpha$ and $\hat V_\beta$ 
are identified via the isomorphism 
\[
X^*(\hM_{\nu,\Fr} / K_{M,\phi}) \otimes \QQ \to X^*(\hM_{\nu,\Fr}) \otimes \QQ ,
\]
so $\pm \hat \alpha^\sharp$ is identified with $\pm \hat \beta^\sharp$.
\end{proof}
We also note the following:
\begin{lemma} \label{lemma:R-trivial-unipotent}
The group $R(\CO_1)$ is trivial. In particular, 
$\End_{G_{\nu}}(i_{P_{\nu}}^{G_{\nu}} \chi^{\un})\simeq \mathcal{H}_{\CO_1}$ 
is an affine Hecke algebra associated to the root datum \eqref{eqn:root-datum-1comma-nu}.
\end{lemma}
\begin{proof}
We may identify $R(\CO_1)$ with the quotient of $\bW(M_{\nu},1)$ by the subgroup generated by simple reflections 
$s_{\alpha}$, for $\alpha \in \Sigma_{\CO_1,\mu}$.  By Proposition~\ref{prop:levis}, this is all of $\bW(M_{\nu},1)$, 
so $R(\CO_1)$ is trivial. The last statement then follows from applying Theorem \ref{thm:solleveld} to the block 
$\Rep (G_\nu)_{[M_\nu,1]}$ of $\Rep(G_\nu)$.
\end{proof}

The upshot of Proposition~\ref{prop:root equality} is that the root data associated to the two endomorphism 
algebras of $i_{P_{\nu}}^G \chi^{\un}$ and $i_{P}^G \mathcal{W}_M^{[M,\sigma]}$ differ only by
dividing out $K_{M,\phi}$. Moreover, the set of simple roots $\Delta_{\CO_{\sigma}}$ (coming from the parabolic 
$P$) is identified with the set of simple roots (for the root datum associated to 
$i_{P_{\nu}}^G \chi^{\un}$) corresponding to the parabolic $P_{\nu}$.  We next show that the 
q-parameters attached to these two sets of simple roots coincide.

Let $\alpha$ be an element of $\Delta_{\CO,\mu}$, corresponding to a Levi $\LM_{\alpha}$ of $\LG$ minimally 
containing $\LM$. By Proposition \ref{prop:levis} we have a corresponding minimal Levi $(\LM_{\nu})_{\alpha}$ of 
$\LG_{\nu}$.  Attached to this choice of $\alpha$, we have two sets of q-parameters:
the parameters $q_{\alpha}, q^*_{\alpha}$, that determine the locus of unramified twists of $\phi$ for which 
$L(1, \Ad_{\hM_{\alpha}} \phi) = \infty$, and the parameters $q_{\alpha,\nu}, q_{\alpha,\nu}^*$, that determine 
the locus of L-parameters $\tau$ in $X^1_{\LM_{\nu}}$ for which $L(1, \Ad_{(\hM_{\alpha})_{\nu}} \tau) = \infty$.

Recall that we have a natural L-homomorphism $\iota_{\phi}: \LM_{\nu} \rightarrow \LM$ from 
\eqref{eqn:L-hom-iota-phi}; composing an L-parameter with this L-homomorphism induces a surjection 
$X^1_{\LM_{\nu}} \rightarrow X^{\phi}_{\LM}$, as in \eqref{eqn:moduli-reduction-thm}, that takes the trivial 
L-parameter to $\phi$. We have the following result. 
\begin{lemma}\label{lem:isom-H0}
For any L-parameter $\tau$ in $X^1_{\LM_{\nu}}$, there is an equality of meromorphic functions:
\[
L \big(s, \Ad_{(\hM_{\alpha})_{\nu}} \tau \big) = L \big( s, \Ad_{\hM_{\alpha}} \varphi \big) 
\quad\text{ for }\; s \in \CC,
\]
where $\varphi = \iota_{\phi} \circ \tau$. In particular, $\iota_\phi$ provides a bijection 
$\widehat{\mathrm{Pol}}_{L_\nu,1} \xrightarrow{\sim} \widehat{\mathrm{Pol}}_{L,\phi}$.
\end{lemma}
\begin{proof}
By the definitions of L-functions and of the L-homomorphism $\iota_\phi$:
\begin{align}\label{eq:6.L}
\begin{split}
L \big( s, \Ad_{\hM_{\alpha}} \varphi \big) &= 
\det \big( 1 - q^{-s} \Ad(\varphi (\Fr)) , \Lie (\hM_\alpha )^{\phi (I_F)} \big)^{-1} \\
&= \det \big( 1 - q^{-s} \Ad(\tau (\Fr)) , \Lie (\hM_\nu)_\alpha \big)^{-1} =
L \big( s, \Ad_{(\hM_\nu)_{\alpha}} \tau \big) .
\end{split}
\end{align}
The sets $\widehat{\mathrm{Pol}}_{L_\nu,1}$ and $\widehat{\mathrm{Pol}}_{L,\phi}$ are the loci where the 
functions in \eqref{eq:6.L} with $s=1$ have poles. By construction, these are matched by $\iota_\phi$.
\end{proof}

\begin{cor} \label{cor:parameters equality}
We have $q_{\alpha} = q_{\alpha,\nu}$ and $q^*_{\alpha} = q^*_{\alpha,\nu}$.
\end{cor}
\begin{proof}
By construction, the functions ${\hat X}_{\alpha}$ and ${\hat X}_{\alpha,\nu}$ are both translates of the character 
${\hat h}_{\alpha}^{\vee}$. Lemma \ref{lem:isom-H0}, combined with the description in \eqref{eqn:Plancherel-cases} 
of the loci where $L(s=1, \Ad_{\hM_{\alpha}} \varphi)$ has a pole, shows that the loci where 
${\hat X}_{\alpha} = q_{\alpha}^{\pm 1}$ and ${\hat X}_{\alpha,\nu} = q_{\alpha,\nu}^{\pm 1}$ agree.  This is 
already enough to conclude that the functions ${\hat X}_{\alpha}$ and ${\hat X}_{\alpha,\nu}$ coincide, and thus 
that $q_{\alpha} = q_{\alpha,\nu}$.

The same reasoning also shows that $q^*_{\alpha} \neq 1$ if, and only if, $q^*_{\alpha,\nu} \neq 1$, and in this 
case the locus where ${\hat X}_{\alpha} = -(q^*_{\alpha})^{\pm 1}$ agrees with the locus where 
${\hat X}_{\alpha,\nu} = -(q^*_{\alpha,\nu})^{\pm 1}$; from this we conclude that $q^*_{\alpha} = q^*_{\alpha,\nu}$. 
\end{proof}

Next we check that the base-point $\sigma'$ in Theorem \ref{thm:solleveld} can be chosen as $\sigma$, 
so that it corresponds to the trivial $M_\nu$-representation and to the trivial L-parameter for $M_\nu$ via $\iota_\phi$ and the setup from Section \ref{sec:conjecture}. 

\begin{lemma}\label{lem:6.basepoints}
We have $\mu_{\alpha,\sigma}(\sigma) = 0$ for all $\alpha \in \Delta_{\CO_\sigma}$.

If we rewrite \eqref{eq:6.7} and \eqref{eqn:Silberger-Solleveld} with $\sigma$ in the role of $\sigma'$,
then $q_\alpha \geq q^*_\alpha$ for all $\alpha \in \Delta_{\CO_\sigma}$.
\end{lemma}
\begin{proof}
For $\hG_\nu$, we are in the setting of unipotent L-parameters for a 
quasi-split group. The poles of $L(s,\mathrm{Ad}_{(\hM_\nu)_\alpha} \phi)$ at
$s = 1$ and of the function $\mu_{M_\nu,\alpha}$ are known; see for example \cite{shahidi}.
Since $\phi$ preserves a pinning, there is a unique element $z_\alpha \in \hM_{\nu,\Fr}$ such that 
$L(1,\mathrm{Ad}_{(\hM_\nu)_\alpha} z_\alpha \phi) = \infty$, and $z_\alpha$ corresponds to a
$\chi_\alpha \in \Xnr (M)$ which is trivial on $Z(M_{\nu,\alpha})$ and takes values in $\RR_{>0}$. 
Then we also have $\mu_{M_\nu,\alpha} (\chi_\alpha) = \infty$. Silberger's formula 
\eqref{eqn:Silberger-Solleveld} shows that $\mu_{M_\nu,\alpha} (1) = 0$. 

By \cite[\S 1]{Borel-Iwahori}, $\chi_\alpha$ is determined by its value at
the version of $h_\alpha^\vee$ for $M_\nu$, which is
\[
q_{\nu,\alpha} = |I_\nu s_\alpha I_\nu / I_\nu|. 
\]
If $q^*_{\nu,\alpha} \neq 1$, then there is a simple affine reflection
$s_\alpha^*$ parallel to $s_\alpha$, and
\[
q^*_{\nu,\alpha} = |I_\nu s^*_\alpha I_\nu / I_\nu| \leq q_{\nu,\alpha}.
\]
Using the formula \eqref{eqn:Silberger-Solleveld}, this identifies all poles of $\mu_{M_\nu,\alpha}$.
In view of Corollary \ref{cor:parameters equality}, also $q_\alpha \geq q^*_\alpha$. The conventions on
base-points in the proof of Corollary \ref{cor:parameters equality} mean that we are now using
\eqref{eq:6.7} and \eqref{eqn:Silberger-Solleveld} with $\sigma$ in the role of $\sigma'$.

The above also determines the locus of $\phi$ where $L(s=1,\mathrm{Ad}_{\hM_{\nu,\alpha}} \phi)$ has a pole.
By Lemma \ref{lem:isom-H0} this is also the locus of $\phi$ where
$L(1,\mathrm{Ad}_{\hM_{\alpha}} \phi)$ has a pole.
By the twist and Plancherel compatibilities of $\phi$ and $\sigma$, the pole locus of $\phi$ corresponds to 
the pole locus of $\mu_{M,\alpha}$. Hence $\mu_{M,\alpha}(\chi_\alpha \otimes \sigma) = \infty$ and
the formula \eqref{eqn:Silberger-Solleveld} shows that $\mu_{\alpha,\sigma} (\sigma) = 0$.
\end{proof}

Putting all the above preparations together, we obtain: 
\begin{thm} \label{thm:rep thy isom}
Given Langlands compatible $\phi$ and $\sigma$, such that $\phi$ preserves a pinning of $\hat G_\nu$,
there is a canonical isomorphism:
\[
\End \big( i_{P}^G \mathcal{W}_M^{[M,\sigma]} \big) \cong 
\End \big( i_{P_{\nu}}^{G_{\nu}} \chi^{\un} \big)^{K_{M,\phi}} \rtimes \CC [R(\CO_{\sigma})].
\]
\end{thm}
\begin{proof}
By Theorem \ref{thm:solleveld} and Lemma \ref{lem:6.basepoints}, the left-hand side is canonically isomorphic 
to $\CH_{\CO_{\sigma}} \rtimes \CC [R(\CO_{\sigma})]$. By Propositions \ref{prop:levis}, \ref{prop:root equality}, 
Lemma \ref{lemma:R-trivial-unipotent} and Corollary \ref{cor:parameters equality}, there is a canonical isomorphism
\[
\End \big( i_{P_{\nu}}^{G_{\nu}} \chi^{\un} \big) \cong \CH_{\CO_1} ,
\]
where $\CH_{\CO_1}$ has the same roots and the same q-parameters as 
$\CH_{\CO_\sigma}$. In \eqref{eqn:root-datum-1comma-nu} and \eqref{eqn:root-system-for-HOsigma} we saw that
the lattices in the root data for these affine Hecke algebras differ only by dividing out $K_{M,\phi}$. We let 
$K_{M,\phi}$ act on $\CH_{\CO_1}$ via translations on $\Xnr (M_\nu)$. We then have natural algebra isomorphisms
\begin{align*}
\End(i_{P_{\nu}}^{G_{\nu}} \chi^{\un})^{K_{M,\phi}} \cong \CH_{\CO_1}^{K_{M,\phi}} &
= \CC [\Xnr (M_\nu)]^{K_{M,\phi}} \otimes_\CC \CH (W(G_\nu,M_\nu),q) \\
& \cong \CC [\Xnr (M_\nu) / K_{M,\phi}] \otimes_\CC \CH (\bW_{\CO_\sigma},q) = \CH_{\CO_\sigma} .
\end{align*}
Here $\otimes_\CC$ means a tensor product as vector spaces, not as algebras; it expresses the Bernstein 
presentation of an affine Hecke algebra.
\end{proof}

We note that the action of $K_{M,\phi}$ in Theorem \ref{thm:rep thy isom} is the same as the action we found in Proposition \ref{prop:4.actionXonSGnu}. The action of $R(\CO_\sigma)$ on 
$\CH_{\CO_\sigma}$ is described in \eqref{eq:6.psir} and \eqref{eq:6.chir}, which also shows 
how it acts on $\End \big( i_{P_{\nu}}^{G_{\nu}} \chi^{\un} \big)^{K_{M,\phi}}$\!.

\subsection{Compatibilities}

Theorem~\ref{thm:rep thy isom} is close to Theorem \ref{thm:automorphic reduction}. 
To finish the proof, we must still verify the compatibility of the isomorphism of Theorem~\ref{thm:rep thy isom} with parabolic induction 
and Whittaker data.  In addition, we must identify $R(\CO_{\sigma})$ with ${\widehat R}_{G,\phi}$, in a manner 
that carries the action of ${\widehat R}_{G,\phi}$ on $\End(\CS^1_{G_{\nu}})$ to the action of $R(\CO_{\sigma})$ on 
$\End(i_{P_{\nu}}^{G_{\nu}} \chi^{\un})$, via the isomorphism of these two endomorphism algebras arising from Theorem~\ref{thm:unipotent categorical}.

We first discuss the compatibility of our isomorphisms with parabolic induction. Recall that we identify 
$\CC \big[ X^* \big( (M/M^1)^{\Stab(\sigma)} \big) \big]$ with $\CC[\CO_\sigma]$ via the base-point $\sigma$ of
$\CO_\sigma$. Bernstein's presentation \cite{BD} of the affine Hecke algebra $\CH_{\CO_{\sigma}}$ gives an isomorphism:
\begin{equation}
\CH_{\CO_{\sigma}} \cong \CC[\CO_{\sigma}] \otimes_{\CC} \CH(\bW_{\CO_{\sigma}},q)
\end{equation}
where $\CH(\bW_{\CO_{\sigma}},q)$ denotes the Iwahori-Hecke algebra associated to the finite Weyl group 
$\bW_{\CO_{\sigma}}$ and the q-parameters $q_{\alpha}$. The parameters $q^*_\alpha$ are encoded via
the cross relations in the tensor product. 

Let $\bQ$ be a standard parabolic subgroup of $\bG$, with Levi subgroup $\bL$ containing $\bM$. By construction, 
$\bW_L (M,\sigma) \subset \bW (M,\sigma)$, and this restricts to inclusions 
\begin{equation}\label{eqn:inclusion-RO}
\bW_{\CO_\sigma,L} \subset \bW_{\CO_\sigma} \quad \text{and} \quad R(\CO_{\sigma,L}) \subset R(\CO_\sigma). 
\end{equation}
Then we have a similar algebra isomorphism:
\begin{equation}\label{eq:6.26}
\CH_{\CO_{\sigma},L} \cong \CC[\CO_{\sigma}] \otimes_{\CC} \CH(\bW_{\CO_{\sigma},L}, q),
\end{equation}
where $\CH_{\CO_{\sigma},L}$ is the Hecke algebra appearing in the description of $\End \big( i_{P \cap L}^L
\mathcal{W}_M^{[M,\sigma]} \big)$--as in \cite{Heiermann,solleveld-endomorphisms}--and 
$\bW_{\CO_{\sigma},L}$ is the Weyl group of the associated root datum.  There is a natural map from 
$\CH(\bW_{\CO_{\sigma},L},q)$ to $\CH (\bW_{\CO_{\sigma}},q)$, taking a Hecke operator $T_{s_\alpha}$ for a 
simple reflection $s_\alpha$ in $\bW_{\CO_{\sigma},L}$ to the Hecke operator corresponding to the same simple 
reflection in $\bW_{\CO_{\sigma}}$. This, combined with the identity on $\CC[\CO_\sigma]$, gives us a canonical 
injective algebra homomorphism $\ind_L^G: \CH_{\CO_{\sigma},L} \rightarrow \CH_{\CO_{\sigma}}$. 
The parabolic induction functor induces an injection
\begin{equation}
i_{Q}^G: \End \big( i_{P \cap L}^L \mathcal{W}_M^{[M,\sigma]} \big) \rightarrow 
\End \big( i_{P}^G \mathcal{W}_M^{[M,\sigma]} \big) ,
\end{equation}
which has the following description on the level of Hecke algebras:

\begin{prop} \label{prop:Hecke induction}
Let $\bQ$ be a standard parabolic subgroup of $\bG$, with Levi subgroup $\bL$ containing $\bM$. 
Then there is a commutative diagram:
\begin{equation}
\begin{tikzcd}
\End_L \big( i_{P \cap L}^L \mathcal{W}_M^{[M,\sigma]} \big) \arrow[]{r}{\cong}\arrow[]{d}[swap]{i_{Q}^G} 
& \CH_{\CO_{\sigma},L} \rtimes \CC[R(\CO_{\sigma,L})]\arrow[]{d}{\ind_L^G\rtimes\mathrm{incl}}\\
\End_G \big( i_{P}^G \mathcal{W}_M^{[M,\sigma]} \big) \arrow[]{r}{\cong}
& \CH_{\CO_{\sigma}} \rtimes \CC[R(\CO_{\sigma})]
\end{tikzcd}
\end{equation}
where the right-hand vertical map is given by $\ind_L^G$ on the first factor and the natural inclusion \eqref{eqn:inclusion-RO} of
$R(\CO_{\sigma,L})$ into $R(\CO_{\sigma})$ on the second.
\end{prop}
\begin{proof}
This is implicit in the constructions of~\cite{solleveld-endomorphisms,OSgeneric}.  We thus provide a sketch of the argument, 
which essentially follows by tracing through the constructions \textit{loc.cit}.

For conciseness, let $\mathcal{E}$ and $\mathcal{E}_L$ denote the algebras 
$\End_G \big( i_{P}^G \mathcal{W}_M^{[M,\sigma]} \big)$ and 
$\End_L \big( i_{P \cap L}^L \mathcal{W}_M^{[M,\sigma]} \big)$, respectively.  
The algebra $\mathcal{E}$ has a $\CC[\CO_{\sigma}]$-basis $\{\CT_w'\}$ for 
$w \in \bW(M,\sigma) = \bW_{\CO_{\sigma}} \rtimes R(\CO_{\sigma})$ described in the proof 
of~\cite[Theorem A.1]{OSgeneric}.  Similarly, we have a $\CC[\CO_{\sigma}]$-basis $\CT'_{w,L}$ of $\mathcal{E}_L$, 
indexed by $w \in \bW_{\CO_{\sigma},L} \rtimes R(\CO_{\sigma,L})$. Moreover, for $w \in R(\CO_{\sigma})$, the 
construction in~\cite[\S 2]{solleveld-principal} explains how to renormalize $\CT'_w$ to obtain an element $N_w$ in 
such a way that the isomorphism of $\mathcal{E}$ with $\CH_{\CO_{\sigma}} \rtimes \CC[R(\CO_{\sigma})]$, sends the
elements $\CT'_w$ for $w \in \bW_{\CO_{\sigma}}$ to the corresponding Hecke operators in $\CH_{\CO_{\sigma}}$, whereas 
the $N_w$ for $w\in R(\CO_{\sigma})$ are sent to the corresponding elements of the group ring $\CC[R(\CO_{\sigma}]$.  
A similar statement holds for the corresponding elements $\CT'_{w,L}$ and $N_{w,L}$ of $\mathcal{E}_L$. 

The content of this proposition is the claim that the map $i_{Q}^G: \mathcal{E}_L \rightarrow \mathcal{E}$ 
takes $\CT'_{w,L}$ to $\CT'_w$ for all $w\in\bW_{\CO_{\sigma},L}$, and $N_{w,L}$ to $N_w$ for all 
$w\in R(\CO_{\sigma,L})$.  Each such operator is obtained by a series of normalizations (taking place in 
\cite[\S5]{solleveld-endomorphisms} and \cite[\S2]{solleveld-principal}) from rational endomorphism algebras of
$i_{P}^G \mathcal{W}_M^{[M,\sigma]}$ that are constructed from intertwining operators in 
\cite[\S4]{solleveld-endomorphisms}.  We omit the precise details of these normalizations, noting only that for 
a fixed $w$ in $\bW_{\CO_{\sigma,L}}$ (resp.~$R(\CO_{\sigma,L})$), the normalizations needed are the same for 
$\CT'_w$ and $\CT'_{w,L}$ (resp.~$N_w$ and $N_{w,L}$).  The claim thus reduces to the invariance of intertwining 
operators under parabolic induction, which is well-known from \cite{Wal}.
\end{proof}

Applying Proposition \ref{prop:Hecke induction} and Theorem \ref{thm:rep thy isom} to the Levi subgroup $L$ of $G$ 
and the Levi subgroup $L_{\nu}$ of $G_{\nu}$, we obtain:

\begin{prop} \label{prop:parabolic compatibility}
Let $\bQ$ be a standard parabolic subgroup of $\bG$, with Levi subgroup $\bL$ containing $\bM$. 
There is a commutative diagram:
\begin{equation}
\begin{tikzcd}
\End_L \big( i_{P \cap L}^L \mathcal{W}_M^{[M,\sigma]} \big) \arrow[]{r}{\cong}\arrow[]{d}[swap]{i_{Q}^G}
& \End_{L_{\nu}} \big( i_{P_{\nu} \cap L_{\nu}}^{L_{\nu}} \chi^{un} \big)^{K_{M,\phi}} \rtimes 
\CC[R(\CO_{\sigma,L})] \arrow[]{d}{i_{Q_{\nu}}^{G_{\nu}}\rtimes\mathrm{incl}}\\
\End_G \big( i_{P}^G \mathcal{W}_M^{[M,\sigma]} \big) \arrow[]{r}{\cong}  
& \End_{G_{\nu}} \big( i_{P_{\nu}}^{G_{\nu}} \chi^{\un} \big)^{K_{M,\phi}} \rtimes \CC[R(\CO_{\sigma})]
\end{tikzcd}
\end{equation}
in which the right-hand vertical map is given by $i_{Q_{\nu}}^{G_{\nu}}$ on the first factor and 
the natural inclusion of $R(\CO_{\sigma,L})$ into $R(\CO_{\sigma})$ on the second.
\end{prop}

The compatibility with Whittaker data is an easy consequence of \cite{solleveld-endomorphisms}. 
Let $\St$ denote the one-dimensional Steinberg module of $\CH(\bW_{\CO_{\sigma}},q)$, extended to 
a module of $\CH(\bW_{\CO_{\sigma}},q) \rtimes \CC [R(\CO_\sigma)]$ on which $R(\CO_\sigma)$ acts via
the character $r \mapsto \det_{(M /M^1) \otimes \QQ} (r)$. Then by \cite[Theorem 6.2]{solleveld-endomorphisms}, 
we have an isomorphism of $\CH_{\CO_{\sigma}} \rtimes \CC[R(\CO_{\sigma})]$-modules:
\begin{equation}\label{eqn:Hom-Whit-Whit}
\Hom \big( i_{P}^G \mathcal{W}_M^{[M,\sigma]}, \cW_G^{[M,\sigma]} \big) \cong 
\ind_{\CH(\bW_{\CO_{\sigma}},q) \rtimes \CC [R(\CO_\sigma)]}^{\CH_{\CO_{\sigma}} \rtimes \CC[R(\CO_{\sigma})]} \St .
\end{equation}
As a $\CC [\CO_\sigma]$-module, this is simply the regular representation on $\CC [\CO_\sigma]$.
Applying \eqref{eqn:Hom-Whit-Whit} both as written and to $\{M_{\nu},G_{\nu}, \chi^{\un}\}$, we obtain an isomorphism 
of modules for $\CH_{\CO_{\sigma}} \cong \CH_{\CO_1}^{K_{M,\phi}}$:
\begin{align}
\begin{split}
\Hom \big( i_{P_{\nu}}^{G_{\nu}} \chi^{\un}, \cW_{G_{\nu}}^{[M_{\nu},1]} \big)^{K_{M,\phi}} &\cong
\big( \ind_{\CH(\bW_{\CO_1},q)}^{\CH_{\CO_1}} \St \big)^{K_{M,\phi}} =
\ind_{\CH(\bW_{\CO_1},q)}^{\CH_{\CO_1}^{K_{M,\phi}}} \St \\ 
&\cong \ind_{\CH(\bW_{\CO_{\sigma}},q)}^{\CH_{\CO_{\sigma}}} \St \cong
\Hom \big( i_{P}^G \mathcal{W}_M^{[M,\sigma]}, \cW_G^{[M,\sigma]} \big) .
\end{split}
\end{align}
Next we compare the groups $R(\CO_{\sigma})$ and ${\widehat R}_{G,\phi}$.

\begin{lemma}\label{lem:identifying-RO-with-hatR}
There is a natural isomorphism identifying $R(\CO_{\sigma})$ with ${\widehat R}_{G,\phi}$.
\end{lemma}
\begin{proof}
By Weyl compatibility of $\sigma$ and $\phi$, we obtain an isomorphism of $\bW(M,\sigma)$ with $\bW(\hM,\phi)$.  
Since ${\widehat R}_{G,\phi}$ is the subgroup of $\bW(\hM,\phi)$ preserving $\LP_{\nu}$, it is identified under 
this isomorphism with the subgroup of $\bW(M,\sigma)$ preserving the $\bP_{\nu}$-positive roots of $G_{\nu}$. 
This is then identified--under the isomorphism of root data constructed by combining Propositions \ref{prop:levis}, 
\ref{prop:root equality}, Lemma \ref{lemma:R-trivial-unipotent} and Corollary \ref{cor:parameters equality}, 
between \eqref{eqn:root-datum-1comma-nu} and \eqref{eqn:root-datum-Osigma}--with the subgroup of $\bW(M,\sigma)$
preserving the $\bP$-positive roots of the root datum associated to 
$\End \big( i_{P}^G \mathcal{W}_M^{[M,\sigma]} \big)$, which is precisely $R(\CO_{\sigma})$.
\end{proof}

To complete the proof of Theorem \ref{thm:automorphic reduction}, it thus remains to show:

\begin{prop} \label{prop:action compatibility}
The isomorphism of $R(\CO_{\sigma})$ with ${\widehat R}_{G,\phi}$ constructed in Lemma 
\ref{lem:identifying-RO-with-hatR} identifies the action of $R(\CO_{\sigma})$ on 
$\End \big( i_{P_{\nu}}^{G_{\nu}} \chi^{\un} \big)^{K_{M,\phi}}$ with the action of ${\widehat R}_{G,\phi}$ 
on $\End(\CS^1_{G_{\nu}})^{K_{M,\phi}}$, via the isomorphism of $\End(\CS^1_{G_{\nu}})$ with 
$\End \big( i_{P_{\nu}}^{G_{\nu}} \chi^{\un} \big)$ induced by \eqref{eqn:LLC1M}. The latter action is
induced by the action of $\widehat{R}_{G,\phi}$ on $\overline{X}^\phi_{\LM}$.
\end{prop}
\begin{proof}
Recall from \S\ref{subsec:sheaf} that $\widehat{R}_{G,\phi} \cong \pi_{0,G,\phi} /
\pi_{0,M,\phi}$ and that the action of $\widehat{R}_{G,\phi}$ on 
$\End (\CS^1_{G_\nu})^{K_{M,\phi}}$ lifts to an action of $\pi_{0,G,\phi}$ on 
\[
\End (\CS^1_{G_\nu}) \cong \End (I_{P_\nu}^{G_\nu} \chi^\un ) \cong \CH (G_\nu,I_\nu) .
\]
In Proposition \ref{prop:4.actionXonSGnu} and Lemma \eqref{lem:4.actionXonHGI}, we saw 
that the 
action of $x \in \pi_{0,G,\phi}$ on these algebras is induced by the pinned automorphism
$\Ad_{x_P}$ of $G_\nu$ (and of $\hG_\nu$) followed by a twist by $\zeta (x) \in 
\Xnr (G_\nu) \cong Z(\hG_\nu)_\Fr^\circ$. Here $\Ad_{x_P}$ and $\zeta (x)$ are determined by
the action of $\pi_{0,G,\phi}$ on $\overline{X}^\phi_{\LM}$:
\begin{equation}\label{eq:6.xPzeta}
x \cdot m \phi = \Ad_{x_P}(m) \zeta (x) \phi = \zeta (x) \Ad_{x_P}(m) \phi 
\quad m \in \hM_\nu .   
\end{equation}
The actions in Proposition \ref{prop:4.actionXonSGnu} and Lemma \eqref{lem:4.actionXonHGI} 
also describe automorphisms of 
\[
\End (\CS^1_{G_\nu})^{\pi_{0,M,\phi}} \cong 
\End (I_{P_\nu}^{G_\nu} \chi^\un )^{K_{M,\phi}} \cong \CH (G_\nu,I_\nu)^{K_{M,\phi}} ,
\]
and then they can be regarded as actions of $\widehat{R}_{G,\phi}$. In Lemma 
\ref{lem:4.splittingR} we identified $\widehat{R}_{G,\phi}$ with the subgroup of
$\bW (\hM,\phi)$ that stabilizes $\LP_\nu$, thus in particular Weyl compatibility applies
to $\widehat{R}_{G,\phi}$. Viewing \eqref{eq:6.xPzeta} as an equality in 
$\overline{X}^1_{\LM_\nu} / \pi_{0,M,\phi}$, by Weyl compatibility, it is equivalent to 
\begin{equation}\label{eq:6.xchisigma}
x \cdot (\chi \otimes \sigma) = \zeta (x) \otimes \Ad_{x_P}(\chi) \otimes \sigma \in \CO_\sigma .
\end{equation}
Let $r$ be the image of $x$ in 
\[
R(\CO_\sigma) \cong \widehat{R}_{G,\phi} \cong \pi_{0,G,\phi} / \pi_{0,M,\phi} .
\]
Comparing \eqref{eq:6.xchisigma} with \eqref{eq:6.psirchir}, we see that
$\Ad_{x_P} = \psi_r$ and $\zeta (x) = \chi_r$. Hence the action of $x \pi_{0,M,\phi} \in 
\widehat{R}_{G,\phi}$ on $\End (I_{P_\nu}^{G_\nu} \chi^\un )^{K_{M,\phi}}$ can also be
described as induced by the pinned automorphism $\psi_r$ on $G_\nu$ followed by the
twist by $\chi_r$. This coincides with the action of $r$ studied in Theorems 
\ref{thm:solleveld} and \ref{thm:rep thy isom}.
\end{proof}

\section{Main Results} \label{sec:main}

\subsection{Construction of the functor}
\begin{thm} \label{thm:single block}
Let $\bM \subset \bL$ be standard Levi subgroups of $\bG$ and let $\sigma$ be an irreducible supercuspidal 
$(U_M,\psi_M)$-generic representation of $M$.  Let $\phi$ be a supercuspidal Langlands 
parameter for $M$ that is Langlands compatible with $\sigma$. 
Then we have a fully faithful functor:
$$\LocL_L^{\phi}: \Rep(L)_{[M,\sigma]} \rightarrow \IndCoh(X_{\LL}^{\phi})$$
such that $\LocL_L^{\phi}$ takes $\cW_L^{[M,\sigma]}$ to the structure sheaf of $X_{\LL}^{\phi}$. 
These functors are compatible with parabolic induction in the sense of Conjecture~\ref{conj:categorical}.
\end{thm}
\begin{proof}
The isomorphism:
\[
\End_L \big( i_{P_L}^L \cW_M^{[M,\sigma]} \big) \cong 
\End_{L_{\nu}} \big( i_{P_{L_{\nu}}}^{L_{\nu}} \chi^{un} \big)^{K_{L,\phi}} \rtimes \CC[{\widehat R}_{L,\phi}],
\]
combined with Theorem~\ref{thm:Springer reduction} and the isomorphism of 
$\End_{L_{\nu}} \big( i_{P_{L_{\nu}}}^{L_{\nu}} \chi^{\un} \big)$ with $\End(\CS^1_{L_{\nu}})$ yield an 
isomorphism of $\End(S_L^{\phi})$ with $\End_L \big( i_{P_L}^L \cW_M^{[M,\sigma]} \big)$.  
Let $\mathcal{E}_L$ denote either of these algebras.

We then have an equivalence of stable $\infty$-categories between $\Rep(L)_{[M,\sigma]}$ and the category of right $\mathcal{E}_L$-modules given by $V \mapsto \Hom_L \big( i_{P_L}^L \mathcal{W}_M^{[M,\sigma]}, V\big)$.

We also have an equivalence between right $\mathcal{E}_L$-modules and the full subcategory of 
$\IndCoh(X_{\LL}^{\phi})$ generated by $\CS_L^{\phi}$, given by $M \mapsto M \otimes_{\mathcal{E}_L} \CS_L^{\phi}$. 
Define the functor $\LocL_L^{\phi}$ to be the composition of these two functors.

The compatibility of the functors $\LocL^{\phi}_L$ with parabolic induction follows from the commutativity of the 
diagram of Theorem~\ref{thm:automorphic reduction}, thus it remains to check that $\LocL^{\phi}_L \big( \cW_L^{[M,\phi]} \big)$ 
is isomorphic to $\CO_{X_{\LL}^{\phi}}$.  Theorem~\ref{thm:Springer reduction} shows that $\CO_{X_{\LL}^{\phi}}$ 
is in the essential image of $\LocL^{\phi}_L$, and arises from the right $\mathcal{E}_L$-module 
$\Hom(\CS^{\phi}_L, \CO_{X_{\LL}^{\phi}})$, which by our hypotheses is necessarily isomorphic to 
$\Hom \big( i_{P_L}^L \cW_M^{[M,\sigma]}, \cW_L^{[M,\sigma]} \big)$. Thus the claim follows.
\end{proof}

\subsection{Generic categorical local Langlands}
When we have a suitable weak generic supercuspidal correspondence (see Definition \ref{def:wgsc}), we can deduce 
a categorical Langlands correspondence on a large direct factor of $\Rep(G)$ that we now describe.
Let $\Rep(G)^{\gen}$ denote the product of the full subcategories $\Rep(G)_{[M,\sigma]}$ for $\sigma$ generic 
with respect to $(U_M,\psi_M)$; it is a full subcategory of $\Rep(G)$.  We then have:

\begin{thm} \label{thm:main2}
Suppose that we have a weak generic supercuspidal correspondence $\Phi$ for $G$. 
Then Conjecture~\ref{conj:categorical} holds on the direct factor $\Rep(G)^{\gen}$ of $\Rep(G)$.  More precisely, for each standard Levi subgroup $\bM$ of $\bG$ (including $\bG$ itself), $\Phi$ gives rise to fully faithful embeddings:
$$\LocL_M^{\gen}: \Rep(M)^{\gen} \hookrightarrow \IndCoh(X_{\LM})$$
that are compatible with parabolic induction and Whittaker data in the sense of Conjecture~\ref{conj:categorical}.
\end{thm}
\begin{proof}
For each inertial equivalence class $[M,\sigma]$, where $\bM$ is a standard Levi of $\bG$ and $\sigma$ is an irreducible generic supercuspidal representation of $M$, the associated supercuspidal parameter $\Phi_{\bM}({\sigma})$ is Langlands compatible with $\sigma$. Thus by Theorem~\ref{thm:single block} there is a fully faithful functor 
\[
\LocL_G^{\Phi_M (\sigma)}: \Rep(G)_{[M,\sigma]} \rightarrow 
\IndCoh \big(X^{\Phi_{\bM}(\sigma)}_{\LG} \big).
\]
The functor $\LocL_G^{\gen}$ is the product of these functors over all generic inertial
equivalence classes of $G$, and the functors on Levi subgroups are constructed similarly. 
The required compatibilities are immediate from those proven in Theorem~\ref{thm:single block}.
\end{proof}

\begin{thm}\label{thm:unramified-classical-group-thm}
Let $\bG$ be one of the following quasi-split groups:
\begin{equation}\label{eqn:full-list-of-groups}
\GL_n, \mathrm{SL}_n, \PGL_n, \mathrm{U}_n, \Sp_{2n}, \SO_n, \SO_{2n}^*, 
\mathrm{GSpin}_n, \mathrm{GSpin}_{2n}^*, G_2.
\end{equation} 
(Here * means a quasi-split group defined by an order two automorphism of the Dynkin diagram.)\\
Then Conjecture \ref{conj:categorical} holds on the direct factor $\Rep(G)^{\gen}$ of $\Rep(G)$.   
\end{thm}
\begin{proof}
In view of Theorem \ref{thm:main2}, it suffices to check that we have a weak generic supercuspidal correspondence for $G$. This is done in Appendix \ref{appendix}. 
\end{proof}

\appendix
\addtocontents{toc}{\protect\setcounter{tocdepth}{1}}

\section{Plancherel compatibility and classical local Langlands correspondences}\label{appendix}

The goal of this appendix is to find alternative conditions equivalent to Plancherel compatibility, as in Definition
\ref{defn:Plancherel-compatibility}. We will work in greater generality, i.e.~for supercuspidal representations of
reductive $p$-adic groups which need not be quasi-split. This will help us prove that many known instances of 
a classical local Langlands correspondence give rise to a weak (generic) supercuspidal correspondence, as in Definition \ref{def:wgsc} and Remark \ref{rem:wsc}. This will complete the proof of Theorem \ref{thm:unramified-classical-group-thm}.

Let $\sigma$ be a supercuspidal $M$-representation.~Representations of the form $i_P^G (\sigma \otimes \chi)$ 
with $\chi \in \Xnr (M)$ are not necessarily tempered. Temperedness happens if and only if $\sigma \otimes \chi$ is
tempered. On the family of tempered representations of the form $i_P^G (\sigma \otimes \chi)$, 
by \cite[\S VIII]{Wal}, the Plancherel density $\mu_{Pl}$ has the form
\begin{equation}\label{eq:A.1}
\mu_{Pl} (i_P^G (\sigma \otimes \chi)) \textup{d} \omega = 
\mu_{G,\sigma}(\sigma \otimes \chi) \mathrm{fdeg}(\sigma) \textup{d} \omega,
\end{equation}
where d$\omega$ is a suitably normalized Haar measure on the variety of tempered representations in $\CO_\sigma$.
The formula \eqref{eq:A.1} shows that $\mu_{Pl} (i_P^G (\sigma \otimes \chi))$ is, like $\mu_{G,\sigma}$, a 
rational function of $\chi \in \Xnr (M)$. Hence we may regard $\mu_{Pl} (i_P^G (\sigma \otimes \chi))$ also
as a rational function on $\CO_\sigma$.

Let $L \subset G$ be a Levi subgroup minimally containing $M$. A result of Heiermann provides an alternative
characterization of $\mathrm{Pol}_{L,\sigma}$ (as in Definition 
\ref{defn:Plancherel-compatibility}). 
To formulate this, we recall that an irreducible $G$-representation $\pi$ 
is called essentially square-integrable if it admits a central character $Z(G) \to \CC^\times$ and 
$\pi |_{G^\der}$ is square-integrable. 

\begin{prop}\label{prop:Heiermann}
For $\sigma' \in \CO_\sigma$, the following are equivalent:
\begin{enumerate}[(i)]
\item $\mu_{L,\sigma} (\sigma') = \infty$ (i.e.~$\sigma' \in \mathrm{Pol}_{L,\sigma}$),
\item $\mu_{Pl} (i_{P_L}^L (\sigma')) = \infty$,
\item there exists an essentially square-integrable $L$-representation whose cuspidal support can be 
represented by $(M,\sigma')$.
\end{enumerate}
\end{prop}
\begin{proof}
(i) and (ii) are equivalent by equation \eqref{eq:A.1}. For the equivalence of (i) and (iii) see
\cite[Th\'eor\`eme 8.6 and Corollaire 8.7]{Hei04}.
\end{proof}

We want to describe $\widehat{\mathrm{Pol}}_{L,\phi}$ (as in Definition 
\ref{defn:Plancherel-compatibility}) similarly to Proposition \ref{prop:Heiermann}.(iii).
It is expected that via a local Langlands correspondence essentially square-integrable representations 
correspond to discrete Frobenius-semisimple 
L-parameters, so we will consider discrete L-parameters. Alternative reformulations of $\widehat{\mathrm{Pol}}_{L,\phi}$
involve local $L$-, $\epsilon$- and $\gamma$-factors. Recall that for a Langlands parameter $\phi = (\rho,N) \in X_\LG$, 
we have the adjoint L-function
\[
L(s, \Ad_{\hG} \phi )^{-1} = 1 - q^{-s} \det \big( \Ad \, \rho (\Fr), \Lie (\hG)^{I_F} \cap 
Z_{\Lie (\hG)}(N) \big)  \quad\text{ for }\; s \in \CC,
\]
the adjoint $\epsilon$-factor $\epsilon (s, \Ad_{\hG} \phi) \in \CC^\times$  
and the adjoint $\gamma$-factor
\begin{equation}\label{eq:A.gamma}
\gamma (s, \Ad_{\hG} \phi) = \epsilon (s, \Ad_{\hG} \phi) L(1-s, \Ad_{\hG} \phi) L(s,\Ad_{\hG} \phi)^{-1} 
\quad\text{ for }\; s \in \CC .
\end{equation}
Given $\phi = (\rho,N)$, all the possible monodromy operators for $\rho$ belong to
\[
\Lie (\hG)_\rho := \{ X \in \Lie (\hG) : \Ad \, \rho (w) X = \| w \| X \; \forall w \in W_F \}.
\]
This is a vector space on which $Z_{\hG}(\rho)$ acts (via the adjoint representation of $\hG$),
with a unique open orbit.

\begin{prop}\label{prop:A.pol}
Let $\phi = (\rho,N) \in \overline{X}_\LM$ be a discrete Frobenius-semisimple Langlands parameter 
and let $\varphi = (\tilde \rho, N) \in (Z(\hM)^{I_F})_\Fr^\circ \phi$.
Let $\LL \subset \LG$ be a Levi subgroup minimally containing $\LM$. The following are equivalent:
\begin{enumerate}[(i)]
\item $L(s=1, \Ad_{\hL} \varphi) = \infty$ (i.e.~$\varphi \in \widehat{\mathrm{Pol}}_{L,\phi}$),
\item $\gamma (s=0,\Ad_{\hL} \varphi) = \infty$,
\item $H^0 (W_F, \Ad_{\hL} \varphi (1))$ is nonzero, where ``(1)'' stands for a twist by the cyclotomic
character of $W_F / I_F$,
\item there exists a nonzero nilpotent $N_L \in Z_{\Lie (\hL)}(N)$ such that $(\tilde \rho,
N + N_L)$ is a discrete parameter in $X_{\LL}$,
\item $N$ does not belong to the open $Z_{\hL} (\tilde \rho)$-orbit in $\Lie (\hL)_{\tilde \rho}$,
\item there exists a discrete Frobenius-semisimple L-parameter $\varphi_L \in \overline{X}_\LL$ such that the
closure $\overline{\Ad (\hL) \varphi_L}$ of the orbit of $\varphi_L$ contains $\varphi$.
\end{enumerate}
\end{prop}
\begin{proof}
``(i) $\Longleftrightarrow$ (ii)'': The functions $\epsilon (s, \Ad_{\hL} \phi)$ and
$L(s, \Ad_{\hL} \phi)^{-1}$ are holomorphic, so $L(s=1, \Ad_{\hL} \varphi)$ and $\gamma (s=0, \Ad_{\hL} \varphi)$
have the same poles. 

``(i) $\Longleftrightarrow$ (iii)'': See the first paragraph of the proof of \cite[Proposition 6.10]{conjecture}.

``(iii) $\Longrightarrow$ (iv)'': As noted \textit{loc. cit.}, $H^0 (W_F, \Ad_{\hL} \varphi (1))$ is isomorphic to
the set of $N_L \in \Lie (\hL)_{\tilde \rho}$ with $[N_L,N] = 0$. For any such $N_L$, we know that 
$(\tilde \rho, N + N_L)$ is an L-parameter for $L$. By (iii), we can take $N_L \neq 0$.
Consider a Levi subgroup $\LH \subset \LL$ such that $(\tilde \rho, N + N_L)$ factors through
$\LH$ and $\LH$ is minimal for this property. Since $(\tilde \rho, N)$ is discrete in $\overline{X}_\LM$, we deduce that 
$\hH \cap \hM$ must be $\hM$. Hence $\LH \subset \LM$. The discreteness of $(\tilde \rho,N) \in \overline{X}_\LM$ 
also implies that $H^0 (W_F, \Ad_{\hM} \varphi (1))$ is zero. It follows that $N_L \notin \Lie (\hM)$ and
$\LH \neq \LM$. Therefore $\LH = \LL$, and equivalently $(\tilde \rho, N + N_L) \in \overline{X}_\LL$ 
is discrete. 

``(iv) $\Longrightarrow$ (iii)'': As in the above argument, any such $N_L$ gives rise to a nonzero element of 
$H^0 (W_F, \Ad_{\hL} \varphi (1))$.

``(i) $\Longleftrightarrow$ (v)'': This is the equivalence between (1) and (2) in \cite[Proposition 6.10]{conjecture}.

``(iv) $\Longrightarrow$ (vi)'': As we observed above, $N_L \notin \Lie (\hM)$. Hence the L-parameters
$(\tilde \rho, N + z N_L)$ with $z \in \CC^\times$ are conjugate by elements of $Z(\hM)^{W_F}$. In
particular $(\tilde \rho, N) \in \overline{\Ad (\hL) \varphi_L}$.

``(vi) $\Longrightarrow$ (v)'': We write $\varphi_L = (\rho_L,N_L)$. Since the image of $\rho_L$ is a finitely 
generated group of semisimple elements, its $\Ad (\hL)$-orbit is closed. Therefore $\tilde \rho \in 
\Ad (\hL) \rho_L$, and we may assume without loss of generality that $\rho_L = \tilde \rho$. The discrete L-parameter $\phi_L$ is
open by \cite[Proposition 7.2]{Sol26}, which means precisely that $N_L$ is in the open $Z_{\hL}(\rho_L)$-orbit in
$\Lie (\hL)_{\rho_L} = \Lie (\hL)_{\tilde \rho}$. Since $(\tilde \rho, N)$ is not discrete in $\overline{X}_{\LL}$, we deduce that 
$N$ does not belong to this open orbit.
\end{proof}

To compare $\mathrm{Pol}_{L,\sigma}$ and $\widehat{\mathrm{Pol}}_{L,\phi}$ effectively, we want to
match the cases from Proposition \ref{prop:Heiermann}.(iii) and Proposition \ref{prop:A.pol}.(vi), via
a local Langlands correspondence. The technical challenge here is that it is not obvious that $\varphi_L$ and
$\varphi$ give rise to representations of a reductive $p$-adic group and its Levi subgroup, as  
one of the issues here is that the cuspidal support map for representations of reductive $p$-adic groups cannot be 
expressed entirely in terms of Langlands parameters. To this end, we need to consider enhanced (Frobenius-semisimple)
Langlands parameters, where enhancements of L-parameters $\phi$ for $G$ are given by irreducible representations 
$\epsilon$ of $\pi_0 (Z_{\hG}(\phi))$.\footnote{Note that different component groups similar to $Z_{\hG}(\phi)$ are 
also possible, for our purposes this setup suffices.}
The notions of cuspidality and cuspidal support for such enhanced Langlands parameters were defined in
\cite[\S 6--7]{AMS1}. As is the case for representations of reductive $p$-adic groups, the cuspidal
support is only defined up to $\hG$-conjugacy.  It is expected that the local Langlands correspondence matches 
supercuspidal representations with cuspidal enhanced L-parameters and is compatible with the cuspidal support maps on 
both sides. Furthermore one expects that generic irreducible representations correspond to open Langlands
parameters with trivial enhancements. The next lemma shows that all these expectations together imply that
generic supercuspidal representations have supercuspidal L-parameters. 

\begin{lemma}\label{lem:5.10}
Let $(\phi,\epsilon) = (\rho,N,\epsilon)$ be a \textit{cuspidal} enhanced Langlands parameter, 
and assume that $\epsilon$ is the trivial representation of $\pi_0 (Z_{\hG}(\phi))$. 
Then $N = 0$ and $\rho$ is supercuspidal.
\end{lemma}
\begin{proof}
We rewrite $\phi$ as a Langlands parameter $\phi' : W_F \times \mathrm{SL}_2 (\CC) \to \LG$ 
with the same $\rho |_{I_F}$ and the same $N$, but with an adjusted image of $\Fr$. Set 
$\hH := Z_{\hG}(\phi' (W_F))$, then $\pi_0 (Z_{\hG}(\phi)) \cong \pi_0 (Z_{\hH}(N))$. 
The cuspidality requirement from \cite{AMS1} says that $\phi'$ is discrete and $(N,\epsilon)$
is cuspidal for $\hH$, which means that the associated equivariant local system on the 
$\hH$-orbit of $N$ in Lie$(\hH)$ is cuspidal in the sense of \cite{Lus-int}. Since 
$\epsilon$ is the trivial representation, this says that the trivial 
local system on the adjoint orbit of $N$ is cuspidal. 

The classification of cuspidal local systems in \cite{Lus-int} shows that this 
only happens when $N = 0$. Alternatively, the trivial local system on the $\hH$-orbit of $N$
always appears in the parabolic induction (with respect to a maximal torus $\hT$ of $\hH$) 
of the trivial local system on $\{0\} \in \Lie (\hT)$. The latter local system is cuspidal, 
so by uniqueness of cuspidal supports
the trivial local system on the orbit of $N$ is non-cuspidal when $N \neq 0$.

Since $\phi'$ is discrete, $\phi = (\rho,N=0)$ is also discrete. Hence the L-parameter 
$\rho$ is already discrete. 
\end{proof}

Consider a discrete Frobenius-semisimple Langlands parameter $\phi = (\rho,N) \in \overline{X}_\LM$. As before, we write
\[
\nu := \rho |_{I_F} ,\; \hG_\nu := Z_{\hG}(\nu)^\circ ,\;\text{and}\; \hM_\nu := Z_{\hM}(\nu)^\circ. 
\]
By Steinberg's theorem, the semisimple endomorphism $\Ad \, \rho (\Fr)$ stabilizes a Borel pair 
$(\hT_\nu, \hB_{M_\nu})$ of $\hM_\nu$. Since $\hP_\nu$ is $\Fr$-stable and $\rho$ takes values in $\LM$, we know that 
$\Ad \, \rho (\Fr)$ stabilizes $\hP_\nu$ and its unipotent radical $U_{\hP_\nu}$. Then $(\hT_\nu, 
\hB_\nu := \hB_{\bM_\nu} \bU_{\hP_\nu})$ is a $\rho (\Fr)$-stable Borel pair of $\hG_\nu$.~We extend this to a
pinning $(\hT_\nu, \hB_\nu, \{\mu_\alpha\})$. Since all pinnings of the connected complex reductive group
$\hG_\nu$ are conjugate and the normalizer of $(\hT_\nu,\hB_\nu)$ in $\hG_\nu$ is $\hT_\nu$, there exists an element  
$m_\phi \in \hT_\nu \subset \hM_\nu$ such that 
\[
\Ad (\rho (\Fr)) (\hT_\nu, \hB_\nu, \{\mu_\alpha\}) = \Ad (m_\phi) (\hT_\nu, \hB_\nu, \{\mu_\alpha\}) .
\]
Then $\Fr_\nu := m_\phi^{-1} \rho (\Fr) \in \LG$ normalizes $\hM_\nu$ and $\hG_\nu$, and fixes this pinning.
We note that this is a variant of Corollary \ref{cor:pinning} in our current setting.
By \cite[Theorem 3.4]{moduli}, we may assume that $\Fr_\nu^n = \Fr^n \in \LG$ for some $n \in \ZZ_{>0}$,
and then $\Fr_\nu$ is semisimple. Now we define an L-group 
\[
\LG_\nu := \hG_\nu \rtimes W_F ,
\]
where the action of $W_F$ on $\hG_\nu$ factors through $W_F / I_F$ and $\Fr^n$ acts as $\Ad \, \Fr_\nu^n$.
Let $G_\nu$ be the quasi-split $F$-group with L-group $\LG_\nu$. We have an L-homomorphism
\[
\iota_\nu : \LG_\nu \to \LG \text{ given by } \iota_\nu |_{\hG_\nu} = \mathrm{id} \;\text{and}\;
\iota_\nu (w) = \rho (w) \text{ for } w \in I_F \;\text{and}\; \iota_\nu (\Fr) = \Fr_\nu .
\]
We note that $\iota_\nu (\LM_\nu) \subset \LM$ and that the unipotent L-parameter 
$\phi_\nu := (m_\phi,N)$, which sends $\Fr$ to $m_\phi \Fr_\nu$, satisfies
$\iota_\nu (\phi_\nu) = \phi$.

As before, let $\LL \subset \LG$ be a Levi subgroup minimally containing $\LM$.

\begin{thm}\label{thm:A.3}
For $\varphi \in (Z(\hM)^{I_F})_\Fr^\circ \phi$, the following are equivalent:
\begin{enumerate}[(i)]
\item There exists a Frobenius-semisimple discrete $\varphi_L \in \overline{X}_\LL$ such that
$\overline{\Ad (\hL) \varphi_L}$ contains $\varphi$.
\item For every enhancement $\epsilon \in \Irr \big( \pi_0 (Z_{\hM}(\varphi)) \big)$ such that
$(\varphi, \epsilon)$ is cuspidal, there exists an enhanced discrete Frobenius-semisimple 
L-parameter $(\varphi_L, \epsilon_L)$ for $L$, such that the cuspidal support of
$(\varphi_L,\epsilon_L)$ is represented by $(\LM,\varphi,\epsilon)$.
\end{enumerate}
\end{thm}
\begin{proof}
``(ii) $\Longrightarrow$ (i)'': Assume (ii) and write $\varphi := (\tilde \rho, N)$. 
By \cite[\S 8]{AMS1}, upon choosing suitable representatives we can arrange that $\varphi_L = 
(\tilde \rho, N_L)$. Then the construction of the cuspidal support map \textit{loc.~cit.} 
implies that $N \in \overline{\Ad (Z_{\hL}(\tilde \rho)) N_L}$. Therefore $\varphi \in 
\overline{\Ad (\hL) \varphi_L}$.

``(i) $\Longrightarrow$ (ii)'':
We first prove this for $(\hM_\nu,\varphi_\nu)$ with $\iota_\nu (\varphi_\nu) = \varphi$, i.e.~in the special 
case of unipotent Langlands parameters. Besides Theorem \ref{thm:unipotent categorical}, there is also 
a classical local Langlands correspondence for unipotent representations \cite{Solleveld-unipotent}. 
It can be formulated as a bijection between:
\begin{itemize}
\item Frobenius-semisimple unipotent L-parameters $\psi \in X^1_{\LG_\nu}$ with enhancements 
$\varepsilon \in \Irr \big( \pi_0 (Z_{\hG}(\psi)) \big)$, up to $\hG_\nu$-conjugacy;
\item irreducible unipotent representations $\pi$ of a suitable collection of inner twists 
$\widetilde{G}_\nu$ of $G_\nu$. 
\end{itemize}
We denote this correspondence by 
\begin{equation}\label{eqn:classical-unipLLC}
(\psi,\varepsilon) \mapsto \pi (\psi,\varepsilon) \quad \text{and} \quad
\pi \mapsto (\phi_\pi, \epsilon_\pi).
\end{equation}
It matches discrete L-parameters with essentially square-integrable representations and
cuspidal enhanced L-parameters with supercuspidal representations. It is shown in \cite{FOS} that this
unipotent LLC satisfies the conjectures about Plancherel densities and formal degrees from \cite{HII}; 
these results imply, among others, that for every essentially square-integrable $M_\nu$-representation $\sigma$, 
there exists a $c_\sigma \in \RR_{>0}$ such that
\begin{equation}\label{eq:A.2}
\mu_{Pl}(i_{P_{L_\nu}}^{L_\nu} (\sigma \otimes \chi)) = c_\sigma 
\gamma (0, \Ad_{\hL_\nu} \phi_{\sigma \otimes \chi}) \text{ as rational functions of } \chi \in \Xnr (M_\nu).
\end{equation}
By Propositions \ref{prop:Heiermann} and \ref{prop:A.pol}, it means that this (classical) unipotent LLC
restricts to a bijection
\begin{equation}\label{eq:A.3}
\CO_\sigma \supset \mathrm{Pol}_{L,\sigma} \xrightarrow{\sim} 
\widehat{\mathrm{Pol}}_{L,\phi_\sigma} \subset (Z(\hM_\nu)^{I_F})_\Fr^\circ \phi_\sigma .
\end{equation}
Suppose (i) holds and let $\varepsilon_\nu$ be an 
enhancement of $\varphi_\nu$ such that $(\varphi_\nu, \varepsilon_\nu)$ is cuspidal. Then $\sigma_\nu :=
\pi (\varphi_\nu,\varepsilon_\nu)$, via \eqref{eqn:classical-unipLLC}, is a supercuspidal $\widetilde{M}_\nu$-representation for some inner twist 
$\widetilde{M}_\nu$ of $M_\nu$.~By Proposition \ref{prop:A.pol}, $\gamma (0,\Ad_{\hM_\nu} \varphi_\nu) = \infty$ 
and thus by \eqref{eq:A.3} we have $\sigma_\nu \in \mathrm{Pol}_{L_\nu,\sigma_\nu}$. Then by Proposition \ref{prop:Heiermann}, we know that $I_{P_{L_\nu}}^{L_\nu} \sigma_\nu$ has an essentially square-integrable subquotient, which we can denote
as $\pi (\varphi_{L_\nu}, \varepsilon_{L_\nu})$. By the compatibility of the (classical) unipotent LLC \eqref{eqn:classical-unipLLC} with cuspidal supports,
Sc$(\varphi_{L_\nu}, \varepsilon_{L_\nu})$ is represented by $(\varphi_\nu, \varepsilon_\nu)$.

Now we prove (i) $\Longrightarrow$ (ii) in the general case. The difference between $X^\phi_{\LL}$ and $X^1_{\LL_\nu}$ comes from the group $\pi_0 (Z_{\hL} (\nu))$. More concretely, the relevant component groups for $\varphi = (\tilde \rho,N) \in (Z(\hM)^{I_F})_\Fr^\circ \phi$ are
\begin{align*}
\pi_0 (Z_{\hL} (\varphi)) = \pi_0 \big( Z_{\hL} (\nu) \cap Z_{\hL} (\tilde \rho (\Fr),N) \big) \;\text{ and }\;
\pi_0 (Z_{\hL_\nu} (\varphi_\nu) ) = \pi_0 \big( Z_{\hL} (\nu)^\circ \cap 
Z_{\hL} (\tilde \rho (\Fr),N) \big) ,
\end{align*}
where $\iota_\nu (\varphi_\nu) = \varphi$.
Let $\varepsilon \in \Irr \big( \pi_0 (Z_{\hM} (\varphi)) \big)$ be an enhancement such that 
$(\varphi,\varepsilon)$ is cuspidal. Let $\varepsilon_\nu$ be an irreducible constituent of
$\varepsilon |_{\pi_0 (Z_{\hM_\nu} (\varphi_\nu) )}$. 
Let $(\varphi_{L_\nu}, 
\varepsilon_{L_\nu})$ be as in the above special case of unipotent parameters, such that $\varphi_{L_\nu}$ and 
$\varphi_\nu$ differ only in their monondromy operators. 

The cuspidal support maps are compatible with restriction to the neutral component of a complex
reductive group, see \cite[\S 5]{AMS1} and \cite[Corollary 2.4.4]{DillerySchwein}. This means that the
following sets coincide:
\begin{itemize}
\item irreducible constituents of 
$\ind_{\pi_0 (Z_{\hM_\nu} (\varphi_\nu))}^{\pi_0 (Z_{\hM} (\varphi))} \varepsilon_\nu$;
\item  $\varepsilon' \in \Irr \big( \pi_0 (Z_{\hM} (\varphi)) \big)$ such that 
$(\LM ,\varphi, \varepsilon')$ arises as the cuspidal support of $(\varphi_L, \varepsilon'_L)$
for some irreducible constituent $\varepsilon'_L$ of  
$\ind_{\pi_0 (Z_{\hL_\nu} (\varphi_{L,\nu}))}^{\pi_0 (Z_{\hL} (\varphi_L))} \varepsilon_{L_\nu}$, where $\iota_\nu (\varphi_{L,\nu}) = \varphi_L$.
\end{itemize}
In particular, we can find an enhancement $\varepsilon_L$ of $\varphi_L$ such that 
Sc$(\varphi_L, \varepsilon_L)$ is represented by 
$(\LM, \varphi, \varepsilon)$.
\end{proof}
We are now ready to show that any local Langlands correspondence with a few standard properties satisfies Plancherel compatibility and gives rise to a weak (generic)
supercuspidal correspondence. Let
\begin{equation}\label{eqn:appendix-classical-LLC}
\pi \mapsto (\phi_\pi, \epsilon_\pi)
\end{equation}
denote a classical local Langlands correspondence 
(we assume that it exists for $G$ and its Levi subgroups). We first establish the following sufficient criterion 
(Theorem \ref{thm:A.4}) for a classical local Langlands correspondence to give a weak (generic)
supercuspidal correspondence; then we will verify the criterion for a large class of reductive $p$-adic groups.  

\begin{thm}\label{thm:A.4}\ 
\begin{enumerate}
\item Suppose the correspondence \eqref{eqn:appendix-classical-LLC} has the following properties:
\begin{enumerate}[(i)]
\item $\pi$ is essentially square-integrable if and only if $\phi_\pi$ is discrete.
\item The correspondence \eqref{eqn:appendix-classical-LLC} is compatible with the cuspidal support maps for essentially square-integrable
representations on the group side and for enhanced discrete Frobenius-semisimple L-parameters on the Galois side.
\item For every supercuspidal $M$-representation $\sigma$, \eqref{eqn:appendix-classical-LLC} gives a $\Xnr (M)$-equivariant bijection 
\[
\CO_\sigma \xrightarrow{\sim} (Z(\hM)^{I_F} )_\Fr^\circ (\phi_\sigma ,\epsilon_\sigma) .
\]
\end{enumerate}
Then $\sigma$ and $\phi_\sigma$ are Plancherel compatible, for every supercuspidal $M$-representation $\sigma$.    
\item Suppose that in addition to (i)--(iii) the correspondence is ``enhanced'' Weyl compatible, in the following sense:\footnote{Weyl compatibility in the body of the paper did not
explicitly involve enhancements, but therein implicitly the enhancements were trivial representations
of component groups. Therefore the condition we impose on $\epsilon_{w \sigma}$ here is 
automatically fulfilled in the body of the paper, and is a generalized version of the Weyl compatibility defined in Definition \ref{defn:Weyl-compatible}.}\\
(iv) For every supercuspidal $M$-representation $\sigma$ and every $w \in \bW (G,M)$ corresponding to $\hat w \in 
\bW (\hG,\hM)$, we have $\phi_{w \cdot \sigma} = \Ad (\hat w) \phi_\sigma$ and 
$\epsilon_{w \cdot \sigma} = \epsilon_\sigma \circ \Ad (\hat w)^{-1}$. 

\noindent Then \eqref{eqn:appendix-classical-LLC} defines a weak supercuspidal correspondence for $G$ and its 
Levi subgroups (Remark \ref{rem:wsc}).
\item Suppose in addition to (i)--(iv) the following: \\
(v) the enhancement $\epsilon_\pi$ is trivial for every generic supercuspidal $M$-representation $\pi$.

\noindent Then \eqref{eqn:appendix-classical-LLC} defines a weak generic supercuspidal correspondence for $G$ and its 
Levi subgroups.
\end{enumerate}
\end{thm}
\begin{proof}
(1) The twist compatibility of $\sigma \mapsto \phi_\sigma$ follows from (iii).

Let $\LL \subset \LG$ be a Levi subgroup minimally containing $\LM$, and let $\sigma \otimes \chi \in 
\mathrm{Pol}_{L,\sigma}$. By Proposition \ref{prop:Heiermann}, $i_{P_L}^L (\sigma \otimes \chi)$ has an 
essentially square-integrable subquotient $\pi$. It corresponds to an enhanced discrete Frobenius-semisimple
L-parameter $(\phi_\pi, \epsilon_\pi)$ for $L$. Then Sc$(\pi)$ is represented by $(M,\sigma \otimes \chi)$,
thus by (ii) Sc$(\phi_\pi,\epsilon_\pi)$ is represented by $(\LM ,\phi_{\sigma \otimes \chi}, 
\epsilon_{\sigma \otimes \chi})$. As in the proof of (ii) $\Longrightarrow$ (i) in Theorem \ref{thm:A.3}, 
this implies that $\phi_{\sigma \otimes \chi} \in \overline{\Ad (\hL) \phi_\pi}$. By Proposition \ref{prop:A.pol},
$\phi_{\sigma \otimes \chi} \in \widehat{\mathrm{Pol}}_{L,\phi_\sigma}$.

Conversely, consider $\chi' \in \Xnr (M)$ such that $\phi_{\sigma \otimes \chi'} \in 
\widehat{\mathrm{Pol}}_{L,\phi_\sigma}$. By Proposition \ref{prop:A.pol} there exists a discrete
Frobenius-semisimple L-parameter $\phi_L \in \overline{X}_{\LL}$ such that $\phi_{\sigma \otimes \chi'} \in 
\overline{\Ad (\hL) \phi_L}$. By (ii) and the supercuspidality of $\sigma \otimes \chi'$, we know that 
$(\phi_{\sigma \otimes \chi'}, \epsilon_{\sigma \otimes \chi'})$ is cuspidal. By Theorem \ref{thm:A.3}, there exists an enhancement $\epsilon_L$ such that Sc$(\phi_L,\epsilon_L)$ is represented by
$(\LM, \phi_{\sigma \otimes \chi'}, \epsilon_{\sigma \otimes \chi'})$. By (i), the $L$-representation $\pi_L$
corresponding to $(\phi_L, \epsilon_L)$ is essentially square-integrable, and by (ii) it has cuspidal
support $(M,\sigma \otimes \chi')$. Then $\pi_L$ is isomorphic to a subquotient of 
$i_{P_L}^L (\sigma \otimes \chi')$, and by Proposition \ref{prop:Heiermann} we know that 
$\sigma \otimes \chi' \in \mathrm{Pol}_{L,\sigma}$.

(2) By part (1), all the conditions in Definition \ref{def:wgsc} are met.

(3) By (i) and (ii), every supercuspidal representation $\pi$ has a cuspidal enhanced L-parameter. If $\pi$ is
moreover generic, then by (v) and Lemma \ref{lem:5.10} the L-parameter $\phi_\pi$ is supercuspidal. Hence the
weak supercuspidal correspondence from part (2) specializes to a weak generic supercuspidal correspondence.
\end{proof}

We now proceed to check that known instances of classical local Langlands correspondences
satisfy the conditions from Theorem \ref{thm:A.4}, and thus qualify as weak supercuspidal
correspondences. 

\subsection{Inner twists of general linear groups}

The local Langlands correspondence in this case is well-known, an account of it can be found in \cite[\S 2]{ABPS16}.
To account for nontrivial inner twists, we compute the component groups of Langlands parameters in 
$\SL_n (\CC)$. The properties (i), (iii) and (iv) hold by construction, while (ii) was verified in 
\cite[discussion after Conjecture 7.8]{AMS1}. 
It is also well-known that every supercuspidal representation of a Levi subgroup $M$ of 
$\GL_n (F)$ is generic \cite{GelfandKazhdan} and that it has a supercuspidal L-parameter. 
(Alternatively, the enhancement for any irreducible $M$-representation is trivial.)

\subsection{Inner twists of projective linear groups}

The representations of an inner twist of $\PGL_n (F)$ form a subcategory of the representations of the
associated inner twist of $\GL_n (F)$, and similarly for enhanced L-parameters. All the desired properties
of the LLC follow from those for inner twists of $\GL_n (F)$.

\subsection{Special linear groups}

The LLC for $G^\sharp = \SL_n (F)$ is constructed from that for $G = \GL_n (F)$, as in 
\cite{ABPS16, GelbartKnapp, HiragaSaito}. Let $\pi \in \Irr (\GL_n (F))$ and let $\pi^\sharp$ be an irreducible
constituent of $\Res^G_{G^\sharp} \pi$. Then the L-parameter $\phi_{\pi^\sharp} \in X_{\LG^\sharp}$ is 
obtained from $\phi_\pi \in X_{\LG}$ via the quotient map $\GL_n (\CC) \twoheadrightarrow \PGL_n (\CC)$.

Let $X^G (\pi)$ be the stabilizer in $\Hom (G / G^\sharp, \CC^\times)$ of $\pi \in \Irr (G)$. The group
$X^G (\pi)$ acts on $\pi$ by $G^\sharp$-intertwining operators, which are normalized using the Whittaker 
datum. There is a unique character $\epsilon_{\pi^\sharp} : X^G (\pi) \to \CC^\times$ such that
\[
\pi^\sharp = \Hom_{X^G (\pi)} (\epsilon_{\pi^\sharp}, \pi) .
\]
Furthermore $X^G (\pi)$ is naturally isomorphic to the component group $A_{\pi^\sharp} := \pi_0 \big( Z_{\PGL_n (\C)} (\phi_\pi) \big)$,
thus we can regard $\epsilon_{\pi^\sharp}$ as a character of $A_{\pi^\sharp}$. Therefore, 
\begin{equation}\label{eq:A.5}
(\phi_{\pi^\sharp}, \epsilon_{\pi^\sharp}) \text{ is the enhanced L-parameter of } \pi^\sharp .
\end{equation}
The LLC for Levi subgroups of $\SL_n (F)$ can be constructed in an analogous way, via restriction from
a Levi subgroup $L \subset G$ to $L \cap G^\sharp$.

The properties (i) and (iii) in  Theorem \ref{thm:A.4} hold for this correspondence, because
they hold for $\GL_n (F)$. For the Weyl compatibility, consider a Levi subgroup $M \subset G$, a supercuspidal
$\sigma \in \Irr (M)$ and an irreducible constituent $\sigma^\sharp$ of $\Res^M_{M \cap G^\sharp} \sigma$. Let 
$w \in \bW (G,M) \cong \bW (G^\sharp, M \cap G^\sharp)$. The group
\[
X^M (\sigma) := \mathrm{Stab}_{\Hom (M / M \cap G^\sharp, \CC^\times)}(\sigma)
\]
is isomorphic to $X^{w M w^{-1}}(w \cdot \sigma)$, and moreover we have
\[
w \cdot \sigma^\sharp \cong \Hom_{X^{w M w^{-1}}(w \cdot \sigma)} (w \cdot \epsilon_{\sigma^\sharp}, w \cdot \sigma) .
\]
By the Weyl compatibility for $\GL_n (F)$, we have 
\[
\phi_{w \cdot \sigma^\sharp} = \phi_{(w \cdot \sigma)^\sharp} = \Ad (w) \phi_{\sigma^\sharp} 
\quad \text{and} \quad \epsilon_{w \cdot \sigma^\sharp} = w \cdot \epsilon_{\sigma^\sharp} = 
\epsilon_{\sigma^\sharp} \circ \Ad (w)^{-1} .
\]

\begin{lemma}
The LLC for Levi subgroups of $\SL_n (F)$ from \eqref{eq:A.5} is compatible with cuspidal support maps.
\end{lemma}
\begin{proof}
Let $\pi$ and $\pi^\sharp$ be as above, and assume that Sc$(\pi) = (M,\sigma)$. Then the cuspidal support of
$\pi^\sharp$ is $(M \cap G^\sharp, \sigma^\sharp)$ for some constituents $\sigma^\sharp$ of
$\Res^M_{M \cap G^\sharp} \sigma$.

The cuspidal support map for L-parameters of $\GL_n (F)$-representations is easy to describe: the L-parameter 
$\phi_\sigma$ is obtained from $\phi_\pi$ by replacing the monodromy operator $N_\pi$ by 0.
The group $\hM$ centralizes $N_\pi$ and the inclusion $Z_{\hM} (\phi_\sigma) \hookrightarrow Z_{\hG}(\phi_\pi)$
induces an injection
\begin{equation}\label{eq:A.6}
A_{\sigma^\sharp} := \pi_0 \big( Z_{\hM / Z(\GL_n (\CC))} (\phi_{\sigma^\sharp}) \big) \hookrightarrow
A_{\pi^\sharp}  := \pi_0 \big( Z_{\PGL_n (\C)} (\phi_\pi) \big) .
\end{equation}
In this case, the cuspidal support map preserves the characters of $A_{\sigma^\sharp}$, thus 
Sc$(\phi_{\pi^\sharp}, \epsilon_{\pi^\sharp})$ can be represented by $\big( \phi_{\sigma^\sharp},
\epsilon_{\pi^\sharp} |_{A_{\sigma^\sharp}} \big)$.

On the other hand, the injection \eqref{eq:A.6} corresponds to an injection $X^M (\sigma) \hookrightarrow X^G (\pi)$.
For $x \in X^G (\pi)$, the $G^\sharp$-intertwining operator on $\pi$ is obtained from the 
$M \cap G^\sharp$-intertwining operator on $\sigma$ via parabolic induction. It follows that
\[
\mathrm{Sc}(\pi) \cong \big( M, \Hom_{X^M (\sigma)} (\epsilon_{\pi^\sharp} |_{X^M (\sigma)}, \sigma) \big),
\]
which has enhanced L-parameter $\big( \phi_{\sigma^\sharp}, \epsilon_{\pi^\sharp} |_{X^M (\sigma)} \big)$.
Thus the LLC for $\SL_n (F)$ preserves cuspidal supports. 

The same argument works when $\GL_n (F)$ is replaced by a Levi subgroup $L \supset M$ and 
$\SL_n (F)$ is replaced by $L \cap \SL_n (F)$.
\end{proof}

Every supercuspidal representation $\sigma^\sharp$ of $M \cap G^\sharp$ appears in the
restriction of a supercuspidal $M$-representation $\sigma$. Then $\sigma$ is generic and
$\phi_\sigma$ is supercuspidal, which implies that $\phi_{\sigma^\sharp}$ is supercuspidal. 
The normalization of the intertwining operators from $X^M (\sigma)$ via the 
Whittaker datum implies that
\[
\sigma^\sharp = \Hom_{X^M (\sigma)} (\epsilon_{\sigma^\sharp}, \sigma) 
\]
is generic if and only if $\epsilon_{\sigma^\sharp}$ is trivial.

\subsection{Quasi-split classical groups and their pure inner forms}

Here we consider the following as classical groups: unitary groups, symplectic groups,
special orthogonal groups and general spin groups. By a quasi-split classical group
we mean a classical group which is either split or defined by an order two automorphism
of the Dynkin diagram (and a quadratic extension of local fields).

A local Langlands correspondence for pure inner twists of quasi-split classical groups
is due to M\oe glin, see \cite{MoRe,MoTa}. The conditions (i)--(iv) in  Theorem \ref{thm:A.4} 
were proven in \cite{AMS4}. 
M\oe glin uses a Whittaker datum to pin down the correspondence, by requiring that for
any generic irreducible representation the corresponding L-parameter $\phi$ is enhanced with
the trivial representation of the component group $\pi_0 (Z_{\hG}(\phi))$. 
By \cite{AMS4} and Lemma \ref{lem:5.10}, generic supercuspidal representations have
supercuspidal L-parameters in M\oe glin's LLC and (v) holds.

\subsection{\texorpdfstring{$\mathbf{G_2}$}{G2}}

A local Langlands correspondence for $G_2 (F)$ was constructed in \cite{AubertXuG2}. There
(i), (ii) and (v) from Theorem \ref{thm:A.4} are verified. Conditions (iii) and (iv) from 
Theorem \ref{thm:A.4} are proven in \cite[\S 4]{AubertXuHecke}.

\subsection{Arbitrary tamely ramified reductive $p$-adic groups: nonsingular blocks}

In \cite{Kaletha}, a local Langlands correspondence for non-singular supercuspidal representations
was constructed. It seems likely that such representations account for almost all generic
supercuspidal representations of tamely ramified reductive $p$-adic groups, and for almost all
supercuspidal L-parameters. A local 
Langlands correspondence for all Bernstein blocks whose cuspidal supports are nonsingular
of depth zero was constructed in \cite{SolleveldXu2026}, and it satisfies all properties needed
in Theorem \ref{thm:A.4}, and thus defines a weak (generic) supercuspidal correspondence for
those Bernstein blocks. A generalization of \cite{SolleveldXu2026} to arbitrary depths
is work in progress.

\bibliographystyle{amsalpha}
\bibliography{bibfile}

\end{document}